\documentclass[pdflatex,sn-mathphys-num]{sn-jnl}

\usepackage{geometry}
\usepackage{graphicx}%
\usepackage{multirow}%
\usepackage{amsmath,amssymb,amsfonts}%
\usepackage{amsthm}%
\usepackage{mathrsfs}%
\usepackage[title]{appendix}%
\usepackage{xcolor}%
\usepackage{textcomp}%
\usepackage{manyfoot}%
\usepackage{booktabs}%
\usepackage{algorithm}%
\usepackage{algorithmicx}%
\usepackage{algpseudocode}%
\usepackage{listings}%

\theoremstyle{thmstyleone}%
\newtheorem{theorem}[subsubsection]{Theorem}
\newtheorem{proposition}[subsubsection]{Proposition}%
\newtheorem{lemma}[subsubsection]{Lemma}%
\newtheorem{corollary}[subsubsection]{Corollary}%

\theoremstyle{thmstyletwo}%
\newtheorem{example}[subsubsection]{Example}%
\newtheorem{remark}[subsubsection]{Remark}%
\newtheorem{remarks}[subsubsection]{Remarks}%
\numberwithin{equation}{subsection}

\theoremstyle{thmstylethree}%
\newtheorem{definition}[subsubsection]{Definition}%

\usepackage{tikz-cd}
\usetikzlibrary{matrix,decorations.pathreplacing, calc, positioning,fit}	
\usepackage[all]{xy}
\SelectTips{cm}{}
\usepackage{tikz}
\usepackage{verbatim}

\usepackage{mathdots}
\RequirePackage{xspace}
\RequirePackage{etoolbox}
\RequirePackage{varwidth}
\RequirePackage{enumitem}
\RequirePackage{tensor}
\RequirePackage{mathtools}
\RequirePackage{longtable}
\RequirePackage{multirow}
\RequirePackage{makecell}
\RequirePackage{bigints}

\newcommand{\sF}{\ensuremath{\mathscr{F}}\xspace}

\newcommand{\BA}{\ensuremath{\mathbb{A}}\xspace}

\newcommand{\BC}{\ensuremath{\mathbb{C}}\xspace}
\newcommand{\BD}{\ensuremath{\mathbb{D}}\xspace}

\newcommand{\BF}{\ensuremath{\mathbb{F}}\xspace}
\newcommand{\BG}{\ensuremath{\mathbb{G}}\xspace}

\newcommand{\BK}{\ensuremath{\mathbb{K}}\xspace}

\newcommand{\BM}{\ensuremath{\mathbb{M}}\xspace}
\newcommand{\BN}{\ensuremath{\mathbb{N}}\xspace}

\newcommand{\BQ}{\ensuremath{\mathbb{Q}}\xspace}
\newcommand{\BR}{\ensuremath{\mathbb{R}}\xspace}

\newcommand{\BV}{\ensuremath{\mathbb{V}}\xspace}

\newcommand{\BZ}{\ensuremath{\mathbb{Z}}\xspace}

\newcommand{\bA}{\ensuremath{\mathbf{A}}\xspace}
\newcommand{\bB}{\ensuremath{\mathbf{B}}\xspace}

\newcommand{\bi}{\ensuremath{\mathbf{i}}\xspace}
\newcommand{\bj}{\ensuremath{\mathbf{j}}\xspace}

\newcommand{\bw}{\ensuremath{\mathbf{w}}\xspace}

\newcommand{\CA}{\ensuremath{\mathcal{A}}\xspace}
\newcommand{\CB}{\ensuremath{\mathcal{B}}\xspace}

\newcommand{\CD}{\ensuremath{\mathcal{D}}\xspace}

\newcommand{\CF}{\ensuremath{\mathcal{F}}\xspace}
\newcommand{\CG}{\ensuremath{\mathcal{G}}\xspace}

\newcommand{\CI}{\ensuremath{\mathcal{I}}\xspace}

\newcommand{\CK}{\ensuremath{\mathcal{K}}\xspace}
\newcommand{\CL}{\ensuremath{\mathcal{L}}\xspace}
\newcommand{\CM}{\ensuremath{\mathcal{M}}\xspace}

\newcommand{\CO}{\ensuremath{\mathcal{O}}\xspace}
\newcommand{\CP}{\ensuremath{\mathcal{P}}\xspace}

\newcommand{\CS}{\ensuremath{\mathcal{S}}\xspace}

\newcommand{\CX}{\ensuremath{\mathcal{X}}\xspace}
\newcommand{\CY}{\ensuremath{\mathcal{Y}}\xspace}

\newcommand{\ad}{\mathrm{ad}}
\DeclareMathOperator{\Aut}{Aut}

\DeclareMathOperator{\End}{End}
\newcommand{\Fil}{\ensuremath{\mathrm{Fil}}\xspace}
\DeclareMathOperator{\Gal}{Gal}
\newcommand{\GL}{\mathrm{GL}}
\DeclareMathOperator{\Gr}{Gr}
\newcommand{\GU}{\mathrm{GU}}

\newcommand{\id}{\ensuremath{\mathrm{id}}\xspace}

\DeclareMathOperator{\im}{im}

\newcommand{\loc}{\ensuremath{\mathrm{loc}}\xspace}
\newcommand{\naive}{\ensuremath{\mathrm{naive}}\xspace}
\newcommand{\OGr}{\mathrm{OGr}}
\DeclareMathOperator{\ord}{ord}

\newcommand{\red}{\ensuremath{\mathrm{red}}\xspace}
\DeclareMathOperator{\Res}{Res}
\newcommand{\Sh}{\mathrm{Sh}}

\DeclareMathOperator{\spann}{span}
\DeclareMathOperator{\Spec}{Spec}
\newcommand{\spin}{\ensuremath{\mathrm{spin}}\xspace}
\newcommand{\SO}{{\mathrm{SO}}}

\DeclareMathOperator{\sgn}{sgn}
\DeclareMathOperator{\tr}{tr}
\newcommand{\U}{\mathrm{U}}
\newcommand{\wt}{\widetilde}
\newcommand{\wh}{\widehat}
\newcommand{\ov}{\overline}
\newcommand{\lra}{\longrightarrow}
\newcommand{\ini}{\text{in}}
\newcommand{\lcm}{\text{lcm}}	

\DeclareMathOperator{\Span}{Span}
\newcommand{\WT}{\text{WT}}

\DeclareMathOperator{\Adm}{\text{Adm}}
\newcommand{\Fl}{\mathcal{F}\ell}
\newcommand{\aform}{\ensuremath{\langle\text{~,~}\rangle}\xspace}
\newcommand{\sform}{\ensuremath{(\text{~,~})}\xspace}

\newenvironment{altenumerate}
	{\begin{list}
			{(\theenumi) }
			{\usecounter{enumi}
				\setlength{\labelwidth}{0pt}
				\setlength{\labelsep}{0pt}
				\setlength{\leftmargin}{0pt}
				\setlength{\itemsep}{\the\smallskipamount}
				\renewcommand{\theenumi}{\roman{enumi}}
		}}
		{\end{list}}
	\newenvironment{altenumerate2}
	{\begin{list}
			{\textup{(\theenumii)} }
			{\usecounter{enumii}
				\setlength{\labelwidth}{0pt}
				\setlength{\labelsep}{0pt}
				\setlength{\leftmargin}{2em}
				\setlength{\itemsep}{\the\smallskipamount}
				\renewcommand{\theenumii}{\alph{enumii}}
		}}
		{\end{list}}
	
	\newenvironment{altitemize}
	{\begin{list}
			{$\bullet$}
			{\setlength{\labelwidth}{0pt}
				\setlength{\itemindent}{5pt}
				\setlength{\labelsep}{5pt}
				\setlength{\leftmargin}{0pt}
				\setlength{\itemsep}{\the\smallskipamount}
		}}
		{\end{list}}

\setitemize[0]{leftmargin=*,itemsep=\the\smallskipamount}
\setenumerate[0]{leftmargin=*,itemsep=\the\smallskipamount}
	
\renewcommand{\to}{%
		\ifbool{@display}{\longrightarrow}{\rightarrow}%
	}
	\let\shortmapsto\mapsto
	\renewcommand{\mapsto}{%
		\ifbool{@display}{\longmapsto}{\shortmapsto}%
	}
	\newlength{\olen}
	\newlength{\ulen}
	\newlength{\xlen}
	\newcommand{\xra}[2][]{%
		\ifbool{@display}%
		{\settowidth{\olen}{$\overset{#2}{\longrightarrow}$}%
			\settowidth{\ulen}{$\underset{#1}{\longrightarrow}$}%
			\settowidth{\xlen}{$\xrightarrow[#1]{#2}$}%
			\ifdimgreater{\olen}{\xlen}%
			{\underset{#1}{\overset{#2}{\longrightarrow}}}%
			{\ifdimgreater{\ulen}{\xlen}%
				{\underset{#1}{\overset{#2}{\longrightarrow}}}
				{\xrightarrow[#1]{#2}}}}%
		{\xrightarrow[#1]{#2}}
	}
	\makeatother
	\newcommand{\xyra}[2][]{%
		\settowidth{\xlen}{$\xrightarrow[#1]{#2}$}%
		\ifbool{@display}%
		{\settowidth{\olen}{$\overset{#2}{\longrightarrow}$}%
			\settowidth{\ulen}{$\underset{#1}{\longrightarrow}$}%
			\ifdimgreater{\olen}{\xlen}%
			{\mathrel{\xymatrix@M=.12ex@C=3.2ex{\ar[r]^-{#2}_-{#1} &}}}%
			{\ifdimgreater{\ulen}{\xlen}%
				{\mathrel{\xymatrix@M=.12ex@C=3.2ex{\ar[r]^-{#2}_-{#1} &}}}
				{\mathrel{\xymatrix@M=.12ex@C=\the\xlen{\ar[r]^-{#2}_-{#1} &}}}}}%
		{\mathrel{\xymatrix@M=.12ex@C=\the\xlen{\ar[r]^-{#2}_-{#1} &}}}%
	}
	\makeatletter
	\newcommand{\xla}[2][]{%
		\ifbool{@display}%
		{\settowidth{\olen}{$\overset{#2}{\longleftarrow}$}%
			\settowidth{\ulen}{$\underset{#1}{\longleftarrow}$}%
			\settowidth{\xlen}{$\xleftarrow[#1]{#2}$}%
			\ifdimgreater{\olen}{\xlen}%
			{\underset{#1}{\overset{#2}{\longleftarrow}}}%
			{\ifdimgreater{\ulen}{\xlen}%
				{\underset{#1}{\overset{#2}{\longleftarrow}}}
				{\xleftarrow[#1]{#2}}}}%
		{\xleftarrow[#1]{#2}}
	}
	\newcommand{\isoarrow}{%
		\ifbool{@display}{\overset{\sim}{\longrightarrow}}{\xrightarrow\sim}%
	}
	\renewcommand{\lra}{%
		\ifbool{@display}{\longleftrightarrow}{\leftrightarrow}%
	}
	\newcommand{\undertilde}{\raisebox{0.4ex}{\smash[t]{$\scriptstyle\sim$}}}

\begin{document}

\title[Ramified unitary local model]{On the moduli description of ramified unitary local models of signature $(n-1,1)$}


\author{\fnm{Yu} \sur{Luo}}\email{yluo237@wisc.edu}
\affil{\orgdiv{Department of Mathematics}, \orgname{University of Wisconsin-Madison}, \orgaddress{\street{480 Lincoln Drive}, \city{Madison}, \postcode{53706}, \state{WI}, \country{USA}}}


\abstract{We provide a moduli description of the ramified unitary local model of signature $(n-1,1)$ with arbitrary parahoric level structure, assuming the residue field has characteristic not equal to $2$, thereby confirming a conjecture of Smithling \cite{Smithling2015}. 
Our approach involves writing down explicit equations for the special fiber and proving that they define a normal, Cohen–Macaulay scheme, which is also of independent interest.
As applications, we obtain moduli descriptions for:
(1) ramified unitary Pappas–Zhu local models with arbitrary parahoric level;
(2) the irreducible components of their special fiber in the maximal parahoric case;
(3) integral models of ramified unitary Shimura varieties with arbitrary (quasi-)parahoric level.}

\maketitle
\tableofcontents

\section{Introduction}\label{intro}
\subsection{Background}\label{intro_back}
Given a Shimura variety $\Sh_K(G,h)$ defined over the reflex field $E$, one of the fundamental problems in arithmetic geometry is constructing and studying nice integral models of $\Sh_K(G,h)$ defined over the ring of integers $\CO_E$.
While those integral models have very complicated geometric structures, their local behaviors can be governed by some schemes defined in terms of linear algebra data, which are called local models.
	
To be more precise, suppose $\CO$ is the completion of $\CO_E$ at a non-archimedean place. Ideally, given an integral model $\CS_K(G,h)$ of the Shimura variety over $\Spec \CO$, there exists a corresponding local model $M_K(G,h)$ over $\Spec \CO$, and vice versa.
Moreover, their geometric structures are related by the \emph{local model diagram}:
\begin{equation*}
\begin{aligned}
\xymatrix{
&\wt{\CS}_K(G,h)\ar[ld]_-{\pi}\ar[dr]^-{\wt{\varphi}}\\
\CS_K(G,h) &&  M_K(G,h),
}
\end{aligned}
\end{equation*}
where $\pi$ is a principal homogeneous space of a smooth group scheme over $\CO$, and $\wt{\varphi}$ is a smooth morphism. 
These morphisms have the same relative dimension, hence $\CS_K(G,h)$ and $M_K(G,h)$ are \'etale locally identified; cf.\ \cite{RZ}.
Consequently, we can reduce questions of local nature about $\CS_K(G,h)$ to the corresponding questions about $M_K(G,h)$.
Since local models are constructed from linear algebra data, they should be easier to study than the integral models $\CS_K(G,h)$.

Motivated by this, it is natural to ask for the existence of ``nice'' local models.
In the celebrated paper \cite{PZ2013}, Pappas and Zhu gave a uniform group-theoretic construction of the local models for tamely ramified groups. Their construction was further refined by He, Pappas, and Rapoport \cite{HPR}, among others work. For tamely ramified groups, such refined model has been confirmed to be ``canonical'' in the sense of Scholze-Weinstein, cf.\, \cite[Corollary 2.17]{HPR}.

On the other hand, since the group-theoretic construction cannot provide a moduli description, it remains an interesting problem to find a moduli description for Pappas-Zhu's model.
In the PEL setting, Rapoport and Zink \cite{RZ} constructed natural moduli functors for integral models of Shimura varieties and their associated local models based on lattice-theoretic data. 
However, Pappas \cite{Pappas2000} observed that these functors may fail to be flat.
It is now known that the Pappas–Zhu model is a closed subscheme of the Rapoport-Zink model, see, for example, \cite[Lemma 4.1]{HPR}. 
The task is therefore to modify the Rapoport–Zink's moduli functor so that it describes the Pappas–Zhu model.

\subsection{Ramified unitary local model}\label{intro_ramified-unitary-lm}
In this paper, we focus on the local model associated with the unitary group defined by a hermitian form over a quadratic extension $F/F_0$ of a $p$-adic field. 
When $F/F_0$ is split or unramified, only algebraic groups of type $A$ are involved, and G\"ortz \cite{Gortz2001} has shown that the corresponding Rapoport–Zink models are flat.
In contrast, when $F/F_0$ is ramified, the associated Rapoport-Zink models may fail to be flat, see \cite{Pappas2000}.
		
To be more precise, let $F/F_0$ be a ramified quadratic extension of $p$-adic fields with $p\neq 2$\footnote{Let us note that the case $p=2$ is fundamentally more difficult and is not excluded merely for the sake of simplicity, we maintain this assumption \emph{throughout the paper.}}. 
Let $(V,\phi)$ be an $F/F_0$-hermitian space of dimension $n\geq 3$\footnote{The case when $n\leq 2$ is well-understood and has different patterns; cf. for example, \cite[\S 8]{RSZ2018}.} and let $m:=\lfloor n/2\rfloor$. When the hermitian form $\phi$ splits\footnote{The case where the hermitian form $\phi$ is non-split, so that the unitary group $\U(V,\phi)$ may be non-quasi-split, will be treated in \S \ref{application_nonsplit}.}, the conjugacy classes of parahoric subgroups of $\U(V)$ can be labeled by nonempty subsets $I\subset \{0,\cdots,m\}$ satisfying the condition
\begin{equation}\label{intro_back:local-dynkin}
	\text{ if }n\text{ is even, then }m-1\in I\Rightarrow m\in I
	\text{ (for odd $n$, no condition on $I$ is imposed)}.
\end{equation}
These index sets relate to the local Dynkin diagram; cf.\ \cite[4.a]{PR2008} or \S \ref{general_dyn}.

We also fix a pair $(r,s)$ with $r+s=n$, called the \emph{signature}. In this paper, we focus on the case $(r,s)=(n-1,1)$. 
Given such data, we define the Rapoport-Zink local model $M_I^{\naive}(r,s)$ associated with the quasi-split unitary group $\U(V)$, 
This is called the ``naive'' local model, since it is not flat in general; see \S \ref{moduli_naive} for the precise definition. 
		
In \cite[\S 11]{PR2008}, Pappas and Rapoport defined the local model $M_I^{\loc}(r,s)$ as the scheme-theoretic closure of the generic fiber of $M_I^\naive(r,s)$. This model coincides with the corresponding Pappas-Zhu local model, cf.\ Proposition \ref{application_nonsplit:PZ-vs-loc}.
As a first step toward a moduli description of $M_I^{\loc}(r,s)$, Pappas \cite{Pappas2000} introduced the \emph{wedge condition} to the moduli functor; cf.\ \S \ref{moduli_PR}.
Let $M_I^\wedge(r,s)$ denote the closed subscheme of $M_I^\naive(r,s)$ defined by the wedge condition, this give rise to a chain of closed subschemes of the naive model, which are equalities when restricting to the generic fibers:
\begin{equation*}
	M_I^\loc(r,s)\subset M_I^\wedge(r,s)\subset M^\naive_I(r,s).
\end{equation*}
In the self-dual maximal parahoric case, i.e., when $I=\{0\}$, Pappas conjectured that $M^\wedge_{\{0\}}(r,s)=M^\loc_{\{0\}}(r,s)$, and verified this for the signature $(r,s)=(n-1,1)$.
		
However, for the other index sets $I$, the wedge condition is not enough. The next advance came in \cite{PR2009} where Pappas and Rapoport introduced the \emph{spin condition}; cf.\ \S \ref{moduli_PR}.
Denote by $M_I^{\spin}(r,s)$ the closed subscheme of $M_I^\wedge(r,s)$ defined by this condition, they conjectured that $M_I^{\spin}(r,s)=M_I^{\loc}(r,s)$.
Later, Rapoport, Smithling and Zhang \cite{RSZ2017} verified that when $n$ is even, the inclusion $M_{\{m\}}^\loc(n-1,1)\subset M_{\{m\}}^\spin(n-1,1)$ is an equality.
In earlier work \cite{Smithling2011, Smithling2014}, Smithling showed that the spin models have the same underlying topological space as the local models under the inclusion $M^\loc_I(r,s)\subset M^\spin_I(r,s)$.
However, in \cite{Smithling2015}, in response to a question of Rapoport, he showed that $M^{\spin}_{\{m\}}(n-1,1)$ is not equal to $M^\loc_{\{m\}}(n-1,1)$ when $n\geq 5$ is odd\footnote{For another counterexample, see \cite[Remark 9.14]{RSZ2018}.}.
In the same paper, Smithling introduced the \emph{strengthened spin condition}, which defines a closed subscheme $M_I(r,s)$; cf.\ \S \ref{moduli_ss}. The strengthened spin model fits into the chain of closed immersions:
\begin{equation*}
	M_I^\loc\subset M_I\subset M_I^\spin\subset M_I^\wedge\subset M_I^\naive.
\end{equation*}
Smithling conjectured that $M_I(r,s)=M_I^{\loc}(r,s)$, and confirmed this in the case $I=\{m\}$ with $(r,s)=(n-1,1)$ for odd $n\geq 5$.
Later, Yu verified in her thesis \cite{Yu2019} that when $n=2m$, we have $M_{\{m-1,m\}}^\loc(n-1,1)=M_{\{m-1,m\}}(n-1,1)$.
		
The main result of this paper is the following:
		
\begin{theorem}\label{intro_back:main}
For any nonempty subset $I\subset\{0,1,\cdots,m\}$ satisfying \eqref{intro_back:local-dynkin}, the inclusion $M_I^\loc(n-1,1)\subset M_I(n-1,1)$ is an equality.
In other words, for signature $(n-1,1)$, the quasi-split ramified unitary local model is represented by the strengthened spin model.
\end{theorem}

\begin{remark}
Theorem \ref{intro_back:main} also holds when $F/F_0=k(u)/k(t)$ is a ramified quadratic extension of function fields with $u^2=t$, where $k$ is a perfect field of characteristic $p\neq 2$. 
\end{remark}

We further obtain a moduli description of the Pappas-Zhu model associated with \emph{any} $F/F_0$-hermitian space $(V,\phi)$ where $F/F_0$ is a ramified quadratic extension of $p$-adic fields with $p\neq 2$. See \S \ref{application_nonsplit} for details.

Suppose $(G,\{\mu\},K)$ is a local model triple \cite[\S 2.1]{HPR} where $G=\GU(V,\phi)$ is the unitary similitude group of the hermitian space $(V,\phi)$. Assume that the geometric cocharacter class $\{\mu\}$ gives the signature $(n-1,1)$, and $K$ is a parahoric subgroup of $G(F)$.
The result is the following:
\begin{theorem}\label{intro_back:non-split}
The Pappas-Zhu local model $M^\loc_K(G,\{\mu\})$ is represented by the moduli functor $\CM_\CL(G,\mu)$ defined in \S \ref{application_nonsplit_moduli}.
\end{theorem}

\begin{remark}
The moduli functor here is defined only with the strengthened spin condition. In particular, we do not need to impose the Kottwitz condition or the wedge condition, see Remark \ref{moduli_ss:imply-kott}.
\end{remark}

\subsection{Ramified unitary local models of strongly non-special parahoric subgroups}
Using the theory of (partial) affine flag varieties and their relation to local models, we can reduce the proof of Theorem \ref{intro_back:main} into the cases when the parahoric subgroups are maximal; see \S \ref{general} for details of the reduction. 
Furthermore, we will focus on the following cases.

\begin{definition}\label{intro_back:general-max}
Suppose $n\geq 3$. A non-empty subset $I\subset\{0,1,\cdots,m\}$ is called \emph{strongly non-special} if: (1) it equals a singleton set $\{\kappa\}$; (2) when $n=2m$, $\kappa$ is not equal to $0,m-1,m$; (3) when $n=2m+1$, $\kappa$ is not equal to $0,m$.
\end{definition}

In \S \ref{moduli_setup_alg-group}, we will relate non-empty index sets $I$ with parahoric subgroups. A strongly non-special index set $I$ will correspond to a maximal parahoric subgroup which is neither special nor self-dual.
Note that the first example of strongly non-special index appears when $n=5$ and $\kappa=1$.

We will prove that when the index is strongly non-special,
the strengthened spin models represent the local models:

\begin{theorem}\label{intro_back:max}
Let $n\geq 5$, suppose the signature is $(n-1,1)$ and $I=\{\kappa\}$ is strongly non-special. 
\begin{altenumerate}
\item The strengthened spin model $M_{\{\kappa\}}$ represents the local model $M^\loc_{\{\kappa\}}$. It is flat, normal, and Cohen-Macaulay of dimension $n$.
\item The special fiber $M^\loc_{\{\kappa\},s}$ is reduced and has two irreducible components $Z_1$ and $Z_2$, each of which is normal and Cohen-Macaulay of dimension $n-1$.
\item The intersection of the irreducible components $Z_1\cap Z_2$ is an irreducible quadratic cone of dimension $n-2$, which is normal and Cohen-Macaulay.
\item Let the ``worst point'' be the unique closed Schubert cell contained in the special fiber of the local model; cf.\ Lemma \ref{equ_chart:wp}.
The closed subvarieties $Z_1$, $Z_2$, and their intersection $Z_1\cap Z_2$ are all singular, and their singular locus consists solely of the worst point. 
In particular, the local model $M^\loc_{\{\kappa\}}$ does not admit a semi-stable reduction.
\item(\S \ref{application_moduli-irr})  The irreducible components $Z_1$ and $Z_2$, as well as their intersection $Z_1\cap Z_2$, have moduli descriptions.
\end{altenumerate}
\end{theorem}
\begin{remarks}
\begin{altenumerate}
\item We exclude the cases $I=\{0\}$ and $I=\{m\}$ for all $n$, and the case $I=\{m-1,m\}$ when $n$ is even, since they already appear in \cite{Pappas2000,Smithling2015,Yu2019}, and they have different geometric structures.
But our proof generalizes to all those cases except when $I=\{m\}$ for $n=2m$; cf.\ Remarks \ref{max_red:recover}.
\item All those properties, including the moduli descriptions of the irreducible components of the special fiber, also apply to the ramified unitary Pappas-Zhu model $M_K^\loc(G,\{\mu\})$ of signature $(n-1,1)$ when $K$ is a strongly non-special parahoric subgroup. For instance, they remain valid for the local model associated with a non-split Hermitian space $(V,\phi)$. See \S \ref{application_nonsplit} for more comments.
\item Some results in Theorem \ref{intro_back:max} have been proved using the group-theoretic data. For instance, by \cite{PZ2013} (resp.\ \cite{HR22}), all the ramified unitary local models are normal (resp. Cohen-Macaulay). By \cite{HPR}, the ramified unitary local models with a strongly non-special parahoric level structure do not admit semi-stable reductions. Comparing to their proofs, ours are more ad hoc and direct.

\item Let $*\in M^{\loc}_{\{\kappa\}}$ denote the worst point of the local model. It is proved in \cite[Corollary 1.3.2]{HLS-regular} that the complement $M^{\loc}_{\{\kappa\}}\setminus \{*\}$ admits a semi-stable reduction.

\item As a by-product, we also verify that the local model $M^\loc_{\{\kappa\}}$ is Gorenstein when $n=4\kappa+2$. Determining whether the local model is Gorenstein in the remaining cases remains an interesting open question, see also Remark \ref{max_red:non-gorenstein}.
\end{altenumerate}
\end{remarks}

To compare the strengthened spin model with the local model,  it suffices to verify that the special fiber of the strengthened spin model is reduced, since both models are topologically flat; cf.\ \S \ref{moduli_fc}.

Using the equivariant action of the loop group, we can further restrict to an affine neighborhood $U_{\{\kappa\}}$ of the worst point in the strengthened spin local model $M_{\{\kappa\}}$; cf.\ \S \ref{equ_chart}.
The proof of Theorem \ref{intro_back:max} then follows from the following explicit computation of $U_{\{\kappa\}}$:
\begin{theorem}\label{intro_back:maxco}
Let $U_{\{\kappa\}}$ be the open affine chart defined in the previous paragraph.
Given a matrix $X$ of indeterminates, let $\bigwedge^2 X$ denote the set of all its $2$-minors. Let $\CO=\CO_F$ be the ring of the integers in $F$, with perfect residue field $k$.
Under the same assumption as in Theorem \ref{intro_back:max}, we have:
\begin{altenumerate}
\item The open affine chart $U_{\{\kappa\}}$ is isomorphic to the spectrum of the ring
\begin{equation*}
    R_{\CO}\simeq \frac{\CO[\bA,\bB]}{\bigwedge^2(\bA,\bB),\bA-\bA^t,\tr(\bA H)-2\pi}.
\end{equation*}
Here $\bA$, resp.,\ $\bB$
\footnote{The $\bA$ and $\bB$ we define here is the same as in \eqref{max_red:sub}, but different from $A$ and $B$ in \eqref{equ_chart:reordered-X}.}
is a matrix of indeterminates of size $(n-2\kappa)\times (n-2\kappa)$, resp.,\ $(n-2\kappa)\times 2\kappa$, and $H$ is the anti-diagonal unit matrix of size $(n-2\kappa)\times(n-2\kappa)$.
    \item The special fiber $U_{\{\kappa\},s}$ of the open affine chart is isomorphic to the spectrum of the ring
\begin{equation*}
    R_s\simeq \frac{k[\bA,\bB]}{\bigwedge^2(\bA,\bB),\bA-\bA^t,\tr(\bA H)}.
\end{equation*}
The ring $R_s$ is reduced and Cohen-Macaulay of dimension $n-1$. 

\item The affine chart $U_{\{\kappa\},s}$ has two irreducible components, isomorphic to the spectra of the rings
\begin{equation*}
    R_{s,1}\simeq \frac{k[\bA,\bB]}{\bA,\bigwedge^2 \bB}\quad\text{and}\quad R_{s,2}\simeq \frac{k[\bA,\bB]}{\bigwedge^2(\bA,\bB),\bA-\bA^t,\tr(\bA H),\bB^tH\bB}.
\end{equation*}
Both are normal and Cohen-Macaulay of dimension $n-1$, with the singular locus defined by $\bA=0,\bB=0$.

\item The intersection of $\Spec R_{s,1}$ and $\Spec R_{s,2}$ is isomorphic to the spectrum of the ring
\begin{equation*}
    R_{s,12}\simeq \frac{k[\bA,\bB]}{\bA,\bigwedge^2\bB,\bB^tH\bB}.
\end{equation*}
The scheme $\Spec R_{s,12}$ is normal, Cohen-Macaulay, and irreducible of dimension $n-2$, with the singular locus defined by $\bA=0,\bB=0$.
\end{altenumerate}
\end{theorem}

The main body of the paper will focus on the proof of Theorem \ref{intro_back:maxco} (ii) and (iv). Part (i) of the theorem is irrelevant to the main result, and its proof will be postponed to Section \ref{intequ}. Part (iii) of the theorem will be a consequence of (ii) and (iv).

To get the defining equations in (ii), we mimic Arzdorf's computation \cite{Arzdorf2009}, which is further generalized in \cite{Smithling2015, Yu2019}.
We explicitly write down the defining equations of the ring $R_s$ using the ``worst term''; cf.\ Definition \ref{equ_ss-setup:wt-def}.
We then show that $R_s$ is Cohen-Macaulay and generically smooth. By Serre’s criterion, it follows that $R_s$ is reduced; see Corollary \ref{max_red:lm}.
This part is based on Conca's work about the symmetric determinant varieties \cite{Conca1994-1,Conca1994-2}.
Next, we prove that $R_{s,12}$ is normal using a Gr\"obner basis computation.
The remaining assertions follow from conceptual arguments.

\begin{remark}
\begin{altenumerate}
\item The commutative algebra results in Theorem \ref{intro_back:max}, in particular the normal and Cohen-Macaulay properties of $R_{s,2}$ and $R_{s,12}$, may be of independent interests.
\item It is quite interesting to point out that the defining equations of the local models $M^\loc_{\{\kappa\}}$ given in Theorem \ref{max_red:reduced_sym} look the same for even and odd $n$. 
From the group-theoretic point of view, these two cases are expected to be different. For instance, they have different types of local Dynkin diagrams. 
The same phenomenon also appears in the orthogonal setting: the even and odd dimensional special orthogonal groups fall into two classes of (local) Dynkin diagrams. But the defining equations of the local models have the same form; see \cite{Zachos}.

Another interesting phenomenon concerns the singularity structure. In both the unitary and orthogonal cases, the only singular point on each irreducible component of the special fiber is the worst point. Moreover, blowing up the worst point yields a regular integral model. The orthogonal case is treated in \cite{PZa2022}.
For the ramified unitary case, this was first studied by Kr\"amer in \cite{Kramer} for self-dual level, and later extended to all vertex levels in \cite{HLS-regular}.
\end{altenumerate}
\end{remark}
		
\subsection{Moduli description of the local model of the unitary group}\label{intro_moduli}
For the reader’s convenience and future reference, we briefly summarize the moduli description of the unitary local models of signature $(n-1,1)$, with appropriate citations.
The axioms {\fontfamily{cmtt}\selectfont LM1-LM8} used here are listed in \S \ref{moduli}.

We start with the relative setup. Let $F_0$ be a $p$-adic local field and let $F$ be an \'etale quadratic extension of $F_0$. 
Let $(V,\phi)$ be a $F/F_0$-hermitian space. Consider now the unitary similitude group $G=\GU(V,\phi)$ defined over $F_0$. After base change to an algebraic closure $\ov{F_0}$, the group splits as $G_{\ov{F_0}}\simeq \BG_m\times\GL_n$. 

We fix a minuscule cocharacter $\mu$ given by $z\mapsto (z,\mathrm{diag}(z^{(1)},1^{(n-1)}))$, corresponding to the signature $(n-1,1)$ considered in this paper. 
Let $\CL=\{L_i\}_{i\in I}$ be a self-dual lattice chain in the sense of \cite[Chapter 3]{RZ}. These data determine the (Pappas-Zhu) local model $M^{\loc}_{\CL}(G,\{\mu\})$. 

When $F/F_0$ is split or unramified, the local model can be defined by the axioms {\fontfamily{cmtt}\selectfont LM1-LM5} (see \S \ref{moduli_naive}) and is flat with semi-stable reduction, as shown in \cite{Gortz2001}, see also \cite[Appendix A]{RSZ-2021} for the case when $p=2$. 

When $F/F_0$ is ramified and $p\neq 2$, the local model is defined using the axioms {\fontfamily{cmtt}\selectfont LM1-LM4} and {\fontfamily{cmtt}\selectfont LM8}. It is flat, normal, and Cohen-Macaulay. 

In certain special cases, the moduli description can be simplified, as discussed in \S \ref{intro_ramified-unitary-lm}. These cases are summarized in the table below. Moreover, in some cases, there exist variants of the integral model—often referred to as splitting models—which are not canonical models in the sense of Scholze–Weinstein, but exhibit better geometric properties; we include these in the comments.
{\small{\setlongtables
\renewcommand{\arraystretch}{1.25}
\begin{longtable}{|c|c|c|c|c|}
\hline
\begin{varwidth}{\linewidth}
   \centering
  Index 
\end{varwidth} 
&
\begin{varwidth}{\linewidth}
   \centering
 Moduli 
\end{varwidth} 
   & \begin{varwidth}{\linewidth}
   \centering
   Reference
\end{varwidth} 
   &  \begin{varwidth}{\linewidth}
   \centering
 Comments 
\end{varwidth}       
       \\
    \hline
   $\begin{array}{cc}
   I=\{0\}\\
   \text{(self-dual)}
   \end{array}$ 
 & {\fontfamily{cmtt}\selectfont LM1-LM6} &   \cite{Pappas2000} 
 & Semi-stable model with moduli description \cite{Kramer}
  \\
\hline
$\begin{array}{cc}
   n=2m, I=\{m\}\\
   \text{($\pi$-modular)}
   \end{array}$
 & {\fontfamily{cmtt}\selectfont LM1-LM6+LM7} &   \cite{RSZ2017} 
 & A variant of spin condition is in \cite[\S 3.1]{RSZ2017}
   \\
\hline
$\begin{array}{cc}
   n=2m+1, I=\{m\}\\
   \text{(almost $\pi$-modular)}
   \end{array}$
 & {\fontfamily{cmtt}\selectfont LM1-LM5+LM8} &   \cite{Smithling2015} 
 & Semi-stable model without moduli description \cite{Richarz13}
    \\
\hline
$\begin{array}{cc}
 n=2m \\
 I=\{m-1,m\}
 \end{array}$
 & {\fontfamily{cmtt}\selectfont LM1-LM5+LM8} &   \cite{Yu2019} 
 & Semi-stable model with moduli description \cite{Yu2019} 
   \\
\hline
$I$ strongly non-special
 & {\fontfamily{cmtt}\selectfont LM1-LM5+LM8} &   Thm. \ref{intro_back:max} 
 & $\begin{array}{cc}
 \text{Splitting model with moduli description \cite{Zachos-Zhao-2024}}\\
 \text{Semi-stable model without moduli description \cite{HLS-regular}}
 \end{array}$
   \\
\hline
$I$ arbitrary
 & {\fontfamily{cmtt}\selectfont LM1-LM5+LM8} &   Thm. \ref{intro_back:non-split} 
 &    \\
\hline
\end{longtable}}}

\begin{remark}
\begin{altenumerate}
\item For $n\leq 3$, the moduli description of the integral models admits further refinements, see \cite[Remarks 5.2.1]{LRZ25}.
\item For $p=2$, only partial results are available; see the works of Kirch \cite{Kirch2018} and J. Yang \cite{JYang-2024}.
\end{altenumerate}
\end{remark}

For global applications, we are also interested in the algebraic group over $\BQ_p$ defined by
$$
G^{\BQ_p}:=\{g\in\mathrm{Res}_{F_0/\BQ_p}\mathrm{GU}(V)\mid c(g)\in\BG_{m,\BQ_p}\},
$$
where $c(g)$ denotes the similitude factor. After base change to an algebraic closure, we have
$$
G^{\BQ_p}\otimes{\ov{\BQ}_p}\simeq \BG_m\times (\GL_n)^{[F_0:\BQ_p]}.
$$
We choose the nontrivial cocharacter $\mu$ for one of the $\GL_n$ and take the trivial cocharacter for the others. See \cite[\S 2.2]{RSZ-2021} for a more detailed discussion.
In this case, the (absolute) Pappas-Zhu local model $M^{\loc}_\CL(G^{\BQ_p},\{\mu^{\BQ_p}\})$ is well-behaved, thanks to the work of Levin \cite{Levin16}. A moduli description of this model has been studied by several authors \cite{PR2003,PR2005,RZ17,Mi2022,RSZ-2021,KRZ2024}, using the Eisenstein condition, and has been extended to full generality in \cite{LMZ25}. We refer the reader to \cite[\S 3]{LMZ25} for a comprehensive treatment of this question, which is independent of the content of this paper.

\subsection{The structure of the paper}
The organization of the paper is as follows.
In Section 2, we review different moduli functors related to the ramified unitary local models with the split hermitian form, including the naive, wedge, spin, and strengthened spin models. We will also define the local model and discuss its relations with those moduli functors. 
In Section 3, we focus on the strongly non-special case. We will translate the axioms defining the strengthened spin models into commutative algebra relations.
In Section 4, we simplify those relations and study the geometric properties of the strengthened spin model for the strongly non-special parahoric subgroup, in particular, we deduce Theorem \ref{intro_back:maxco} (ii)-(iv).
In Section 5, we review the theory of affine flag varieties, and use them to prove Theorem \ref{intro_back:main}.
In Section 6, we give moduli descriptions for any ramified unitary local models with signature $(n-1,1)$, whose hermitian form are not necessarily split. We will also discuss a moduli description of the irreducible components of the special fibers local model in the strongly non-special case. Then, we will give a global application to Shimura varieties.
In Section 7, we give the proofs of some propositions left in Section 3.
In Section 8, we give the defining equations of the local model at an affine neighborhood of the worst point over the ring of integers.

\subsection{Acknowledgments}
The author wants to thank his advisor Michael Rapoport for helpful discussions and for reading drafts of this paper.
The author wants to thank Qiao He and Yousheng Shi for presenting the author with this problem. 
The author wants to thank Daniel Erman, Brian Lawrence, and Qihang Li for helpful discussions. 
The author wants to thank Ulrich G\"ortz, Thomas Haines, Jie Yang, and the referee for helpful comments.
		
\subsection{Conflict of Interest Statement}
On behalf of all authors, the corresponding author states that there is no conflict of interest. 
\section{The moduli problems}\label{moduli}
In this section, we review the definitions of $M_I^\loc$, $M_I^\naive$, $M_I^\wedge$, and $M_I^\spin$ from \cite{PR2009}, and the strengthened spin model $M_I$ from \cite{Smithling2015}.

\subsection{Basic setup}\label{moduli_setup}
\subsubsection{Notation}\label{moduli_setup_notation}
Assume $\CO_F/\CO_{F_0}$ is a quadratic extension of complete discrete valuation rings with the same perfect residue field $k$ with $\text{char}(k)=p\neq 2$ and uniformizers $\pi$, resp. $\pi_0$ such that $\pi^2=\pi_0$.

We will work with a fixed integer $n\geq 3$. For $i\in\{1,\cdots,n\}$, we write 
\begin{equation*}
    i^\vee:=n+1-i.
\end{equation*}
For $i\in \{1,\cdots,2n\}$, we write
\begin{equation*}
    i^*:=2n+1-i.
\end{equation*}
For $S\subset\{1,\cdots,2n\}$, we write
\begin{equation*}
    S^*=\{i^*\mid i\in S\}\text{ and }S^\perp=\{1,\cdots,2n\}\setminus S^*.
\end{equation*}	
We also define
\begin{equation*}
    \sum S:=\sum_{i\in S} i.
\end{equation*}
		
For a real number $a$, we write $\lfloor a\rfloor$ for the greatest integer $\leq a$, and $\lceil a\rceil$ for the least integer $\geq a$.
We write $a,\cdots,\wh{b},\cdots,c$ for the list $a,\cdots,c$ with $b$ omitted.

Suppose $X$ is a scheme defined over a discrete valuation ring $\CO$; we will denote by $X_\eta$ (resp.,\ $X_s$) the generic (resp.,\ special) fiber of the scheme.
\subsubsection{Linear algebra setup of the split hermitian form}\label{moduli_setup_alg-group}
Consider the vector space $F^n$ with the standard $F$-basis $e_1,\cdots,e_n$. Consider the split $F/F_0$-Hermitian form
\begin{equation*}\label{moduli_setup:herm}
    \phi: F^n\times F^n\rightarrow F,\quad \phi(ae_i,be_j)=\bar{a}b\delta_{ij^\vee},\quad a,b\in F,
\end{equation*}
where $a\mapsto\bar{a}$ is the nontrivial element of $\Gal(F/F_0).$ 
Attached to $\phi$ are the respective ``alternating'' and ``symmetric'' $F_0$-bilinear forms $F^n\times F^n\rightarrow F_0$ given by
\begin{equation*}
\langle x,y\rangle:=\frac{1}{2}\tr_{F/F_0}(\pi^{-1}\phi(x,y))
\quad\text{and}\quad
(x,y):=\frac{1}{2}\tr_{F/F_0}(\phi(x,y)).
\end{equation*}		
		
For each integer $i=bn+c$ with $0\leq c<n$, define the standard $\CO_F$-lattices
\begin{equation}\label{moduli_setup:lattice}
	\Lambda_i:=\sum_{j=1}^c\pi^{-b-1}\CO_F e_j+\sum_{j=c+1}^n\pi^{-b}\CO_F e_j\subset F^n.
\end{equation}
		
For all $i$, the $\langle\,,\,\rangle$-dual of $\Lambda_i$ in $F^n$ is $\Lambda_{-i}$, by which we mean that
\begin{equation*}
    \{x\in F^n\mid \langle \Lambda_i,x\rangle\subset\CO_{F_0}\}=\Lambda_{-i}.
\end{equation*}
We have a perfect $\CO_{F_0}$-bilinear pairing:
\begin{equation*}\label{moduli_setup:dual}
	\Lambda_i\times\Lambda_{-i}\xrightarrow{\langle\,,\,\rangle} \CO_{F_0}
\end{equation*}
Similarly, $\Lambda_{n-i}$ is the $(\,,\,)$-dual of $\Lambda_i$ in $F^n$. The $\Lambda_i$'s forms a complete, periodic, self-dual lattice chain
\begin{equation*}
    \cdots\subset\Lambda_{-2}\subset\Lambda_{-1}\subset\Lambda_0\subset\Lambda_1\subset\Lambda_2\subset\cdots.
\end{equation*}
		
Let $I \subset \{0,\dotsc,m\}$ be a nonempty subset satisfying \eqref{intro_back:local-dynkin}, and let $r + s = n$ be a partition.  Let
\begin{equation*}
    E = F \quad\text{if}\quad r \neq s \quad\text{and}\quad E = F_0 \quad\text{if}\quad r = s.
\end{equation*}
This field $E$ correspond to the reflex field of the unitary similitude Shimura variety. We will always assume $n=r+s\geq 3$, and  $(r,s)=(n-1,1)$ for our result.
		
The unitary similitude group we are interested in is defined as a reductive group over $F_0$, such that given any $F_0$-algebra $R$, we have
\begin{equation*}
	\GU(F^n,\phi)(R)=\{g\in\GL_F(F^n\otimes_{F_0}R)\mid \langle gx,gy\rangle=c(g)\langle x,y\rangle,c(g)\in R^\times\}.
\end{equation*}
We define the group $\CP_I$ as a smooth group scheme over $\CO_{F_0}$, which represents the moduli problem that associates each $\CO_{F_0}$-algebra $R$ to the group
\begin{equation*}\label{P_I}
	\CP_I(R)=\{g\in\GU(F^n,\phi)(R\otimes_{\CO_{F_0}}F_0)\mid g(\Lambda_i\otimes_{\CO_{F_0}}R)=\Lambda_i\otimes_{\CO_{F_0}}R,\forall i\in I\}.
\end{equation*}
And $\CP_I^\circ$ is defined as the neutral connected component of $\CP_I$. We define $P_I:=\CP_I(\CO_{F_0})$, and $P_I^\circ:=\CP_I^\circ(\CO_{F_0})$. The latter is a (connected) parahoric subgroup of $\GU(V,\phi)$.

\subsection{Naive model}\label{moduli_naive}
For the remaining part of this section, we define several moduli functors related to local models. Historically, these functors served as candidates for the ramified unitary local models with split hermitian form. 
While it is now standard to define local models of specific groups purely in terms of group-theoretic data, the classical moduli-theoretic approach remains essential for the purposes of this paper.
		
The \emph{naive local model $M_I^\naive$} is a projective scheme over $\Spec \CO_E$. It represents the moduli problem that sends each $\CO_E$-algebra $R$ to the set of all families:
\begin{equation*}
    (\sF_i \subset \Lambda_i \otimes_{\CO_{F_0}}R)_{i\in \pm I + n\mathbb{Z}},
\end{equation*}
such that
\begin{altenumerate}		
	\item[\fontfamily{cmtt}\selectfont LM1.]\label{moduli_naive:LM1}
		for all $i$, $\sF_i$ is an $\CO_F \otimes_{\CO_{F_0}} R$-submodule of $\Lambda_i \otimes_{\CO_{F_0}} R$, and an $R$-direct summand of rank $n$;
	\item[\fontfamily{cmtt}\selectfont LM2.]\label{moduli_naive:LM2}
		for all $i < j$, the natural arrow $\Lambda_i \otimes_{\CO_{F_0}} R \to \Lambda_j \otimes_{\CO_{F_0}} R$ carries $\sF_i$ into $\sF_j$;
	\item[\fontfamily{cmtt}\selectfont LM3.]\label{moduli_naive:LM3}
		for all $i$, the isomorphism $\Lambda_i \otimes_{\CO_{F_0}} R \xra[\undertilde]{\pi \otimes 1} \Lambda_{i-n} \otimes_{\CO_{F_0}} R$ identifies
		\begin{equation*}
		    \sF_i \isoarrow \sF_{i-n};
		\end{equation*}
	\item[\fontfamily{cmtt}\selectfont LM4.]\label{moduli_naive:LM4}
		for all $i$, the perfect $R$-bilinear pairing
        \begin{equation*}
            (\Lambda_i \otimes_{\CO_{F_0}} R) \times (\Lambda_{-i} \otimes_{\CO_{F_0}} R)
			\xra{\aform \otimes R} R
        \end{equation*}
		identifies $\sF_i^\perp$ with $\sF_{-i}$ inside $\Lambda_{-i} \otimes_{\CO_{F_0}} R$; and
	\item[\fontfamily{cmtt}\selectfont LM5.]\label{moduli_naive:kott}
        (Kottwitz condition) for all $i$, the element $\pi \otimes 1 \in \CO_F\otimes_{\CO_{F_0}} R$ acts on $\sF_i$ as an $R$-linear endomorphism with characteristic polynomial
		\begin{equation*}
		    \det(T\cdot \id - \pi \otimes 1 \mid \sF_i) = (T-\pi)^s(T+\pi)^r \in R[T].
		\end{equation*}
		\end{altenumerate}
		When $r = s$, the polynomial on the right-hand side of {\fontfamily{cmtt}\selectfont LM5} is $(T^2 - \pi_0)^s$.

\begin{remark}
Using axiom {\fontfamily{cmtt}\selectfont LM3}, axiom {\fontfamily{cmtt}\selectfont LM4} is equivalent to requiring that the perfect $R$-bilinear pairing
\begin{equation*}
     (\Lambda_i \otimes_{\CO_{F_0}} R) \times (\Lambda_{n-i} \otimes_{\CO_{F_0}} R)
			\xra{\sform \otimes R} R
\end{equation*}
identifies $\CF_i^\perp$ with $\CF_{n-i}$ inside $\Lambda_{n-i}\otimes_{\CO_{F_0}}R$.
\end{remark}
\subsection{Wedge and spin conditions}\label{moduli_PR}
\subsubsection{Wedge condition}\label{moduli_PR_wedge}
The \emph{wedge condition} on an $R$-point $(\sF_i)_i$ of $M_I^\naive$ is that
\begin{altenumerate}
	\item[\fontfamily{cmtt}\selectfont LM6.]\label{moduli_PR:wedge}
	   if $r \neq s$, then for all $i$,
        \begin{equation*}
            \sideset{}{_R^{s+1}}{\bigwedge} (\,\pi\otimes 1 + 1 \otimes \pi \mid \sF_i\,) = 0
			\quad\text{and}\quad
			\sideset{}{_R^{r+1}}{\bigwedge} (\,\pi\otimes 1 - 1 \otimes \pi \mid \sF_i\,) = 0.
        \end{equation*}
			(There is no condition when $r=s$.)
\end{altenumerate}
		
The \emph{wedge local model $M_I^\wedge$} is the closed subscheme in $M_I^\naive$ cut out by the wedge condition.
Note that the splitting of the hermitian form is unnecessary for defining the wedge condition.

\subsubsection{Spin condition}\label{moduli_PR_spin}
		
Next, we turn to the spin condition, which involves the symmetric form \sform and requires more notation.  Let
\begin{equation*}
    V := F^n \otimes_{F_0} F,
\end{equation*}
regarded as an $F$-vector space of dimension $2n$ via the action of $F$ on the right tensor factor.  Let
\begin{equation*}
    W := \sideset{}{_F^n}\bigwedge V.
\end{equation*} 
		
When $n$ is even, the form \sform splits over $F^n$; when $n$ is odd, \sform splits after tensoring with $F$.
In both cases there is an $F$-basis $f_1,\dotsc,f_{2n}$ of $V$ such that $(f_i,f_j) = \delta_{ij^*}$.  
Hence there is a canonical decomposition of $W$ as an $\SO\bigl(\sform\bigr)(F) \simeq \SO_{2n}(F)$-representation:
\begin{equation*}
	W = W_1 \oplus W_{-1}.
\end{equation*}
Intrinsically, $W_1$ and $W_{-1}$ have the property that for any totally isotropic $n$-dimensional subspace $\sF \subset V$, the line $\bigwedge_F^n\sF \subset W$ is contained in $W_1$ or $W_{-1}$. In this way, they distinguish the two connected components of the orthogonal Grassmannian $\OGr(n, V)$ over $\Spec F$; cf.\ \cite[\S 8.2.1]{PR2009}.
Concretely, $W_1$ and $W_{-1}$ can be described as follows.  
For $S = \{i_1< \dots < i_n\} \subset \{1,\dotsc,2n\}$ of cardinality $n$, let
\begin{equation*}\label{moduli_PR:disp-f-S}
	f_S := f_{i_1} \wedge \dotsb \wedge f_{i_n} \in W.
\end{equation*}
Let $\sigma_S$ be the permutation on $\{1,\dotsc,2n\}$ sending
\begin{equation*}
    \{1,\dotsc,n\} \xra[\undertilde]{\sigma_S} S
\end{equation*} 
in increasing order, and
\begin{equation*}
    \{n+1,\dotsc,2n\} \xra[\undertilde]{\sigma_S} \{1,\dotsc,2n\} \smallsetminus S
\end{equation*}
in increasing order.  For varying $S$ of cardinality $n$, the $f_S$'s form a basis of $W$, and we define an $F$-linear operator $a$ on $W$ such that:
\begin{equation*}
    a(f_S) := \sgn(\sigma_S)f_{S^\perp}.
\end{equation*}
Then, when $f_1,\dotsc,f_{2n}$ is a split basis for \sform,
\begin{equation}\label{moduli_PR:W_+-1}
	W_{\pm 1} = \spann_F\bigl\{\,f_S \pm \sgn(\sigma_S)f_{S^\perp}\mid  \#S=n\,\bigr\}
\end{equation}
is the $\pm 1$-eigenspace for $a$.  Any other split basis is carried onto $f_1,\dotsc,f_{2n}$ by an element $g$ in the orthogonal group.  If $\det g = 1$ then $W_1$ and $W_{-1}$ are both $g$-stable, whereas if $\det g = -1$ then $W_1$ and $W_{-1}$ are interchanged by $g$.  In this way, $W_1$ and $W_{-1}$ are independent of choices up to labeling.

In this paper, we pin down a particular choice of $W_1$ and $W_{-1}$ as in \cite[\S 7.2]{PR2009} and \cite{Smithling2011, Smithling2014}. 
If $n = 2m$ is even, then
\begin{equation}\label{moduli_PR:even-basis}
	-\pi^{-1}e_1,\dotsc,-\pi^{-1}e_m, e_{m+1},\dotsc,e_n,e_1,\dotsc,e_m, \pi e_{m+1},\dotsc, \pi e_{n}	
\end{equation}
is a split ordered $F_0$-basis for \sform in $F^n$, and we take $f_1,\dotsc,f_{2n}$ to be the image of this basis in $V$.  
		
If $n = 2m+1$ is odd, then we take $f_1,\dotsc,f_{2n}$ to be the following split ordered basis
\begin{equation}\label{moduli_PR:odd-basis}
	\begin{gathered}
		-\pi^{-1} e_1 \otimes 1, \dotsc, -\pi^{-1} e_m \otimes 1, e_{m+1}\otimes 1 - \pi e_{m+1} \otimes \pi^{-1}, e_{m+2} \otimes 1,\dotsc, e_n\otimes 1, \\
		e_1\otimes 1,\dotsc, e_m \otimes 1, \frac{e_{m+1} \otimes 1 + \pi e_{m+1} \otimes \pi^{-1}}2, \pi e_{m+2} \otimes 1, \dotsc, \pi e_{n} \otimes 1.
	\end{gathered}
\end{equation}
For $\Lambda$ an $\CO_{F_0}$-lattice in $F^n$, we define a natural $\CO_F$-lattice in $W$:
\begin{equation}\label{moduli_PR:lattice-in-W}
	W(\Lambda) := \sideset{}{_{\CO_F}^n}\bigwedge (\Lambda \otimes_{\CO_{F_0}} \CO_F),
\end{equation}
and we define two sublattices:
\begin{equation*}
    W(\Lambda)_{\pm 1} := W_{\pm 1} \cap W(\Lambda).
\end{equation*}

We now formulate the spin condition.  If $R$ is an $\CO_F$-algebra, then the \emph{spin condition} on an $R$-point $(\sF_i \subset \Lambda_i \otimes_{\CO_{F_0}}R )_i$ of $M_I^\naive$ is:
\begin{altenumerate}
	\item[\fontfamily{cmtt}\selectfont LM7.] \label{moduli_PR:spin}
	    for all $i$, the line $\bigwedge_R^n \sF_i \subset W(\Lambda_i) \otimes_{\CO_{F}}R$ is contained in
        \begin{equation*}
            \im\bigl[ W(\Lambda_i)_{(-1)^s} \otimes_{\CO_{F}} R 
			\to W (\Lambda_i) \otimes_{\CO_{F}}R
			\bigr].
        \end{equation*}
\end{altenumerate}
This defines the spin condition when $r \neq s$.  When $r = s$, $W_{\pm 1}$ is defined over $F_0$ since \sform is already split before extending scalars $F_0 \to F$, and the spin condition on $M_{I,\CO_F}^\naive$ descends to $M_I^\naive$ over $\Spec \CO_{F_0}$.  In all cases, the \emph{spin local model $M_I^\spin$} is the closed subscheme of $M_I^\wedge$ where the spin condition is satisfied.
\begin{remark}
	Our definition of $a$ above is the same as in \cite{Smithling2011, Smithling2014}.  
	As noted in these papers, this agrees only up to sign with the analogous operators denoted $a_{f_1\wedge \dotsb \wedge f_{2n}}$ in \cite[Display 7.6]{PR2009}, and $a$ in \cite[\S 2.3]{Smithling2011}. There is a sign error in the statement of the spin condition in \cite[\S 7.2.1]{PR2009} leading to this discrepancy.
\end{remark}
		
We give an efficient way to calculate the sign $\sgn(\sigma_S)$ occurring in the expression \eqref{moduli_PR:W_+-1} for $W_{\pm 1}$, which will be used in \S \ref{Appendix_second-proof}.
\begin{lemma}[\cite{Smithling2015}, Lemma 2.4.1]\label{moduli_PR:sign-sigma-S}
\begin{equation*}
    \sgn(\sigma_S)=(-1)^{\sum S+\frac{n(n+1)}{2}}=(-1)^{\sum S+\lceil n/2\rceil}.\eqno\qed
\end{equation*}
\end{lemma}
		
\subsection{Strengthened spin condition}\label{moduli_ss}
		
We now formulate Smithling's strengthened spin condition \cite{Smithling2015}, which will be used to define the moduli space  $M_I$.  The idea is to restrict further the intersection in the definition of $W(\Lambda_i)_{\pm 1}$ to incorporate a version of the Kottwitz condition.  We continue with the notation in the previous subsections.

The operator $\pi \otimes 1$ acts $F$-linearly and semi-simply on $V=F^n\otimes_{F_0}F$ with eigenvalues $\pi$ and $-\pi$.  
Let $V_\pi$ and $ V_{-\pi}$ denote their respective eigenspaces.  Let
\begin{equation*}\label{moduli_ss:Wrs}
	W^{r,s} := \sideset{}{_F^r}\bigwedge V_{-\pi} \otimes_F \sideset{}{_F^s}\bigwedge V_\pi.
\end{equation*}
	Then $W^{r,s}$ is naturally a subspace of $W$, and we have the decomposition
\begin{equation*}
    W = \bigoplus_{r+s=n}W^{r,s}.
\end{equation*}
Let
\begin{equation*}\label{moduli_ss:Wrs-pm1}
	W_{\pm 1}^{r,s} := W^{r,s} \cap W_{\pm 1}.
\end{equation*}
For any $\CO_{F_0}$-lattice $\Lambda$ in $F^n$, let
\begin{equation}\label{moduli_ss:Wrs-pmlambda}
	W(\Lambda)_{\pm 1}^{r,s} := W_{\pm 1}^{r,s} \cap W(\Lambda) \subset W,
\end{equation}
where $W(\Lambda)$ is defined in \eqref{moduli_PR:lattice-in-W}.
Given any $\CO_{F}$-algebra $R$, we define
\begin{equation}\label{moduli_ss:L}
    L^{r,s}_{\pm 1}(\Lambda)(R):=\im\bigl[ W(\Lambda)_{\pm 1}^{r,s} \otimes_{\CO_F} R \to W(\Lambda) \otimes_{\CO_{F}}R\bigr].
\end{equation}
The \emph{strengthened spin condition} on an $R$-point $(\CF_i)\in M_I^{\naive}(R)$ is now defined as follows:
		
\begin{altenumerate}
	\item[\fontfamily{cmtt}\selectfont LM8.]\label{moduli_ss:strengthened-spin} \quad For all $i$, the line
 $\bigwedge_R^n \sF_i \subset W(\Lambda_i)\otimes_{\CO_{F}}R$ is contained in $L^{r,s}_{(-1)^s}(\Lambda_i)(R).$
\end{altenumerate}
		
This defines the condition in the case $r \neq s$.  When $r = s$, the condition is defined after base change to $M_{I,\CO_F}^{\naive}$. Since the subspaces $W^{r,s}$ and $W_{\pm 1}$ are Galois-stable, and the condition descends from $M_{I,\CO_F}^\naive$ to $M_I^\naive$.  In all cases, we write $M_I$ for the locus in $M_I^\spin$ where the condition is satisfied.
		
\begin{remark}\label{moduli_ss:non-split-form}
The spin and strengthened spin conditions can be defined for \emph{any} hermitian spaces, 
since we always take the base change $V=F^n\otimes_{F_0}F$ first, over which the hermitian and symmetry forms will split.   
\end{remark}

\begin{remark}\label{moduli_ss:imply-kott}
By \cite[Lemma 5.1]{Smithling2015}, the strengthened spin condition {\fontfamily{cmtt}\selectfont LM8} implies the Kottwitz condition {\fontfamily{cmtt}\selectfont LM5}. 
In Corollary \ref{general:wedge-from-ss}, we will further show that {\fontfamily{cmtt}\selectfont LM8} also implies the wedge condition {\fontfamily{cmtt}\selectfont LM6}.
\end{remark}

The following proposition will be useful in practice:
\begin{proposition}\label{moduli_ss:dual}
    The strengthened spin condition on $\CF_i\subset \Lambda_{i,R}$ is equivalent to the strengthened spin condition on $\CF_{n-i}\subset\Lambda_{n-i,R}$.
\end{proposition}
\begin{proof}
We denote by $(-)^*$ the dual space (resp. lattice) of $V$ (resp. $\Lambda\subset V$), and by $\Lambda^\perp$ the symmetric dual of $\Lambda$. The symmetric form $(\,,\,)$ induces an isomorphism $V^*\xrightarrow{\sim}V$ (resp. $\Lambda^*\xrightarrow{\sim} \Lambda^\perp$).

    Consider the short exact sequence:
\begin{equation*}
    0\rightarrow\CF_{n-i}\rightarrow\Lambda_{n-i,R}\rightarrow\Lambda_{n-i,R}/\CF_{n-i}\rightarrow 0.
\end{equation*}
Taking the $n$-th wedge product, we obtain the following identification:
\begin{equation}\label{moduli_ss:isom}
    \bigwedge^n\CF_{n-i}\cong\bigwedge^{2n}\Lambda_{n-i,R}\otimes \bigwedge^{n}(\Lambda_{n-i,R}/\CF_2)^*\xrightarrow{\sim}\bigwedge^{2n}\Lambda_{n-i,R}\otimes \bigwedge^n \CF_i,
\end{equation}
where the second isomorphism is induced by the isomorphism $\Lambda_{n-i}^*\xrightarrow{\sim}\Lambda_{n-i}^\perp$. 
This identification arises from the Hodge star operator. More precisely, given a lattice $\Lambda\subset F^n$, the base change $\Lambda_{\CO_F}\subset V$ determines an unordered basis of $V$. Choosing a generator $\omega\in\bigwedge^{2n}\Lambda_{\CO_F}=\CD$, we define the Hodge star operator
\begin{equation*}
    \star:\bigwedge^n V\xrightarrow{\sim}\bigwedge^n V^*\xrightarrow{\sim}
    \CD\otimes \bigwedge^n V,
\end{equation*}
where the first map is induced by the perfect pairing $\wedge^nV\times \wedge^n V\rightarrow \wedge^{2n}V$, and the second map is induced by the symmetric form $\sform$.
 
The identification \eqref{moduli_ss:isom} is constructed by first restricting the Hodge star operator $\star$ to the lattices $\wedge^n\Lambda\subset \wedge^n V$, then base changing the lattices from $\CO_F$ to $R$, and finally restricting to the corresponding filtrations.
Therefore, it suffices to show that the Hodge star operator, upon restriction, induces the following isomorphism:
\begin{equation}\label{moduli_ss:isom2}
    \star:  W^{r,s}_{\pm 1}\cap W(\Lambda)\xrightarrow{\sim} \CD\otimes\left(W^{r,s}_{\pm 1}\cap W(\Lambda^\perp)\right).
\end{equation}
If this holds, then by definition \eqref{moduli_ss:L}, we obtain a natural isomorphism:
\begin{equation*}
    L^{r,s}_{\pm}(\Lambda)(R)\simeq \bigwedge^{2n}\Lambda_{R}\otimes L^{r,s}_{\pm 1}(\Lambda^\perp)(R),
\end{equation*}
Combined with \eqref{moduli_ss:isom}, this completes the proof of the proposition.

To confirm the isomorphism \eqref{moduli_ss:isom2}, first note that the perfect pairing $\wedge^n\Lambda_{\CO_F}\times\wedge^n\Lambda_{\CO_F}\rightarrow\wedge^{2n}\Lambda_{\CO_F}$ induces an isomorphism
\begin{equation*}
    \star:W(\Lambda)\xrightarrow{\sim} W(\Lambda^*)\xrightarrow{\sim}\CD\otimes W(\Lambda^\perp).
\end{equation*}
Next, recall that $W^{r,s}=\bigwedge^r V_{-\pi}\otimes \bigwedge^s V_\pi$ is dual to $W^{s,r}$ with respect to both pairings. Therefore, we have $\star(W^{r,s})\subset\CD\otimes W^{r,s}$.
Finally, by \cite[\S 7.1.3]{PR2009}, the submodules $W_\pm$ are defined via the action of  $\CD$, and in particular, are stable under the $\star^{-1}$. Therefore, we have $\star(W_\pm)\subset \CD\otimes W_{\pm}$.
Combining these three observations, we obtain the desired isomorphism.
\end{proof}
		
\subsection{The local model, topological flatness and the coherence conjecture}\label{moduli_fc}
We recall the definition of the ramified unitary local model from \cite[\S 11]{PR2008}. 
We will later show that it coincides with the Pappas-Zhu model, cf.\ Proposition \ref{application_nonsplit:PZ-vs-loc}.
\begin{definition}
	The local model $M_I^\loc$ is defined to be the scheme-theoretic closure of the open embedding $M_{I,\eta}^\naive\hookrightarrow M_I^\naive$.
\end{definition}
By \cite[Proposition 14.14]{GW2020}, a scheme over a Dedekind domain is flat if and only if it agrees with the scheme-theoretic closure of its generic fiber. 
It follows that the local model $M^\loc_I$ is flat over $\CO_E$.     

There exists a closed embedding of the special fiber of the local model into an affine flag variety: 
\begin{equation}\label{equ:ebd-lm-to-flag}
M^\loc_{I,s}\hookrightarrow L\GU/L^+\CP_I^\circ,
\end{equation}
see \cite[\S 11]{PR2008}.
The coherence conjecture \cite[\S 4]{PR2009}, now a theorem due to Zhu \cite{Zhu2014}, predicts that the image of $M^\loc_{I,s}$ coincides with the union of Schubert cells indexed by the ``admissible set'', whose precise definition is not needed in this paper; cf.\ \cite[Theorem 11.3]{PR2008} and \cite[Theorem 8.1]{Zhu2014}.
By the theory of affine Schubert varieties in \cite{PR2008}, we have:
\begin{theorem}\label{moduli_fc:coh}
For $p\geq 3$, the special fiber $M^\loc_{I,s}$ of the local model is reduced, and each irreducible component is normal, Cohen-Macaulay, and 
compatibly Frobenius-split as a subvariety of the affine flag variety $L\GU/L^+\CP_I^\circ$.
\end{theorem}
To use the Frobenius-splitting property, we require that $p\neq 0$.
In \cite{PR2008}, the authors consider only the case where the group $G$ splits over a tamely ramified extension, which is why we assume $p\geq 3$. 
When $p=2$, the unitary group may not be tamely ramified, and the local model can fail to be reduced; cf.\ \cite[\S 2.6]{HPR}. We will prove the reducedness of the special fiber when $p=0$ in \S \ref{general_char-0}.
		
Since the wedge, spin (\cite[7.2.2]{PR2009}) and strengthened spin conditions (\cite[\S 2.5]{Smithling2015}) hold over the generic fiber $M^\naive_{I,\eta}$, we have a chain of closed embeddings, which are equalities when restricting to the generic fibers:
\begin{equation*}\label{moduli_fc:closed-immersions}
	M_I^\loc \subset M_I \subset M_I^\spin \subset M_I^\wedge \subset M_I^\naive.
\end{equation*}
Smithling \cite{Smithling2015} conjectured that $M_I^\loc(r,s) = M_I(r,s)$ for any $(r,s)$ and $I$. He showed that this is true topologically:
\begin{theorem}[\cite{Smithling2011,Smithling2014}]\label{moduli_fc:top-flat}
The scheme $M^\spin_I$ is topologically flat over $\CO_F$; in other words, the underlying topological spaces of $M^\spin_I$ and $M_I^\loc$ coincide.	
\end{theorem}
\begin{remark}\label{moduli_topflat:general}
In his paper, Smithling proved the topological flatness under the assumption that $F/F_0$ is a ramified extension of $p$-adic fields with $p\geq 3$. 
But his result also holds in the set up of \S \ref{moduli_setup_notation}: The key point of the proof is the comparison between different various notions of permissible sets and $\mu$-admissible sets, which is purely group-theoretic. 
Note that the Bruhat-Tits theory used in Smithling's (e.g., the Iwahori–Weyl group) remains valid in the context of \S \ref{moduli_setup_notation}.\end{remark}

The following reduction strategy is well-known among experts.
\begin{corollary}\label{moduli_fc:reduced}
The following are equivalent:
\begin{altenumerate}
\item The closed embedding $\iota:M_I^\loc\hookrightarrow M_I$ is an equality;
\item The special fiber $M_{I,s}$ is reduced.
\end{altenumerate}
\end{corollary}
\begin{proof}
By Theorem \ref{moduli_fc:coh} the special fiber $M^\loc_{I,s}$ is reduced.
If the embedding $\iota$ is an equality, then $M_{I,s}$ is also reduced.
Conversely, suppose that $M_{I,s}$ is reduced. 
Recall that for a locally noetherian scheme $X$ defined over a discrete valuation ring, if the special fiber $X_s$ is reduced, and all maximal points of $X_s$ lift to the scheme-theoretic closure of the generic fiber $X_\eta$, then $X$ is flat; cf.\ \cite[Proposition 14.17]{GW2020}.
Thus, it suffices to verify that all closed points of $M_{I,s}$ lift to $M_{I,s}^\loc$. This is a purely topological statement and follows from the topological flatness established in Theorem \ref{moduli_fc:top-flat}.
\end{proof} 
		
\section{Defining equations at the strongly non-special parahoric subgroup}\label{equ}
The main goal of this and the next section is to prove Theorem \ref{intro_back:max}. 
In this section, we determine the defining equations of the strengthened spin model $M_I$ for signature is $(n-1,1)$ and index set $I=\{\kappa\}$, assuming $\{\kappa\}$ is is strongly non-special.
Note that \S \ref{equ_chart} and \S \ref{equ_nw} apply to all except the $\pi$-modular one. The computations of the strengthened spin condition in \S \ref{equ_ss-comp} are valid for all maximal parahoric level.
		
\subsection{Affine chart of the worst point}\label{equ_chart}
Recall from \eqref{equ:ebd-lm-to-flag} that $M_{I,s}^\loc$ can be identified with a union of Schubert varieties in the affine flag varieties. The geometric special fibers of $M^\loc_I$ and $M_I$ topologically contain the same Schubert cells, including the unique closed Schubert cell, known as the ``worst point''.
\begin{lemma}\label{equ_chart:wp}
    The following $k$-point lies in $M_{\{\kappa\}}$:
\begin{equation*}
	\CF_\kappa=(\pi\otimes 1)(\Lambda_\kappa\otimes k)\subset \Lambda_\kappa\otimes k,\quad
	\CF_{n-\kappa}=(\pi\otimes 1)(\Lambda_{n-\kappa}\otimes k)\subset \Lambda_{n-\kappa}\otimes k.
\end{equation*}
This point is stabilized by the action of $P_{\{\kappa\}}$. Therefore, it represents the worst point.
\end{lemma}
\begin{proof}
Let 
\begin{equation*}
g_S:=\left(\bigwedge_{i=2}^n (e_i\otimes 1-\pi e_i\otimes \pi^{-1})\right) \wedge \frac{1}{2}(e_1\otimes 1+\pi e_1\otimes \pi^{-1}).
\end{equation*}
One can check that $2(-1)^{\kappa+1}\pi^{n-\kappa}g_S\in W_{-1}^{n-1,1}(\Lambda_\kappa)$, lifting the line $\bigwedge^n (\pi\otimes 1)(\Lambda_\kappa\otimes k)\subset W_{-1}^{n-1,1}(\Lambda_\kappa\otimes k)$; cf.\ Proposition \ref{equ_ss-comp:g_S} for explicit calculations. 
By definition in \S \ref{moduli_setup_alg-group}, this subspace is stable under the action  of $P_{\{\kappa\}}$. The same argument applies to $\CF_{n-\kappa}$.

Alternatively, the lemma can be verified using the wedge local model $M_{\{\kappa\}}^\wedge$, which is topologically flat in the strongly non-special case: for odd $n$, this follows directly from \cite{Smithling2011}; for even $n$, it follows from \cite[Proposition 7.4.7]{Smithling2014}, see also \cite[Proposition 3.4]{HLS-basic}.
\end{proof}

The worst point corresponds to the minimal element in both the admissible and permissible sets. As a result, any open neighborhood of the worst point intersects every Schubert cell. Therefore, to study the geometry of $M_{I,s}$, it suffices to restrict to an open affine neighborhood of the worst point.

Notice that $M_I=M_{\{\kappa\}}$ can be embed into $\Gr(n,\Lambda_\kappa\otimes_{\CO_{F_0}}\CO_F)$ as a closed subscheme: by periodicity conditions({\fontfamily{cmtt}\selectfont LM3}), all the algebraic relations in the strengthened spin models can be expressed in terms of $\CF_{\kappa}$ and $\CF_{n-\kappa}$;
the algebraic relations in $\CF_{n-\kappa}$ can be expressed in terms of $\CF_\kappa$ by orthogonality conditions ({\fontfamily{cmtt}\selectfont LM4}).
		
Recall that we have chosen the \emph{standard basis} in \S \ref{moduli_setup:lattice}:
\begin{align}
\begin{split}\label{equ_chart:standard-basis}
	\Lambda_{\kappa,\CO_F}:  &
	\pi^{-1}e_1\otimes 1,\cdots,\pi^{-1}e_{\kappa}\otimes 1,e_{\kappa+1}\otimes 1,\cdots ,e_n\otimes 1;\\ 
	& e_1\otimes 1,\cdots,e_{\kappa}\otimes 1,\pi e_{\kappa+1}\otimes 1,\cdots,\pi e_n\otimes 1.\\
	\Lambda_{n-\kappa,\CO_F}:&
	\pi^{-1}e_1\otimes 1,\cdots,\pi^{-1}e_{n-\kappa}\otimes 1,e_{n-\kappa+1}\otimes 1,\cdots , e_n\otimes 1;\\
    &e_1\otimes 1,\cdots,e_{n-\kappa}\otimes 1,\pi e_{n-\kappa+1}\otimes 1,\cdots,\pi e_n\otimes 1.
\end{split}
\end{align}
The worst point lies in the affine chart $U$ of the Grassmannian defined by the standard basis. 
We will translate local model axioms into linear algebra relations, and use them to obtain defining equations of the strengthened spin model in $U$.
We denote the intersection of $U$ with the moduli functors $M_I^\square$ by $U_I^\square:=M_I^\square\cap U$.
		
A point in the chart $U$ can be presented by a $2n\times n$ matrix $\left(\begin{matrix}\CX\\I_n\end{matrix}\right)$ with respect to the standard basis, and the worst point corresponds to $\CX=0$. 
We further subdivide the matrix into blocks\footnote{Not the same $\bA,\bB$ as in Theorem \ref{intro_back:max}, see \S \ref{max_red}}:
\begin{equation}\label{equ_chart:standard-X}
	\CX=
\begin{tikzpicture}[>=stealth,thick,baseline]
	\matrix [matrix of math nodes,left delimiter=(,right delimiter=)](A){ 
					D		&	M	&	C\\
					F	&	X_4		&	E\\
					B		&	L	&	A\\
				};
	\filldraw[purple] (2.2,0.5) circle (0pt) node [anchor=east]{\fontsize{7}{8} $\kappa$};
	\filldraw[purple] (2.6,0) circle (0pt) node [anchor=east]{\fontsize{7}{8} $n-2\kappa$};
	\filldraw[purple] (2.2,-0.5) circle (0pt) node [anchor=east]{\fontsize{7}{8} $\kappa$};
	\filldraw[purple] (-0.45,1.2) circle (0pt) node [anchor=east]{\fontsize{7}{8} $\kappa$};
	\filldraw[purple] (0.6,1.2) circle (0pt) node [anchor=east]{\fontsize{7}{8} $n-2\kappa$};
	\filldraw[purple] (0.7,1.2) circle (0pt) node [anchor=center]{\fontsize{7}{8} $\kappa$};
\end{tikzpicture}.
\end{equation}
Similarly, we can embed $M_{\{\kappa\}}$ into $\Gr(n,\Lambda_{n-\kappa}\otimes_{\CO_{F_0}}\CO_F)$.
The corresponding standard affine open in $\Gr(n,\Lambda_{n-\kappa}\otimes_{\CO_{F_0}}\CO_F)$ consists of points represented by $2n\times n$ matrices $\left(\begin{matrix}\CY\\ I_n\end{matrix}\right)$.

To simplify the computation, we also choose the following reordered basis,
\begin{align}\label{equ_chart:reordered-basis}
	\begin{split}
		\Lambda_{\kappa,\CO_F}:  &
		e_{n-\kappa+1}\otimes 1,\cdots, e_n\otimes 1,	
		\pi^{-1}e_1\otimes 1,\cdots,\pi^{-1}e_{\kappa}\otimes 1;	
		e_{\kappa+1}\otimes 1,\cdots, e_{n-\kappa}\otimes 1,\\
		&   
		\pi e_{n-\kappa+1}\otimes 1,\cdots, \pi e_n\otimes 1,
		e_1\otimes 1,\cdots,  e_{\kappa}\otimes 1; 	
		\pi e_{\kappa+1}\otimes 1,\cdots,\pi e_{n-\kappa}\otimes 1.\\
		\Lambda_{n-\kappa,\CO_F}:&
		e_{n-\kappa+1}\otimes 1,\cdots, e_n\otimes 1,	
		\pi^{-1}e_1\otimes 1,\cdots,\pi^{-1}e_{\kappa}\otimes 1;
		\pi^{-1} e_{\kappa+1}\otimes 1,\cdots,\pi^{-1} e_{n-\kappa}\otimes 1;\\
		&   
		\pi e_{n-\kappa+1}\otimes 1,\cdots, \pi e_n\otimes 1,
		e_1\otimes 1,\cdots,  e_{n}\otimes 1;	
		e_{\kappa+1}\otimes 1,\cdots, e_{n-\kappa}\otimes 1.
	\end{split}
\end{align}
The points in the open affine charts of $\Gr(n,\Lambda_\kappa\otimes_{\CO_{F_0}}\CO_F)$ and $\Gr(n,\Lambda_{n-\kappa}\otimes_{\CO_{F_0}}\CO_F)$ with respect to this basis can be represented by 
$\left(\begin{matrix}
	X\\I_n
\end{matrix}\right)$, resp.,\ $\left(\begin{matrix}
	Y\\I_n
\end{matrix}\right)$.
Since the two chosen bases differ by a permutation, the corresponding affine charts in $\Gr(n,\Lambda_\kappa\otimes_{\CO_{F_0}}\CO_F)$ are naturally identified.
The worst point corresponds to $X=0$, or equivalently, $Y=0$.
After reordering the basis, the partition \eqref{equ_chart:standard-X} now becomes
\begin{equation}\label{equ_chart:reordered-X}
	X=
	\left(\begin{array}{c|c}
		X_1&X_2\\
		\hline
		X_3&X_4
	\end{array}\right)=	
	\begin{tikzpicture}[>=stealth,thick,baseline]
		\matrix [matrix of math nodes,left delimiter=(,right delimiter=)](A){ 
					A		&	B	&	L\\
					C	&	D		&	M\\
					E		&	F	&	X_4\\
				};
		\filldraw[purple] (2.0,0.5) circle (0pt) node [anchor=east]{\fontsize{7}{8} $\kappa$};
		\filldraw[purple] (2.0,0) circle (0pt) node [anchor=east]{\fontsize{7}{8} $\kappa$};
		\filldraw[purple] (2.4,-0.5) circle (0pt) node [anchor=east]{\fontsize{7}{8} $n-2\kappa$};
		\filldraw[purple] (-0.4,1.1) circle (0pt) node [anchor=east]{\fontsize{7}{8} $\kappa$};
		\filldraw[purple] (0.15,1.1) circle (0pt) node [anchor=east]{\fontsize{7}{8} $\kappa$};
		\filldraw[purple] (0.65,1.1) circle (0pt) node [anchor=center]{\fontsize{7}{8}	$n-2\kappa$};
		\draw (-1,-0.225) -- (1,-0.225);
		\draw (0.25,-0.75) -- (0.25,0.75);
	\end{tikzpicture}.
\end{equation}
We subdivide the matrices $\CY$ and $Y$ into blocks of the same size as $\CX$ and $X$, respectively, with their entries specified in \eqref{equ_nw:Y}. 
For the remaining of the paper, we will mainly work with $\Gr(n,\Lambda_\kappa\otimes_{\CO_{F_0}}\CO_F)$.

\subsection{Naive model and wedge condition}\label{equ_nw}
Conditions {\fontfamily{cmtt}\selectfont LM1}-{\fontfamily{cmtt}\selectfont LM6} produce closed conditions on $U$. We will translate all of them, except {\fontfamily{cmtt}\selectfont LM5}, into algebraic relations in terms of $X$ (cf.\ Remark \ref{moduli_ss:imply-kott}).
With the chosen reordered basis, the transition maps:
$A_{\kappa}: \Lambda_{\kappa,R}\rightarrow \Lambda_{n-\kappa,R}$
and
$A_{n-\kappa}: \Lambda_{n-\kappa,R}\rightarrow \Lambda_{n+\kappa,R},$
in {\fontfamily{cmtt}\selectfont LM2} are represented by the matrices
\begin{equation}\label{equ_nw:trans}
	A_{\kappa}=\left(\begin{matrix}
	I_{2\kappa}    &          &          &   \\
	&          &          &   \pi_0\cdot I_{n-2\kappa}\\
	&          &   I_{2\kappa}&   \\
	&   I_{n-2\kappa}  &          &   
\end{matrix}\right),
\quad
A_{n-\kappa}=\left(
\begin{matrix}
	&          &   \pi_0 \cdot I_{2\kappa}      &   \\
	&   I_{n-2\kappa}  &          &   \\
	I_{2\kappa}    &          &          &   \\
	&          &          &I_{n-2\kappa}
\end{matrix}\right).
\end{equation}
The symmetric pairing $\sform_R:\Lambda_{n-\kappa,R}\times \Lambda_{\kappa,R}\rightarrow R$ in {\fontfamily{cmtt}\selectfont LM4} is represented by the matrix
\begin{equation}\label{equ_nw:M}
	M=\left(\begin{matrix}
	&      &J_{2\kappa}&  \\
	&      &       &-H_{n-2\kappa}\\
	-J_{2\kappa}   &      &       &\\
	&H_{n-2\kappa} &       &
    \end{matrix}\right),
\end{equation}
where 
\begin{equation*}
	H_{l}=
    \left(
    \begin{array}{ccc}
	&        &   1\\
	&\iddots &   \\
	1   &        & 
    \end{array}\right),
	\quad
	J_{2l}=
	\left(\begin{matrix}
		&   H_l\\
		-H_l &
		\end{matrix}\right).
\end{equation*}
We will omit the lower indices of these matrices for simplicity.
By {\fontfamily{cmtt}\selectfont LM4}, we have $\CF_{n-\kappa}=\CF_{\kappa}^\perp$. Using the represented matrix \eqref{equ_nw:M}, we have
\begin{equation}\label{equ_nw:Y}
	Y=
	\left(\begin{matrix}
	Y_1 &   Y_2 \\
	Y_3 &   Y_4 \\
\end{matrix}\right)=
\left(\begin{matrix}
	-J_{2\kappa}   &   \\
	&   H_{n-2\kappa}  
\end{matrix}\right)
    X^t
    \left(\begin{matrix}
	J_{2\kappa}    &       \\
	&   H_{n-2\kappa} 
	\end{matrix}\right)
	=
	\left(\begin{matrix}
	-JX_1^t J &   -JX_3^t H  \\
	H X_2^t J   &   HX_4^t H
		\end{matrix}\right).
\end{equation}

Recall from {\fontfamily{cmtt}\selectfont LM1} the $\pi\otimes 1$-stability:
$(\pi\otimes 1)\CF_{\kappa}\subset\CF_{\kappa}$,
and
$(\pi\otimes 1)\CF_{n-\kappa}\subset\CF_{n-\kappa}$.
For the first inclusion, we obtain
$
(\pi\otimes 1) \left(\begin{matrix} X\\I_n \end{matrix}\right)
=\left(\begin{matrix} \pi_0 I_n\\X \end{matrix}\right)
=\left(\begin{matrix} X\\I_n \end{matrix}\right)T,	
$
for some $T$. Hence $T=X$ and $XT=\pi_0 I_n$. This implies $X^2=\pi_0 I_0$, i.e.,
\begin{equation*}
	\left(\begin{matrix}
		X_1^2+X_2X_3    &   X_1X_2+X_2X_4\\
		X_3X_1+X_4X_3   &   X_3X_2+X_4^2
		\end{matrix}\right)
	=
	\left(\begin{matrix}
		\pi_0 I_{2\kappa}    &   \\
		&\pi_0 I_{n-2\kappa}
	\end{matrix}\right).	
\end{equation*}
	The $(\pi\otimes 1)$-stability on $\CF_{n-\kappa}$ follows from the $(\pi\otimes 1)$-stability of $\CF_{\kappa}$ by \eqref{equ_nw:Y}.
			
The condition {\fontfamily{cmtt}\selectfont LM2} requires that the transition map $A_\kappa$ sends $\CF_\kappa$ into $\CF_{n-\kappa}$. This is equivalent to $(A_\kappa \CF_\kappa,\CF_\kappa)=0$, i.e., $\left(\begin{matrix} X\\I_n \end{matrix}\right)^t A_{\kappa}^t M \left(\begin{matrix} X\\I_n \end{matrix}\right)=0.$
This implies
\begin{equation*}
	\left(\begin{matrix}
		-JX_1+X_3^t H X_3 + X_1^t J    &   -J X_2+X_3^t H X_4\\
		X_2^t J+X_4^t H X_3  &   X_4^t H X_4 -\pi_0 H
	\end{matrix}\right)
	=0.
\end{equation*}
Similarly, the condition that the transition map $A_{n-\kappa}$ sends $\CF_{n-\kappa}$ to $\CF_{n+\kappa}$ is equivalent to $(\CF_\kappa,A_{n-\kappa} \CF_{n-\kappa})=0$. This gives 
\begin{equation*}
	\left(\begin{matrix}
		X_1 J X_1^t -\pi_0 J  &   X_1 J X_3^t-X_2 H\\
		X_3 J X_1^t + HX_2^t    &   X_3 JX_3^t-X_4 H+HX_4^t
	\end{matrix}\right)=0.	
\end{equation*}
Here we have secretly identified $\Lambda_{n+\kappa}$ with $\Lambda_\kappa$ by {\fontfamily{cmtt}\selectfont LM3}.

Finally, the wedge condition {\fontfamily{cmtt}\selectfont LM6} for $(r,s)=(n-1,1)$ implies:
\begin{equation*}
    \begin{matrix}
		\bigwedge^2 (\pi\otimes 1+1\otimes \pi\mid \CF_\kappa)=0,	&	\bigwedge^n (\pi\otimes 1-1\otimes \pi\mid \CF_\kappa)=0;\\
		\bigwedge^2 (\pi\otimes 1+1\otimes \pi\mid \CF_{n-\kappa})=0,	&	\bigwedge^n (\pi\otimes 1-1\otimes \pi\mid \CF_{n-\kappa})=0.
	\end{matrix}
\end{equation*}
The wedge condition on $\CF_{n-\kappa}$ is implied by that on $\CF_\kappa$ via the relation \eqref{equ_nw:Y}.
In the special fiber, where $1\otimes \pi=0$, the axiom {\fontfamily{cmtt}\selectfont LM6} simplifies to
\begin{equation*}
	\bigwedge^2 X=0.
\end{equation*}		
We summarize what we have obtained over the special fiber:
\begin{proposition}\label{equ_nw:summary}
    The coordinate ring of the affine chart over the special fiber $U^\wedge_{\{\kappa\},s}\subset M^\wedge_{\{\kappa\},s}$ is isomorphic to the polynomial ring $k[X]$ modulo the entries of the following matrices:
\begin{altitemize}
    \item[\fontfamily{cmtt}\selectfont LM1.] $X_1^2+X_2X_3,
    \,X_1X_2+X_2X_4,
    \,X_3X_1+X_4X_3,
    \,X_3X_2+X_4^2$,
				
    \item[\fontfamily{cmtt}\selectfont LM2-1.] $-JX_1+X_3^t HX_3+X_1^t J,-JX_2+X_3^tHX_4,X_2^tJ+x_4^tHX_3,X_4^tHX_4$,
				
    \item[\fontfamily{cmtt}\selectfont LM2-2.] $X_1JX_1^t,X_1JX_3^t-X_2H,
    \,X_3JX_1^t+HX_2^t,
    \,X_3JX_3^t-X_4H+HX_4^t$,

    \item[\fontfamily{cmtt}\selectfont LM5.] $\det(tI_n-X)=(t-\pi)^{n-1}(t+\pi)$.
    
    \item[\fontfamily{cmtt}\selectfont LM6.] $\bigwedge^2X$.
\end{altitemize} 
\end{proposition}

One see that the Kottwitz condition is more complicated than the other axioms. However, thanks to Remark \ref{moduli_ss:imply-kott}, it can be omitted in the strengthened spin model.
  
\subsection{Strengthened spin condition: set-ups}\label{equ_ss-setup}	
Next, we want to translate the strengthened spin conditions into linear algebra relations. 
While our approach is inspired by the discussions in \cite[\S 4.1]{Smithling2015} and \cite[\S 1.4.2]{Yu2019}, our setting involves additional complications. In this subsection, we first present the basic definitions and then outline our strategy.
		
\subsubsection{Another basis for $V$}
We choose another basis for $V=F^n\otimes_{F_0}F$. Let $g_1,\cdots,g_{2n}$ be the following ordered $F$-basis for $V$,
\begin{equation}
	e_1\otimes 1-\pi e_1\otimes \pi^{-1},\cdots,e_n\otimes 1-\pi e_n\otimes \pi^{-1};\frac{e_1\otimes 1+\pi e_1\otimes \pi^{-1}}{2},\cdots,\frac{e_n\otimes 1+\pi e_n\otimes \pi^{-1}}{2}.	
\end{equation}
It splits and separates the eigenspaces $V_{-\pi}$ and $V_\pi$.
This is also a split orthogonal ordered basis for the symmetric form $(\,,\,)$, whose transformation matrix to \eqref{moduli_PR:odd-basis} has determinant $1$. 
Hence, we find a basis for the eigenspaces:
\begin{equation*}\label{equ_ss-setup:decomp}
	W_{\pm 1}=
	\Span_F\{g_S\pm \sgn(\sigma_S)g_{S^\perp}\mid \#S=n\},
\end{equation*}
where $g_S\in W$ is defined to be the wedge product with respect to $g_i$ for the ordered index set $S$.

\subsubsection{Types and weights}
We recall some definitions and properties in \cite[\S 4.2, \S 4.3]{Smithling2015}.
		
\begin{definition} 
We say that a subset $S \subset \{1,\cdots,2n\}$ has \emph{type $(r,s)$} if
\[
\#(S\cap \{1,\cdots,n\}) = r
\quad\text{and}\quad
\#(S\cap\{n+1,\cdots,2n\}) = s.
\]
Suppose from now on, $r + s = n$, then the elements $g_S$ defined in \S \ref{equ_ss-setup} for varying $S$ of type $(r,s)$ form a basis for $W^{r,s}$.
\end{definition}
It is easy to verify that $S$ and $S^\perp$ are of the same type, where $S^{\perp}$ is defined in \S \ref{moduli_setup_notation}. We have
\begin{lemma}[\protect{\cite[Lemma 4.2]{Smithling2015}}]\label{equ_ss-setup:span}
	$W_{\pm 1}^{r,s} = \Span_F\{\, g_S \pm \sgn(\sigma_S) g_{S^\perp} \mid S \text{ is of type } (r,s)\,\}$.\qed
\end{lemma}

\begin{remark}\label{equ_ss-setup:sigma-compute}
	We will be interested in $S$ of type $(n-1,1)$. Such an $S$ is of the form
	\[
	S = \left\{1,\cdots, \wh{j}, \cdots, n, n+i\right\}
	\]
	for some $i$, $j \leq n$.  By Lemma \ref{moduli_PR:sign-sigma-S}, we have
	\[
	\sgn(\sigma_S) = (-1)^{\Sigma S+\frac{n(n+1)}{2}} = (-1)^{\frac{n(n+1)}{2} + \frac{n(n+1)} 2 -j + n+i} = (-1)^{n+ i + j }.
	\]
\end{remark}
		
Next, to find a basis for $W(\Lambda_\kappa)_{-1}^{n-1,1}$, we need to determine when a linear combination of elements of the form $g_S-\sgn(\sigma_S)g_{S^\perp}$ lies in $W(\Lambda_\kappa)$. The following definition will help with the bookkeeping.
\begin{definition}
	Let $S \subset \{1,\cdots,2n\}$.  
	The \emph{weight vector $\mathbf{w}_S$} attached to $S$ is the element in $\BN^n$ whose $i$th entry is $\#(S \cap \{i,n+i\})$.
\end{definition}

\begin{example}\label{equ_ss-setup:possible-weights}
If $S$ is of type $(n-1,1)$, then its weight vector $\mathbf{w}_S$ falls into one of two cases.

In the first case, there exist distinct indices $i\neq j$ such that the $i$th entry of $\mathbf{w}_S$ is $2$, the $j$th entry of $\mathbf{w}_S$ is $0$, and all the other entries are $1$. In this situation,  $S = \{1,\cdots, \wh{j},\cdots, n, n+i\}$ can be uniquely recovered from its weight vector.

In the second case, the weight vector is $\mathbf{w}_S = (1,\cdots,1)$. Then all we can say is that $S = \{1,\cdots, \wh{i}, \cdots, n, n+ i\}$ for some $i \in \{1,\cdots, n\}$.
\end{example}
The reason we introduce the weight vector is the following lemma:
\begin{lemma}[\protect{\cite[Lemma 4.7]{Smithling2015}}]\label{equ_ss-setup:weight-compute}
    For any subset $S\subset \{1,\cdots,2n\}$ of cardinality $n$, write
	\[
	g_S=\sum_{S'}c_{S'}e_{S'},\quad c_{S'}\in F.
	\]
Then, every $S'$ with $c_{S'}\neq 0$ has the same weight as $S$. 
\end{lemma}
		
\subsubsection{The worst terms}		
Let $A$ be a finite dimensional $F$-vector space, and let $\CB$ be an $F$-basis of $A$.
		
\begin{definition}\label{equ_ss-setup:wt-def}
	Let $x=\sum_{b\in \CB}c_b b\in A$, with $c_b\in F$. 
	We say that $c_b b$ is a \emph{worst term} of $x$ if it is nonzero, and satisfies
	\begin{equation*}
		\ord_\pi(c_b)\leq \ord_\pi (c_{b'})\text{ for all }b'\in \CB \text{ such that } c_{b'}\neq 0.
	\end{equation*}
	We then define the \emph{worst term part} of $x$ as
	\begin{equation*}
		\WT_\CB(x)=\WT(x)=\sum_{\substack{b \in \CB\\ c_b b \text{ is a worst}\\ \text{term for } x}}c_b b.
	\end{equation*}
\end{definition}
		
Let $\Lambda$ be the $\CO_F$-span of $\CB$ in $A$. A nonzero element $x\in A$ lies in $\Lambda$ if and only if one (hence all) of its worst term lies in $\Lambda$.
When this is the case, and $R$ is a $k$-algebra, the image of $x$ under the natural map
\begin{equation*}
	\Lambda\rightarrow \Lambda\otimes_{\CO_F}R
\end{equation*} 
coincides with the image of $\WT_\CB(x)$.

\subsubsection{Strategy for computing the strengthened spin condition}
By Proposition \ref{moduli_ss:dual}, it suffices to compute the strengthened spin condition on $\CF_{\kappa}$. The procedure consists of the following three steps:
\begin{altenumerate}
	\item Expressing each $g_S$ in terms of the basis elements $e_{S'}$,
	\item Computing $g_S\pm\sgn(\sigma_S)g_{S^\perp}$ and their linear combinations in terms of $e_{S'}$,
	\item Translating these expressions into equations that reflect the strengthened spin condition.
\end{altenumerate}
Working out the algebraic relations for the strengthened spin condition over  $\CO_F$ is more subtle, since taking the $n$th wedge product introduces many terms with higher valuation.
This is why we focus on the special fiber and make use of the notion of the ``worst term''.
		
However, since our final goal is to compute the worst term of the sum $\sum_S g_S-\sgn(\sigma_S)g_{S^\perp}$, as in Corollary \ref{equ_ss-comp:basis}, it is not sufficient to determine only the worst term of each $g_S$.

Therefore, we also need to compute certain higher order terms of $g_S$.
In practice, we find that including the ``second-worst'' terms is sufficient.

Once the defining equations over the special fiber are obtained, similar computations allow us to lift them to defining equations over the ring of integers. This will be carried out in Section \ref{intequ}.

\subsection{Strengthened spin condition: computation}\label{equ_ss-comp}
The next proposition computes the worst term of $g_S$, similar to \cite[Lemma 4.9]{Smithling2015} and \cite[Lemma 1]{Yu2019}.
To simplify notation and save space, we use the following non-standard notation:
\begin{equation}\label{equ_ss-comp:not}
    e_{[i,\wh{n+j}]}:=e_{\{i,n+1,\cdots,\wh{n+j},\cdots,2n\}}=e_i\wedge e_{n+1}\wedge\cdots\wedge\wh{e_{n+j}}\wedge\cdots \wedge e_{2n}.
\end{equation}
\begin{proposition}\label{equ_ss-comp:g_S}
For any $S\subset \{1,\cdots,2n\}$ of type $(n-1,1)$, we have:
\begin{altenumerate}
\item When $S=\{1,\cdots,\hat{i},\cdots,n,n+i\}$ for some $1\leq i\leq \kappa$,
\begin{equation*}
g_S=\frac{1}{2}(-1)^{\kappa+i}\pi^{-(n-\kappa)}\Bigl[
e_{\{n+1,\cdots,2n\}}+\pi\Bigl( 2(-1)^{i+1}e_{[i,\wh{n+i}]}+\sum_{\sigma=1}^n(-1)^\sigma e_{[\sigma,\wh{n+\sigma}]}\Bigr)\Bigr]+o(\pi^{-(n-\kappa-1)}).
\end{equation*}

\item When $S=\{1,\cdots,\hat{i},\cdots,n,n+i\}$ for some $\kappa+1\leq i\leq n$,
\begin{equation*}
g_S=\frac{1}{2}(-1)^{\kappa+i+1}\pi^{-(n-k)}\Bigl[
e_{\{n+1,\cdots,2n\}}+\pi   
\Bigl(2(-1)^{i+1} e_{[i,\wh{n+i}]}+\sum_{\sigma=1}^n(-1)^\sigma e_{[\sigma,\wh{n+\sigma}]}\Bigr)\Bigr]+o(\pi^{-(n-\kappa-1)})
\end{equation*}
		
\item When $S=\{1,\cdots,\hat{j},\cdots,n,n+i\}$ for some $i,j\leq \kappa, i\neq j$,
		\begin{equation*}
			g_S=(-1)^{\kappa+1}\pi^{-(n-\kappa-1)} e_{[i,\wh{n+j}]}+o(\pi^{-(n-\kappa-1)}).
		\end{equation*}
\item When $S=\{1,\cdots,\hat{j},\cdots,n,n+i\}$ for some $i\leq \kappa<j$,
		\begin{equation*}
			g_S=(-1)^{\kappa}\pi^{-(n-\kappa-2)} e_{[i,\wh{n+j}]}+o(\pi^{-(n-\kappa-2)}).
		\end{equation*}
\item When $S=\{1,\cdots,\hat{j},\cdots,n,n+i\}$ for some $j\leq \kappa<i$,
		\begin{equation*}
			g_S=(-1)^{\kappa+1}\pi^{-(n-\kappa)} e_{[i,\wh{n+j}]}+o(\pi^{-(n-\kappa)}).
		\end{equation*}
\item When $S=\{1,\cdots,\hat{j},\cdots,n,n+i\}$ for some $\kappa+1\leq i,j\leq n,i\neq j$,
		\begin{equation*}
			g_S=(-1)^{\kappa}\pi^{-(n-\kappa-1)} e_{[i,\wh{n+j}]}+o(\pi^{-(n-\kappa-1)}).
		\end{equation*}
		\end{altenumerate}
Here $o(\pi^k)$ denotes terms of valuation strictly greater than $k$.
\end{proposition}
The proof of Proposition \ref{equ_ss-comp:g_S} is lengthy and tedious, and contributes little to the understanding of the later sections. Therefore, we postpone it to \S \ref{Appendix_first-proof}.

Note that $(S^\perp)^\perp=S$, and $g_S-\sgn(\sigma_S)g_{S^\perp}=\pm (g_{S^\perp}-\sgn(\sigma_{S^\perp})g_S)$. This symmetry motivates the following definition:
\begin{definition}
Let $S$ is of type $(n-1,1)$, so that $S\cap\{n+1,\cdots,n\}$ consists of a single element, denoted $i_S$.
We say that $S$ \emph{balanced} if $i_S\leq i_{S^\perp}$.
\end{definition}

This notion allows us to describe a basis of $W_{-1}^{n-1,1}$:
\begin{proposition}[\protect{\cite[Proposition 4.5]{Smithling2015}}]\label{equ_ss-comp:g_S-g-S-perp}
	The elements $g_S-\sgn(\sigma_S)g_{S^\perp}$ for balanced $S$ form a basis for $W^{n-1,1}_{-1}.$\qed
\end{proposition}

Now we come to the main result of this subsection, similar to \cite[Lemma 4.10]{Smithling2015} and \cite[Lemma 2]{Yu2019}.
\begin{proposition}\label{equ_ss-comp:g_S-dual}
	For a balanced subset $S=\{1,\cdots,\hat{j},\cdots,n,n+i\}$, we distinguish the following cases. In each cases, the term $g_S$ is computed as follows:
	
\noindent$\bullet$ \textbf{Case 1.} For $S$ satisfying $S=S^\perp$ and $i\neq j$.
	\begin{altenumerate}
	\item When $i\leq \kappa$, we have
	\[
	g_S-\sgn(\sigma_S)g_{S^\perp}
	=2(-1)^\kappa \pi^{-(n-\kappa-2)} e_{[i,\wh{n+j}]}+o(\pi^{-(n-\kappa-2)}).
	\]
				
	\item When $\kappa<i\leq n-\kappa,i\neq j$, we have
	\[
	g_S-\sgn(\sigma_S)g_{S^\perp}
	=2(-1)^{\kappa+1} \pi^{-(n-\kappa-1)} e_{[i,\wh{n+j}]}+o(\pi^{-(n-\kappa-1)}).
	\]
				
	\item When $i>n-\kappa$, we have 
	\[
	g_S-\sgn(\sigma_S)g_{S^\perp}
	=2(-1)^{\kappa+1} \pi^{-(n-\kappa)} e_{[i,\wh{n+j}]}+o(\pi^{-(n-\kappa)}).
	\]
\noindent $\bullet$ \textbf{Case 2.} For $S$ satisfying $i=j$.
						
	\item When $i=j\leq \kappa$, we have
	\[
	g_S-\sgn(\sigma_S)g_{S^\perp}
	=(-1)^{\kappa+1}\pi^{-(n-\kappa-1)}
	\Bigl( e_{[i,\wh{n+i}]}
	+(-1)^n e_{[i,\wh{n+i^\vee}]}\Bigr)+o(\pi^{-(n-\kappa-1)}).
	\]
				
	\item When $\kappa<i=j\leq m$, we have
	\begin{align*}
		&g_{S_i}-\sgn(\sigma_{S_i})g_{S_i^\perp}=(-1)^{\kappa+i+1}\pi^{-(n-\kappa)}\\
		&\qquad\Bigl[e_{\{n+1,\cdots,2n\}}+\pi\Bigl(
		(-1)^{i+1} e_{[i,\wh{n+i}]}
		+(-1)^{i^\vee+1} e_{[i^\vee,\wh{n+i^\vee}]}
		+\sum_{\sigma=1}^n(-1)^\sigma  e_{[\sigma,\wh{n+\sigma}]}\Bigr)\Bigr]+o(\pi^{-(n-\kappa-1)}).
	\end{align*}
				
	\item When $i=j=m+1$ (this will only happen when $n$ is odd), then
	\begin{align*}
		&g_S-\sgn(\sigma_S)g_{s^\perp}=(-1)^{\kappa+m}\pi^{-(n-\kappa)}\\
		&\qquad\Bigl[e_{\{n+1,\cdots,2n\}}+\pi\Bigl(-2(-1)^{m+1}e_{[m+1,\wh{n+m+1}]}+
		\sum_{\sigma=1}^n (-1)^\sigma e_{[\sigma,\wh{n+\sigma}]}\Bigr)\Bigr]+o(\pi^{-(n-\kappa-1)}).
	\end{align*}
\noindent$\bullet$ \textbf{Case 3.} For $S$ satisfying $S \neq S^\perp$ and $i\neq j$.
				
	\item When $i<j^\vee\leq \kappa$, we have
 \begin{equation*}
     g_S-\sgn(\sigma_S)g_{s^\perp}=(-1)^\kappa\pi^{-(n-\kappa-2)}
	\left(
	e_{[i,\wh{n+j}]}-(-1)^{n+i+j}e_{[j^\vee,\wh{n+i^\vee}]}
	\right)+o(\pi^{-(n-\kappa-2)}).
 \end{equation*}
							
	\item When $i\leq \kappa<j^\vee<n-\kappa+1$, we have
\begin{equation*}
    g_S-\sgn(\sigma_S)g_{S^\perp}=(-1)^{n+k+1+i+j}\pi^{-(n-k-1)}e_{[j^\vee,\wh{n+i^\vee}]}+o(\pi^{-(n-k-1)}).
\end{equation*}
				
	\item When $i\leq \kappa,j^\vee\geq n-\kappa+1$, we have
	\[
	g_S-\sgn(\sigma_S)g_{S^\perp}=(-1)^{\kappa+1}\pi^{-(n-\kappa-1)}
	\left(
	e_{[i,\wh{n+j}]}+(-1)^{n+i+j}e_{[j^\vee,\wh{n+i^\vee}]}
	\right)+o(\pi^{-(n-\kappa-1)}).
	\]
				
	\item When $\kappa<i<j^\vee<n-\kappa+1$, we have
	\[
	g_S-\sgn(\sigma_S)g_{S^\perp}=
	(-1)^\kappa\pi^{-(n-\kappa-1)}
	\left( 
		e_{[i,\wh{n+j}]}-(-1)^{n+i+j}e_{[j^\vee,\wh{n+i^\vee}]}
		\right)+o(\pi^{-(n-\kappa-1)}).
	\]
				
	\item When $\kappa<i<n-\kappa+1\leq j^\vee$, we have
	\[
		g_S-\sgn(\sigma_S)g_{S^\perp}=(-1)^{\kappa+1}\pi^{-(n-\kappa)}e_{[i,\wh{n+j}]}+o(\pi^{-(n-\kappa)}).
	\]
				
	\item When $n-\kappa+1\leq i<j^\vee$, we have
	\[
		g_S-\sgn(\sigma_S)g_{S^\perp}=(-1)^{\kappa+1}\pi^{-(n-\kappa)}\left(
		e_{[i,\wh{n+j}]}-(-1)^{n+i+j}e_{[j^\vee,\wh{n+i^\vee}]}
		\right)+o(\pi^{-(n-\kappa)}).
	\]
	\end{altenumerate}
\end{proposition}
We postpone the proof to \S \ref{Appendix_second-proof}.
		
\subsubsection{The lattice $W(\Lambda_\kappa)_{-1}^{n-1,1}$}
Next, we will write the basis of the lattice $W(\Lambda_\kappa)_{-1}^{n-1,1}$ in terms of $e_S$, using Proposition \ref{equ_ss-comp:g_S-dual}.
As mentioned in Example \ref{equ_ss-setup:possible-weights}, when the weight vector $\bw_S\neq (1,1,\cdots,1)$, the set $S$ is uniquely determined by $\bw_S$. 
The most delicate case arises when the weight vector is $(1,1,\cdots,1)$ and we introduce additional notation to handle this situation.
We denote $S_i:=\{1,\cdots,\wh{i},\cdots,n,n+i\}$, and write $a_i$ for $a_{S_i}$. Define
\begin{equation*}
	M:=\lfloor \frac{n+1}{2}\rfloor=\left\{\begin{array}{cc} m+1 & \mbox{ when }n\mbox{ is odd};\\ m & \mbox{ when }n\mbox{ is even}.\end{array}\right.
\end{equation*}
Then $S_1,S_2,\cdots, S_M$ are precisely the balanced subsets with weight vector $(1,1,\cdots,1)$.
		
The next proposition is parallel to \cite[Proposition 4.12]{Smithling2015} and \cite[Proposition 2]{Yu2019}.
		
\begin{proposition}\label{equ_ss-comp:lattice}
	Let $w \in W_{-1}^{n-1,1}$, and write
\begin{equation}\label{equ_ss-comp:lattice-summation}
    w = \sum_{S \text{ \emph{balanced}}}a_S\bigl(g_S - \sgn(\sigma_S) g_{S^\perp}\bigr),
			\quad a_S \in F.
\end{equation}
Then $w$ lies in $W(\Lambda_\kappa)_{-1}^{n-1,1}$ if and only if it satisfies the following:
\begin{altenumerate}			
	\item
	if $S=\{1,\cdots,\wh{i^\vee},\cdots,n,n+i\}$ for some $i\leq \kappa$ then $\ord_\pi(a_S)\geq n-\kappa-2$;
	\item
	if $S=\{1,\cdots,\wh{i^\vee},\cdots,n,n+i\}$ for some $k<i\leq n-\kappa,i\neq i^\vee$, then $\ord_\pi(a_S)\geq n-\kappa-1$;
	\item
	if $S=\{1,\cdots,\wh{i^\vee},\cdots,n,n+i\}$ for some $n-\kappa<i$, then $\ord_\pi(a_S)\geq n-\kappa$;
	\item
	if $S=\{1,\cdots,\wh{j},\cdots,n,n+i\}$ for some $i<j^\vee\leq \kappa$, then $\ord_\pi(a_S)\geq n-\kappa-2$;
	\item
	if $S=\{1,\cdots,\wh{j},\cdots,n,n+i\}$ for some $i\leq \kappa\leq j^\vee<n-\kappa+1$, then $\ord_\pi(a_S)\geq n-\kappa-1$;
	\item
	if $S=\{1,\cdots,\wh{j},\cdots,n,n+i\}$ for some $i\leq \kappa,j^\vee\geq n-\kappa+1$, then $\ord_\pi(a_S)\geq n-\kappa-1$;
	\item
	if $S=\{1,\cdots,\wh{j},\cdots,n,n+i\}$ for some $\kappa<i<j^\vee< n-\kappa+1$, then $\ord_\pi(a_S)\geq n-\kappa-1$;
	\item
	if $S=\{1,\cdots,\wh{j},\cdots,n,n+i\}$ for some $k<i<n-\kappa+1\neq j^\vee$, then $\ord_\pi(a_S)\geq n-\kappa$;
	\item
	if $S=\{1,\cdots,\wh{j},\cdots,n,n+i\}$ for some $n-\kappa+1\leq i<j^\vee$, then $\ord_\pi(a_S)\geq n-\kappa$;
	\item
	if $S=\{1,\cdots,\wh{i},\cdots,n,n+i\}$ for some $i\leq \kappa$, then $\ord_\pi(a_S)\geq n-\kappa$;
	\item
	if $\sum_{i=\kappa+1}^M(-1)^ia_i= 0$, then $\ord_\pi(a_i)\geq n-\kappa-1$ for those $i$;
	\item
	if $\sum_{i=\kappa+1}^M(-1)^ia_i\neq 0$, and
    \begin{altenumerate2}
        \item if $\ord(\sum_{i=\kappa+1}^M(-1)^ia_i)={\displaystyle\min_{\kappa+1\leq i\leq M}\{\ord(a_i)\}}$, then
        ${\displaystyle\min_{\kappa+1\leq i\leq M}\{\ord(a_i)\}}\geq n-\kappa$;
        \item if $\ord(\sum_{i=\kappa+1}^M(-1)^ia_i)\geq{\displaystyle\min_{\kappa+1\leq i\leq M}+1}$, then
    ${\displaystyle\min_{\kappa+1\leq i\leq M}\{\ord(a_i)\}}\geq n-\kappa-1$.
    \end{altenumerate2}
\end{altenumerate}
\end{proposition}
\begin{proof}
By Lemma \ref{equ_ss-setup:weight-compute}, an element $w$ lies in $W(\Lambda_\kappa)_{-1}^{n-1,1}$ if and only if, for each possible weight vector $\mathbf{w}$ (as classified in Example \ref{equ_ss-setup:possible-weights}), we have
\begin{equation*}
	w_{\bw}:=\sum_{\substack{S \text{ balanced}\\ \text{with weight }\mathbf{w}}}a_S\bigl(g_S - \sgn(\sigma_S) g_{S^\perp}\bigr) \in W(\Lambda_\kappa)_{-1}^{n-1,1}.
\end{equation*}
Thus, we reduce to analyzing the partial summation $w_\bw$ of a weight vector, that is, the $\bw$-component of the total sum \eqref{equ_ss-comp:lattice-summation} in $W(\Lambda_\kappa)_{-1}^{n-1,1}$.
			
When $\mathbf{w}\neq (1,1,\cdots,1)$, Example \ref{equ_ss-setup:possible-weights} shows that there is at most one subset $S$ of type $(n-1,1)$ with weight $\mathbf{w}$. These cases correspond to case 1: (i)-(iii) and case 3: (vii)-(xii) in Proposition \ref{equ_ss-comp:g_S-dual}. This establishes (i)-(ix).
			
When $\mathbf{w} = (1,1,\cdots,1)$, the terms corresponding to $S_1,\cdots, S_M$ appear in the summation. Write
\begin{equation*}
	w_{(1,1,\cdots,1)}=\sum_{i=1}^M a_i(g_{S_i}-\sgn(S_i)g_{S_i^\perp})=\sum_{i=1}^\kappa a_i(g_{S_i}-\sgn(S_i)g_{S_i^\perp})+\sum_{i=\kappa+1}^M a_i(g_{S_i}-\sgn(S_i)g_{S_i^\perp}).
\end{equation*}
By Proposition \ref{equ_ss-comp:g_S-dual}, for $1\leq i\leq \kappa$, each difference $g_{S_i}-\sgn(S_i)g_{S_i^\perp}$ has a distinct worst term.
In contrast, for $\kappa+1\leq i\leq M$, the worst terms of each $g_{S_i}-\sgn(S_i)g_{S_i^\perp}$ is the same, namely $e_{\{n+1,\cdots,2n\}}$.
We first consider the partial summation of the form $\sum_{i=\kappa+1}^M a_i(g_{S_i}-\sgn(S_i)g_{S_i^\perp})$. It equals
\begin{align*}
    &(-1)^{\kappa+1}\pi^{(n-\kappa)}\Bigl(\sum_{i=\kappa+1}^M(-1)^ia_i\Bigr)\Bigl[
    e_{\{n+1,\cdots,2n\}}+\pi \Bigl(\sum_{\sigma=1}^n(-1)^\sigma e_{[\sigma,\wh{n+\sigma}]}\Bigr)\Bigr]\\
    &\qquad +(-1)^{\kappa+1}\pi^{(n-\kappa-1)}\sum_{i=\kappa+1}^M (-1)^i a_i\left((-1)^{i+1}e_{[i,\wh{n+i}]}+(-1)^{i^\vee+1}e_{[i^\vee,\wh{n+i^\vee}]}\right)
    +\sum_{i=\kappa+1}^M o(\pi^{(n-\kappa-1)} a_i).
\end{align*}
Let $\Delta:=\{i\mid \kappa+1\leq i\leq M, \ord(a_i)\text{ is minimal}\}$. We distinguish the following cases:
\begin{altitemize}
\item 
If $\sum_{i=\kappa+1}^M (-1)^ia_i=0$, then the summation becomes
\begin{equation*}
    (-1)^{\kappa+1}\pi^{(n-\kappa-1)}\sum_{i=\kappa+1}^M (-1)^i a_i\left((-1)^{i+1}e_{[i,\wh{n+i}]}+(-1)^{i^\vee+1}e_{[i^\vee,\wh{n+i^\vee}]}\right)
    +\sum_{i=\kappa+1}^M o(\pi^{(n-\kappa-1)} a_i).
\end{equation*}
The worst terms are 
\begin{equation*}
    (-1)^{\kappa+1}\pi^{(n-\kappa-1)}\sum_{i\in\Delta} (-1)^i a_i\left((-1)^{i+1}e_{[i,\wh{n+i}]}+(-1)^{i^\vee+1}e_{[i^\vee,\wh{n+i^\vee}]}\right).
\end{equation*}
Therefore, the summation is defined over the ring of integers if and only if $\ord_{i\in\Delta}(a_i)\geq n-\kappa-1$.
\item 
If $\sum_{i=\kappa+1}^M(-1)^ia_i\neq 0$, we further distinguish the following three cases: 
\begin{altenumerate2}
\item When $\ord\left(\sum_{i=\kappa+1}^M (-1)^i a_i\right)=\ord_{i\in\Delta}(a_i)$.
\item When $\ord\left(\sum_{i=\kappa+1}^M (-1)^i a_i\right)=\ord_{i\in\Delta}(a_i)+1$.
\item When $\ord\left(\sum_{i=\kappa+1}^M (-1)^i a_i\right)\geq\ord_{i\in\Delta}(a_i)+2$.
\end{altenumerate2}

\noindent Case (a): Suppose $\ord\Bigl(\sum_{i=\kappa+1}^M (-1)^i a_i\Bigr)=\ord_{i\in\Delta}(a_i)$. Then the worst terms in the partial summation are
$$(-1)^{\kappa+1}\pi^{-(n-\kappa)}\Bigl(\sum_{i=\kappa+1}^M(-1)^ia_i\Bigr)e_{\{n+1,\cdots,2n\}}.$$
Therefore, the partial summation is defined over the ring of integers if and only if $\ord\left(\sum_{i=\kappa+1}^M (-1)^i a_i\right)=\ord_{i\in\Delta}(a_i)\geq n-\kappa$.

\noindent Case (b): Suppose $\ord\left(\sum_{i=\kappa+1}^M (-1)^i a_i\right)=\ord_{i\in\Delta}(a_i)+1$. Then the worst terms in the partial summation are
\begin{equation*}
    (-1)^{\kappa+1}\pi^{(n-\kappa)}\Bigl[\Bigl(\sum_{i=\kappa+1}^M(-1)^ia_i\Bigr)e_{\{n+1,\cdots,2n\}}+\pi\sum_{i\in\Delta} (-1)^i a_i\left((-1)^{i+1}e_{[i,\wh{n+i}]}+(-1)^{i^\vee+1}e_{[i^\vee,\wh{n+i^\vee}]}\right)\Bigr].
\end{equation*}
Therefore, the partial summation is defined over the ring of integers if and only if $\ord_{i\in\Delta}(a_i)\geq n-\kappa-1$.

\noindent Case (c): Suppose $\ord\left(\sum_{i=\kappa+1}^M (-1)^i a_i\right)\geq\ord_{i\in\Delta}(a_i)+2$. Then the worst terms in the partial summation are
\begin{equation*}
    (-1)^{\kappa+1}\pi^{(n-\kappa-1)}\sum_{i\in\Delta} (-1)^i a_i\left((-1)^{i+1}e_{[i,\wh{n+i}]}+(-1)^{i^\vee+1}e_{[i^\vee,\wh{n+i^\vee}]}\right).
\end{equation*}
Therefore, the partial summation is defined over the ring of integers if and only if $\ord_{i\in\Delta}(a_i)\geq n-\kappa-1$.
\end{altitemize}

By computing the worst terms in the partial summation
$\sum_{i=\kappa+1}^M a_i(g_{S_i}-\sgn(S_i)g_{S_i^\perp})$,
we observe that as long as the worst term lies in the lattice, the coefficients of the terms 
$e_{[i,\wh{n+i}]}$ for $1\leq i\leq \kappa$ or $1\leq i^\vee\leq \kappa$ lie in the ring of integers.

On the other hand, for $1\leq i\leq \kappa$, the worst terms of  $g_{S_i}-\sgn(S_i)g_{S_i^\perp}$ are precisely the terms of $e_{[i,\wh{n+i}]}$ and $e_{[i^\vee,\wh{n+i^\vee}]}$. 
The partial summation $\sum_{i=1}^\kappa a_i(g_{S_i}-\sgn(S_i)g_{S_i^\perp})$ lies in $W(\Lambda_\kappa)_{-1}^{n-1,1}$ if and only if (x) holds.

Moreover, the two partial summations in the expression $w_{(1,1,\cdots,1)}=\sum_{i=1}^\kappa+\sum_{i=\kappa+1}^M$ have distinct worst terms, and their worst terms computations are independent.
Hence, (x),(xi), and (xii) follow.
\end{proof}

Combining the computation of the worst terms with the proof of Proposition \ref{equ_ss-comp:lattice}, we obtain the following corollary, which is analogous to \cite[Corollary 4.14]{Smithling2015} and \cite[Proposition 2.]{Yu2019}. 
Recall the definition of $L^{r,s}_{(-1)^s}(\Lambda)(R)$ in \eqref{moduli_ss:L} and that of definition of $e_{[i,\wh{n+j}]}$ in \eqref{equ_ss-comp:not}.
		
\begin{corollary}\label{equ_ss-comp:basis}
Let $R$ be a $k$-algebra. Then the submodule $L^{n-1,1}_{-1}(\Lambda_\kappa)(R)\subset W_{-1}^{n-1,1}(R)$ is a free $R$-module with an $R$-basis given by:
\begin{altenumerate}
\item $e_{\{n+1,\cdots,2n\}}$;
\item $e_{[i,\wh{n+i^\vee}]}$ for $i\neq i^\vee$;
\item $e_{[i,\wh{n+j}]}-(-1)^{n+i+j}e_{[j^\vee,\wh{n+i^\vee}]}$ for $i<j^\vee\leq \kappa,i\neq j$;
\item $e_{[j^\vee,\wh{n+i^\vee}]}$ for $i\leq \kappa<j^\vee<n-\kappa+1$;
\item $e_{[i,\wh{n+j}]}+(-1)^{n+i+j}e_{[j^\vee,\wh{n+i^\vee}]}$ for $i\leq \kappa,j^\vee\leq n-\kappa+1,i\neq j$;
\item $e_{[i,\wh{n+j}]}-(-1)^{n+i+j}e_{[j^\vee,\wh{n+i^\vee}]}$ for $\kappa<i<j^\vee<n-\kappa+1,i\neq j$;
\item  $e_{[i,\wh{n+j}]}$ for $\kappa<i<n-k+1\leq j^\vee,i\neq j$;
\item $e_{[i,\wh{n+j}]}-(-1)^{n+i+j}e_{[j^\vee,\wh{n+i^\vee}]}$ for $n-\kappa+1\leq i<j^\vee,i\neq j$;
\item  $e_{[i,\wh{n+i}]}+(-1)^ne_{[i^\vee,\wh{n+i^\vee}]}$ for $i\leq \kappa$;
\item  $e_{[i,\wh{n+i}]}+(-1)^ne_{[i^\vee,\wh{n+i^\vee}]}$ for $\kappa< i\leq M$;
\item Let $w=\sum_{i=1}^M c_i e_{[i,\wh{n+i}]}\in W(\Lambda_\kappa)\otimes R$, Then $w$ lies in the image if and only if 
    \begin{altenumerate2}
	   \item When $n=2m$, we have $\sum_{i=\kappa}^m (-1)^ic_i=0$;
	   \item When $n=2m+1$, we have $\sum_{i=\kappa}^m (-1)^ic_i+\frac{1}{2}(-1)^{m+1}c_{m+1}=0$.\\
	All elements $w$ of the forms (a) and (b) generate a free submodule, whose basis can be completed to a basis of $L_{-1}^{n-1,1}(\Lambda_\kappa)(R)$ by the elements (i)-(x).
    \end{altenumerate2}
\end{altenumerate}
\end{corollary}
\begin{remark}
	If we drop the assumption that $S$ is balanced and retain the symmetry, then $c_i=c_{i^\vee}$. In this case, the condition in (xi) can be rewritten as $\sum_{i=\kappa+1}^{n-\kappa}(-1)^ic_i=0$.	
\end{remark}
	
\begin{proof}
Part (i)-(x) follow immediately from Proposition \ref{equ_ss-comp:lattice} and its proof.
For part (xi), we again use Proposition \ref{equ_ss-comp:lattice}.
When $\kappa+1\leq \sigma\leq m$, we have $\sigma\neq \sigma^\vee$, and 
\begin{equation*}
	c_\sigma=(-1)^{\sigma}\pi^{-(n-\kappa-1)}\sum_{\substack{\kappa+1\leq i \leq M \\ i\neq \sigma}} (-1)^{\kappa+i+1} a_i=(-1)^\kappa\pi^{-(n-\kappa-1)}a_\sigma.
\end{equation*}
When $\sigma=m+1$, we have
\begin{align*}
	c_{m+1}={}&(-1)^{\kappa+m}\pi^{-(n-\kappa-1)}\Bigl(\sum_{i=\kappa+1}^M (-1)^i a_i\Bigr)+2(-1)^\kappa\pi^{-(n-\kappa-1)}a_{m+1},\\
	={}&	2\cdot (-1)^{\kappa+1}\pi^{-(n-\kappa-1)}a_{m+1}.
\end{align*}		
Applying the requirement that $\sum_{i=\kappa+1}^{M}(-1)^ia_i= 0$, we get (xi).
\end{proof}

\subsection{Coordinate ring of $U_{\{\kappa\}}$ over special fiber}\label{equ_coord}
Next, we translate the strengthened spin condition in Corollary \ref{equ_ss-comp:basis} into algebraic relations on the affine chart $U\subset\Gr(n,\Lambda_\kappa\otimes_{\CO_{F_0}}\CO_F)$. Combined with the computations in \S \ref{equ_nw}, we obtain the coordinate rings of $U_{\{\kappa\},s}$.
		
\subsubsection{Translation of the strengthened spin relations}\label{equ_coord_translation}
In this part, we return to the standard basis \eqref{equ_chart:standard-basis} of $\CF_\kappa$.
Let $R$ be a $k$-algebra. For $\CF_\kappa\in U_{\{\kappa\},s}(R)$, write
$\bigwedge^n_R \mathcal{F}_\kappa=\bigwedge^n_R  \left(\begin{matrix} X\\I_n \end{matrix}\right)=\sum_S c_S e_S\in W(\Lambda_k)\otimes_{\CO_F} R$.
By Example \ref{equ_ss-setup:possible-weights}, for those $S$ of type $(1,n-1)$ whose weight vector is not equal to $(1,1,\cdots,1)$, the coefficient $c_S$ is determined by a unique entry of $\CX$ (or $X$).
To be more precise, if we write $\CX:=(z_{ij})$, then the coefficient $c_{\{i,n+1,\cdots,\wh{n+j},\cdots,2n\}}$ of the term $e_{\{i,n+1,\cdots,\wh{n+j},\cdots,2n\}}$ in $\bigwedge^n_R \mathcal{F}_\kappa$ is given by $(-1)^{j-1}z_{ij}$.
		
\begin{proposition}\label{equ_coord_translation:not-show-up}
The following basis elements $e_S$ do not appear among the terms in the basis of $L^{n-1,1}_{-1}(\Lambda_\kappa)(R)\subset W_{-1}^{n-1,1}(R)$ listed in Corollary \ref{equ_ss-comp:basis}:
\begin{altenumerate}
	\item $e_S$ with $\#(S\cap{1,\cdots,n})\geq 2$;
	\item $e_S$ for $S=\{i,n+1,\cdots,\wh{n+j},\cdots,n\}$, where $\kappa+1\leq j\leq n-\kappa$,  and ($1\leq i\leq \kappa$ or $n-\kappa+1\leq i\leq n$). 
\end{altenumerate}
In particular, this implies the following equalities hold in $U_{\{\kappa\},s}$:
\begin{equation*}
\bigwedge^2 X=0,\quad\text{and}\quad E=F=0
\end{equation*}
See the notation in \eqref{equ_chart:standard-X}.
\end{proposition}
\begin{proof}
This follows directly from Corollary \ref{equ_ss-comp:basis}.
For the reader's convenience, we illustrate Corollary  \ref{equ_ss-comp:basis} with the following figure.
The entry at position $(i,j)$ corresponds to the coefficient of the term $e_{\{i,n+1,\cdots,\wh{n+j},\cdots,2n\}}$. 
We indicate in the figure which entries correspond to basis elements appearing in Corollary \ref{equ_ss-comp:basis}.

Moreover, the dotted line  highlights the terms described in (ii) of Corollary \ref{equ_ss-comp:basis}, while the dashed line corresponds to those in (ix).
			
\usetikzlibrary{positioning}
\begin{center}
\begin{tikzpicture}
    \draw[help lines] (0,0) grid (6,6);
    \filldraw[yellow!6](0,6) -- (2,6) -- (2,4) --cycle;
    \filldraw[yellow!6](0,4) -- (0,6) -- (2,4) --cycle;
    \filldraw[yellow!6](4,0) -- (4,2) -- (6,0) --cycle;
    \filldraw[yellow!6](4,2) -- (6,2) -- (6,0) --cycle;
    \filldraw[cyan!6](0,0) -- (2,0) -- (2,2) --cycle;
	\filldraw[cyan!6](0,0) -- (0,2) -- (2,2) --cycle;
	\filldraw[cyan!6](4,4) -- (6,4) -- (6,6) --cycle;
	\filldraw[cyan!6](4,4) -- (4,6) -- (6,6) --cycle;
	\filldraw[magenta!4] (2,4) rectangle (4,6);
	\filldraw[magenta!4] (2,0) rectangle (4,2);
	\filldraw[black!2] (0,2) rectangle (2,4);
	\filldraw[black!2] (4,2) rectangle (6,4);
	\filldraw[green!3] (2,2) rectangle (4,4);
	\draw[black] (0,0) -- (0,6);
	\draw[black] (2,0) -- (2,6);
	\draw[black] (4,0) -- (4,6);
	\draw[black] (6,0) -- (6,6);
	\draw[black] (0,0) -- (6,0);
	\draw[black] (0,2) -- (6,2);
	\draw[black] (0,4) -- (6,4);
	\draw[black] (0,6) -- (6,6);
	\draw[dashed] (0,6) -- (6,0);
	\draw[densely dotted] (0,0) -- (6,6);
	\filldraw[black] (0,0) circle (1pt) node [anchor=east]{$1$};
	\filldraw[black] (0,2) circle (1pt) node [anchor=east]{$\substack{\kappa+1\\\kappa}$};
	\filldraw[black] (0,4) circle (1pt) node [anchor=east]{$\substack{n-\kappa+1\\n-\kappa}$};
	\filldraw[black] (0,6) circle (1pt) node [anchor=east]{$n$};
	\filldraw[black] (0,0) circle (1pt) node [anchor=north]{$1$};
	\filldraw[black] (2,0) circle (1pt) node [anchor=north]{\tiny $\kappa,\kappa+1$};
	\filldraw[black] (4,0) circle (1pt) node [anchor=north]{\tiny $n-\kappa,n-\kappa+1$};
	\filldraw[black] (6,0) circle (1pt) node [anchor=north]{\tiny $n$};
	\draw[black] (1.33,0.66) circle (0pt) node[cyan][anchor=center]{(v)};
	\draw[black] (0.66,1.33) circle (0pt) node[cyan][anchor=center]{(v)};
	\draw[black] (4.66,5.33) circle (0pt) node[cyan][anchor=center]{(v)};
	\draw[black] (5.33,4.66) circle (0pt) node[cyan][anchor=center]{(v)};
	\draw[black] (3,1) circle (0pt) node[magenta][anchor=center]{(vii)};
	\draw[black] (3,5) circle (0pt) node[magenta][anchor=center]{(iv)};
	\draw[black] (4.66,0.66) circle (0pt) node[orange][anchor=center]{(viii)};
	\draw[black] (5.33,1.33) circle (0pt) node[orange][anchor=center]{(viii)};
	\draw[black] (0.66,4.66) circle (0pt) node[orange][anchor=center]{(iii)};
	\draw[black] (1.33,5.33) circle (0pt) node[orange][anchor=center]{(iii)};
	\draw[black] (3,2.33) circle (0pt) node[red][anchor=center]{(vi)};
	\draw[black] (2.33,3) circle (0pt) node[red][anchor=center]{(vi)};
	\draw[black] (3.66,3) circle (0pt) node[blue][anchor=center]{(i)};
	\draw[black] (3,3.66) circle (0pt) node[blue][anchor=center]{(i)};
\end{tikzpicture}
\end{center}
Now part (ii) follows directly from the figure and implies that $z_{ij}=0$ for $\kappa+1\leq j\leq n-\kappa$, and ($1\leq i\leq \kappa$ or $n-\kappa+1\leq i\leq n$), hence $E=F=0$.
Part (i) is straightforward without the figure since only the terms of type $(n,0)$ and $(n-1,1)$ appear; this gives $\bigwedge^2 X=0$.
\end{proof}

Next, we translate Corollary \ref{equ_ss-comp:basis} (iii), (v), (vi), (viii), (ix), (x), and (xi) into algebraic relations in $\CX$. The following lemma, which can be verified by direct computation, will be useful in the remaining part of the subsection.
\begin{lemma}\label{equ_coord:trans-ad}
	Suppose $M=(m_{ij})$ and $N=(n_{ij})$ are two $\kappa\times \kappa$ matrices. Then $m_{ij}=n_{\kappa+1-j,\kappa+1-i}$ for any $i,j$ is equivalent to $M=HN^tH:=N^{\ad}.$\qed
\end{lemma}
		
\subsubsection{}\label{equ_coord_translation:BC}
By Corollary \ref{equ_ss-comp:basis} (iii), the condition $i<j^\vee\leq \kappa$ implies that $i\leq \kappa,j\geq n-\kappa+1$, $j^\vee\leq k$ and $i^\vee\geq n-\kappa+1$.
In this case, we have 
$$z_{ij}=c_{\bi\bj}\quad\text{and}\quad z_{j^\vee i^\vee}=c_{\bar\bj\bar\bi},$$
where the indices are defined as follows:
$$\bi=i,\quad \bj=j-(n-2\kappa),\quad \ov\bi=i^\vee-(n-2\kappa)=\kappa+1-\bi,\quad \ov\bj=j^\vee=\kappa+1-\bj.$$
Here we reorder the entries using $\bi,\bj,\bar{\bi}$ and $\bar{\bj}$ so that $C=(c_{\bi\bj})$ with $1\leq \bi,\bj\leq \kappa$. This reordered notation will be used again later for similar purposes.
We have
\begin{equation*}
	c_{\{i,n+1,\cdots,\wh{n+j},\cdots,2n\}}=(-1)^{n+i+j+1}c_{\{j^\vee,n+1,\cdots,\wh{n+i^\vee},\cdots,2n\}}.
\end{equation*}
Hence
\begin{equation}\label{maximal_affine:c-adjoint}
	c_{\bi\bj}=z_{ij}=z_{j^\vee i^\vee}=c_{\bar\bj\bar\bi}.
\end{equation}
By Lemma \ref{equ_coord:trans-ad}, the relation in \eqref{maximal_affine:c-adjoint} from Corollary \ref{equ_ss-comp:basis}(iii)  is equivalent to $C=C^{\ad}$. Similarly, Corollary \ref{equ_ss-comp:basis} (viii) is equivalent to $B=B^{\ad}$.

\subsubsection{}\label{equ_coord_translation:AD}
Next, the basis in Corollary \ref{equ_ss-comp:basis} (v) and (ix) can be reformulated into the basis
\begin{equation*}
	e_{\{i,n+1,\cdots,\wh{n+j},\cdots,2n\}}+(-1)^{n+i+j}e_{\{j^\vee,n+1,\cdots,\wh{n+i^\vee},\cdots,2n\}}\quad\text{for}\quad i\leq \kappa\text{ and }j^\vee\geq n-\kappa+1.
\end{equation*}
In this case, we have $i\leq \kappa,j\leq \kappa$, $j^\vee\geq n-\kappa+1$ and $i^\vee\geq n-\kappa+1$.
Therefore, 
$$
z_{ij}=d_{\bi\bj}\quad\text{and}\quad z_{j^\vee i^\vee}=a_{\bar\bj\bar\bi},
$$
where the indices are defined as follows:
$$
\bi=i,\quad \bj=j,\quad \ov\bi=i^\vee-(n-2\kappa)=k+1-\bi,\quad  \ov\bj=j^\vee=k+1-\bj
$$
Hence
\begin{equation*}
	c_{\{i,n+1,\cdots,\wh{n+j},\cdots,2n\}}=(-1)^{n+i+j}c_{\{j^\vee,n+1,\cdots,\wh{n+i^\vee},\cdots,2n\}}.
\end{equation*}
Hence
\begin{equation*}
	d_{\bi\bj}=z_{ij}=-z_{j^\vee i^\vee}=-a_{\bar\bj\bar\bi}.
\end{equation*}
By Lemma \ref{equ_coord:trans-ad}, we have $D=-A^{\ad}$.
		
\subsubsection{}\label{equ_coord_translation:X4}
The basis in Corollary \ref{equ_ss-comp:basis} (vi) and (viii) can be rewritten as the basis of the form
\begin{equation*}
	e_{\{i,n+1,\cdots,\wh{n+j},\cdots,2n\}}-(-1)^{n+i+j}e_{\{j^\vee,n+1,\cdots,\wh{n+i^\vee},\cdots,2n\}},\text{ for } \kappa<i<j^\vee<n-\kappa+1\text{ and }i+j^\vee\leq n+1.
\end{equation*}
In this case, we have $\kappa<i<n-\kappa+1,\kappa<j<n-\kappa+1$, $\kappa<j^\vee<n-\kappa+1$ and $\kappa<i^\vee<n-\kappa+1$. 
Therefore, 
$$
z_{ij}=x^{(4)}_{\bi\bj}\quad\text{and}\quad z_{j^\vee i^\vee}=x^{(4)}_{\ov\bj\ov\bi},
$$
where the indices are defined as follows:
$$
\bi=i-\kappa,\quad \bj=j-\kappa,\quad \ov\bi=i^\vee-i=(n-2\kappa+1)-\bi,\quad \ov\bj=j^\vee-\kappa=(n-2\kappa+1)-\bj.
$$
Then we obtain
\begin{equation*}
	c_{\{i,n+1,\cdots,\wh{n+j},\cdots,2n\}}=(-1)^{n+i+j+1}c_{\{j^\vee,n+1,\cdots,\wh{n+i^\vee},\cdots,2n\}}.
\end{equation*}
Hence
\begin{equation*}
	x_{\bi\bj}^{(4)}=z_{ij}=z_{j^\vee i^\vee}=x^{(4)}_{\bar\bj\bar\bi}.
\end{equation*}
By Lemma \ref{equ_coord:trans-ad}, Corollary \ref{equ_ss-comp:basis} (vi) and (x) are equivalent to claiming that $X_4=X_4^{\ad}.$
		
\subsubsection{}\label{equ_coord_translation:trace}
For  Corollary \ref{equ_ss-comp:basis} (xi), when $\kappa<i\leq M$, we have 
$$c_i=c_{\{i,n+1,\cdots,\wh{n+i},\cdots,2n\}}=(-1)^{i-1} x_{ii}.$$
Therefore, when $n=2m$, the identity $\sum_{i=\kappa+1}^m (-1)^i c_i=0$ is equivalent to $\sum_{i=\kappa+1}^m x_{ii}=0$, and further equivalent to $\sum_{i=\kappa+1}^{n-\kappa+1}x_{ii}=0$, since $X_4=X_4^{\ad}$.

When $n=2m+1$, the relation $\sum_{i\geq \kappa+1}^m (-1)^i c_i+\frac{1}{2}(-1)^{m+1}c_{m+1}=0$, is again equivalent to
$\sum_{i=\kappa+1}^{n-\kappa+1}x_{ii}=0$, for the same reason.

Combining these two cases, we deduce that the description of (xi) in Corollary \ref{equ_ss-comp:basis} is equivalent to requiring that $\tr(X_4)=0$.
		
To summarize, what we have shown is
\begin{theorem}\label{equ_coord:ss-red}
For the affine chart over the special fiber $U_s\subset \Gr(n,\Lambda_\kappa\otimes_{\CO_{F_0}}\CO_F/(\pi))$, using the notation from \eqref{equ_chart:reordered-X}, the strengthened spin condition is equivalent to the following:
\begin{altenumerate}
	\item $\bigwedge^2X=0,\, X_2=0$,
	\item $C=C^{\ad},\, B=B^{\ad},\, D=-A^{\ad}$,
	\item $X_4=X_4^{\ad},\,\tr(X_4)=0$.
\end{altenumerate}
\end{theorem}
\begin{proof}
Equations in (i) follow from Proposition \ref{equ_coord_translation:not-show-up}.
Equations in (ii) follow from \S \ref{equ_coord_translation:BC} and \S \ref{equ_coord_translation:AD}.
Equations in (iii) follow from \S \ref{equ_coord_translation:X4} and \S \ref{equ_coord_translation:trace}.
\end{proof}

\subsubsection{Coordinate ring of $U_{\{\kappa\}}$}
Combined with Proposition \ref{equ_nw:summary} and Theorem \ref{equ_coord:ss-red}, we deduce the defining equations of $U_{\{\kappa\},s}$.
\begin{theorem}\label{equ_coord:ss-main}
The special fiber $U_{\{\kappa\},s}$ of the affine chart defined in \S \ref{equ_chart} is isomorphic to the spectrum of the ring $k[X]/\CI$, where $\CI$ is the ideal generated by the entries of the following matrices:
\begin{altitemize}
	\item[\fontfamily{cmtt}\selectfont LM1.] $X_1^2+X_2X_3,\, X_1X_2+X_2X_4,\,
	X_3X_1+X_4X_3,X_3X_2+X_4^2$,
				
	\item[\fontfamily{cmtt}\selectfont LM2-1.] $-JX_1+X_3^t HX_3+X_1^t J,\, -JX_2+X_3^tHX_4,\, X_2^tJ+X_4^tHX_3,\, X_4^tHX_4$,
				
	\item[\fontfamily{cmtt}\selectfont LM2-2.] $X_1JX_1^t,\, X_1JX_3^t-X_2H,\, X_3JX_1^t+HX_2^t,\, X_3JX_3^t-X_4H+HX_4^t$,
				
	\item[\fontfamily{cmtt}\selectfont LM8-1.] $\bigwedge^2 X,\, X_2$,
				
	\item[\fontfamily{cmtt}\selectfont LM8-2.] $C-C^{\ad},\, B-B^{\ad},\, D+A^{\ad}$,\, $X_4-X_4^{\ad},\, \tr(X_4)$.\qed
\end{altitemize}
\end{theorem}
		
\section{Reducedness for the strongly non-special parahoric subgroup}\label{max}
In this section, we first simplify the ideal $\CI$ defining $U_{\{\kappa\}}$ in Theorem \ref{equ_coord:ss-main}, and then prove that the scheme is reduced.
		
\subsection{Simplification}\label{max_simp}
In this subsection, we simplify the coordinate ring of $U_{\{\kappa\}}$. The simplification will be carried out in several steps.

\bigskip
\noindent\textbf{Step 1.}
In the first step, since $X_2=0$ modulo $\CI$, it follows that $\CI$ is generated by the entries of the following matrices:
\begin{altenumerate}\label{max_simp-step-1}
    \item $X_1^2,\, X_3X_1+X_4X_3,\, X_4^2$,
			
    \item $-JX_1+X_3^t HX_3+X_1^t J,\, X_3^tHX_4,\, X_4^tHX_3,\, X_4^tHX_4$,
			
    \item $X_1JX_1^t,\, X_1JX_3^t,\, X_3JX_1^t,\, X_3JX_3^t-X_4H+HX_4^t$,
			
    \item $\bigwedge^2 X,\, X_2$,
			
    \item $C-C^{\ad},\, B-B^{\ad},\, D+A^{\ad},\, X_4-X_4^{\ad},\, \tr(X_4)$.
\end{altenumerate}
		
We will show that the matrices in (i),(ii), and (iii) can be further simplified, leading to the following proposition.
\begin{proposition}\label{max_simp:step1}
The ideal $\CI$ is generated by the entries of the following matrices:
\begin{altitemize}
	\item $-JX_1+X_1^tJ+X_3^tHX_3,\, X_3X_1$.
	\item $\bigwedge^2 X,\, X_2$,
	\item $C-C^{\ad},\, B-B^{\ad},\, D+A^{\ad},\, X_4-X_4^{\ad},\, \tr(X_4)$,
\end{altitemize}
\end{proposition}
For the remaining part of the Step 1, we provide a proof of the Proposition \ref{max_simp:step1}. The following Lemma will be useful in the argument.
\begin{lemma}
Let $M=\left(\begin{matrix} X&Y\end{matrix}\right)$, where $X$ is a square matrix.
If 
 $\tr(X)=\lambda$ and $\bigwedge^2 M=0$, then it follows that $XY=\lambda Y$.   
\end{lemma}
\begin{proof}
Write $X=(x_{ij})_{m\times m}$ and $Y=(y_{ij})_{m\times n}$.
Then the $(i,j)$-entry of the product $XY$ is given by $\sum_\sigma x_{i\sigma}y_{\sigma j}$.
			
The condition $\bigwedge^2 M=0$ implies that 
$$x_{i\sigma }y_{\sigma j}-x_{\sigma\sigma}y_{ij}=0$$
for all $i,\sigma$ (including $i=\sigma$).
Since $\tr(X)=\lambda$, the entries of the matrix $XY$ are
\begin{equation*}
	\sum_\sigma x_{i\sigma}y_{\sigma j}=\sum_{\sigma=1}^m x_{\sigma \sigma}y_{ij}=y_{ij}\sum x_{\sigma\sigma}=\lambda y_{ij}.\qedhere
\end{equation*}
\end{proof}

\begin{corollary}\label{max_simp:zeros}
	We have $X_4X_3=0$, $X_4^2=0$ and $X_1^2=0$.
\end{corollary}  
\begin{proof}
Given that $\tr(X_4)=0$ and $\bigwedge^2 X=0$, it follows that $X_4X_3=0$ and $X_4^2=0$.
$D=-A^{\ad}$ implies $\tr(X_1)=0$, hence $X_1^2=0$.
\end{proof}
By Corollary \ref{max_simp:zeros}, the matrices in \eqref{max_simp-step-1}(i) simplify into $X_3X_1=0$.
		
The relation $X_3^tHX_4=0$ is equivalent to $X_4^tHX_3=0$, which in turn is equivalent to $X_4^{\ad}X_3=HX_4^tHX_3=0$.
Hence, it follows from $X_4=X_4^{\ad}$ and $X_3X_4=0$.
Similarly, the relation $X_4^tHX_4=0$ follows from $X_4=X_4^{\ad}$ and $X_4^2=0$.
Therefore, the relations in \eqref{max_simp-step-1}(ii) are equivalent to the single equation $-JX_1+X_3^tHX_3+X_1^tJ=0$.

\begin{lemma}\label{max_simp:step1_ad}
	The relations $C=C^{\ad}, B=B^{\ad},D=-A^{\ad}$ are equivalent to $JX_1=-X_1^tJ$.
\end{lemma}
\begin{proof}
	Since $J^2=I$, the equality $JX_1=-X_1^tJ$ is equivalent to $X_1=JX_1^tJ$. A quick calculation shows that
\begin{equation*}
	JX_1^tJ=\left(\begin{matrix}-HD^tH&HB^tH\\HC^tH&-HA^tH\end{matrix}\right)=
	\left(\begin{matrix}
		-D^\ad  &   B^\ad\\
		C^\ad   &   -A^\ad
	\end{matrix}\right).\qedhere
\end{equation*}
\end{proof}
In particular, the proof shows that $X_1=JX_1^tJ$. 
It follows that $X_1JX_1^t=0$ is a consequence of $X_1^2=0$.
Moreover, the relation $X_1JX_3^t=0$ is equivalent to $X_3JX_1^t=0$, and both follow from $X_3X_1=0$. This finishes the proof of Proposition \ref{max_simp:step1}.\qed

\bigskip
\noindent\textbf{Step 2.}
Applying Lemma \ref{max_simp:step1_ad} to the matrix $-JX_1+X_3^t HX_3+X_1^t J$ in $\CI$, we find that the entries of the matrix
\[
X_1+\frac{1}{2}JX_3^tHX_3
\]
also lie in the ideal $\CI$.
We will use this matrix to further simplify $\CI$.
\begin{proposition}\label{max_simp:simp-step2}
	The ideal $\CI$ is equal to the ideal generated by the entries of the following matrices:
\begin{equation}\label{max_simp:simp-step2-ideal-generator}
	\bigwedge^2 X,\quad X_2,\quad X_4-X_4^{\ad},\quad \tr(X_4),\quad  X_1+\frac{1}{2}JX_3^tHX_3.
\end{equation}
\end{proposition}
		
\begin{proof}
Recall that $X_3=\left(\begin{matrix}E&F\end{matrix}\right)$, where $E$ and $F$ are matrices of the size $(n-2\kappa)\times \kappa$. Then we compute:
\begin{equation*}
	-\frac{1}{2}JX_3^tHX_3=-\frac{1}{2}J\left(\begin{matrix}E^t\\F^t\end{matrix}\right)H\left(\begin{matrix}E&F\end{matrix}\right)
	=-\frac{1}{2}\left(\begin{matrix}F^{\ad}E&F^{\ad}E\\-E^{\ad}E&-E^{\ad}F\end{matrix}\right).
\end{equation*}
Therefore, the equality $X_1=-\frac{1}{2}JX_3^tHX_3$ implies
\begin{equation*}
	A=-\frac{1}{2}F^{\ad}E,\quad B=-\frac{1}{2}F^{\ad}F, \quad C=\frac{1}{2}E^{\ad}E,\quad D=\frac{1}{2}E^{\ad}F.
\end{equation*}
Since $(X^{\ad}Y)^{\ad}=Y^{\ad}X$, we deduce $C=C^{\ad},B=B^{\ad},D=-A^{\ad}$ from $X_1=-\frac{1}{2}JX_3^tHX_3$.
	
Next, we claim that $X_3X_1=0$ follows from the relation $\bigwedge^2 X_3=0$, modulo the ideal generated by the entries of \eqref{max_simp:simp-step2-ideal-generator}. Indeed, we ahve:
\begin{align*}
	X_3X_1={}&\left(\begin{matrix}E&F\end{matrix}\right)\left(\begin{matrix}F^{\ad}E&F^{\ad}F\\-E^{\ad}E&-E^{\ad}F\end{matrix}\right),\\
	={}&\left(\begin{matrix}EF^{\ad}E-FE^{\ad}E&EF^{\ad}F-FE^{\ad}F\end{matrix}\right),\\
	={}&\left(\begin{matrix}
		EHF^tHE-FHE^tHE&
		EHF^tHF-FHE^tHF
	\end{matrix}\right),\\
	={}&\left(\begin{matrix}
		(EHF^t-FHE^t)HE&
		(EHF^t-FHE^t)HF
	\end{matrix}\right).
\end{align*}
The proposition now follows from the Lemma \ref{max_simp:simp-step2-ideal-generator-lem} below.
\end{proof}
\begin{lemma}\label{max_simp:simp-step2-ideal-generator-lem}
The relation $\bigwedge^2 X_3=0$ implies $EHF^t-FHE^t=0$.
\end{lemma}
\begin{proof}
Write $E=(e_{ij})$ and $F=(f_{ij})$. both of size $(n-2k)\times k$. 
We have
$$
HF^{t}=H(f_{ji})=(f_{k+1-j,i})\quad\text{and similarly}\quad HE^{t}=H(e_{ji})=(e_{k+1-j,i}).
$$
Thus, the $(i,j)$-entry of the matrix $EHF^t-FHE^t$ is
\begin{equation*}
	\sum_{s=1}^\kappa e_{is}f_{j,k+1-s} - \sum_{t=1}^\kappa e_{js}f_{i,k+1-s},
\end{equation*}
which vanishes due to the relation $\bigwedge^2 X_3=0$.
\end{proof}
		
\begin{remark}\label{max_simp:simp-step2_comp}
Before we proceed, we note that the relations $X_2=0$ and $\bigwedge^2 X=0$ have important consequences. 
In particular, consider $2$-minors in $X$ of the form $st$, where $s$ is an entry of $X_1$ and $t$ is an entry of $X_4$. This indeed corresponds to a $2$-minor, since:
\begin{equation*}
    \left(\begin{array}{c|c}
		X_1&0\\
		\hline
		X_3&X_4
	\end{array}\right)=
	\begin{tikzpicture}[>=stealth,thick,baseline]
		\matrix [matrix of math nodes,left delimiter=(,right delimiter=)](A){ 
		*		&	\cdots	&	*		&			&			&			\\
		\vdots 	&	{\color{magenta}s}		&	\vdots 	&			&			&			\\
		*		&	\cdots	&	*		&			&			&			\\
		*		&	\cdots	&	*		&	*		&	\cdots	&	*		\\
		\vdots 	&	*		&	\vdots 	&	\vdots	&	{\color{magenta}t}		&	\vdots	\\
		*		&	\cdots	&	*		&	*		&	\cdots	&	\\					
		};
		\draw (-2,0) -- (2,0);
		\draw (0,-2) -- (0,2);
		\draw[cyan] (-0.7,0.7) -- (0.75,-0.75);
		\draw[cyan] (-0.7,-0.7) -- (0.75,0.75);
	\end{tikzpicture}\rightsquigarrow \text{2-minors } s\cdot t-*\cdot0=st.
\end{equation*}
It is natural to \emph{guess} that $\Spec k[X]/\CI$ has two irreducible components, corresponding to the loci $X_1=0$ and $X_4=0$ respectively.
In light of Proposition \ref{max_simp:simp-step2}, these two ``components'' correspond to the vanishing of the following sets of generators:
\begin{altitemize}
	\item $\CI_1$: $X_4, X_2, \bigwedge^2 (X_1,X_3), X_1+\frac{1}{2}JX_3^tHX_3$,
    \item $\CI_2$: $X_1,X_2,\bigwedge^2(X_3,X_4), X_4-X_4^{\ad},\tr(X_4),X_3^tHX_3$.
\end{altitemize}
In Proposition \ref{max_red:irr-red}, we will verify that these two ideals define reduced and irreducible components.
For now, we show that this is true at least at the topological level.
\end{remark}

\begin{lemma}\label{max_simp:simp-step2_comp-radical}
Keep the notations in Remark \ref{max_simp:simp-step2_comp}, we have $\sqrt{\CI}=\sqrt{\CI_1}\cap\sqrt{\CI_2}$.
\end{lemma}
\begin{proof}
Let $K\subset k[X]$ be the ideal generated by the entries of 
\begin{equation*}
	\bigwedge^2 (X_3,X_4), \bigwedge^2(X_1,X_3), X_2, X_4-X_4^{\ad}, \tr(X_4), X_1+\frac{1}{2}JX_3^t HX_3.
\end{equation*}
Let $I, J$ be the ideals of $k[X]$ generated by the entries of $X_1$ and $X_4$ respectively. 
Then we have $\CI_1=I+K$ and $\CI_2=J+K$, and the whole ideal $\CI$ is generated by $IJ+K$.
Consider the inclusion
\begin{equation*}
    IJ+K^2\subset (I+K)(J+K)\subset IJ+K.
\end{equation*}
Taking radicals on both sides, we get $\sqrt{\CI}=\sqrt{\CI_1\CI_2}=\sqrt{\CI_1}\cap\sqrt{\CI_2}$.
\end{proof}

\noindent\textbf{Step 3.}
In the final step, we will show that
\begin{theorem}\label{sm:simp_step3_ring}
	There is an isomorphism of $k$-algebras
\begin{equation*}
	\frac{k[X]}{\bigwedge^2 X, X_2, X_4-X_4^{\ad},\tr(X_4), X_1+\frac{1}{2}JX_3^tHX_3}\xrightarrow{\sim} \frac{k[X_3,X_4]}{\bigwedge^2 (X_3,X_4), X_4-X_4^{\ad},\tr(X_4)}.
\end{equation*}
The map is defined by replacing the entries of $X_1$ by $-\frac{1}{2}JX_3^tHX_3$ and the entries of $X_2$ by $0$.
\end{theorem}
\begin{proof}
We only need to show that the entries of $\bigwedge^2 X$ lies in the ideal $\CI'\subset k[X]$ generated by the entries of the following matrices
\[
\bigwedge^2 (X_3,X_4),\quad X_2,\quad X_4-X_4^{\ad},\quad \tr(X_4),\quad X_1+\frac{1}{2}JX_3^tHX_3.
\]
Notice that the entries of $\bigwedge^2 X$ not already contained in $\bigwedge^2(X_3,X_4)$ fall into the following two cases:
\begin{altenumerate}
	\item The monomials of the form $xy$ where $x$ is an entry of  $X_1$ and $y$ is an entry of $X_4$, see Remark \ref{max_simp:simp-step2_comp}.
	\item The entries of $\bigwedge^2(X_1,X_3)$.
\end{altenumerate}
			
In case (i), to show that $xy\in\CI'$, we express $x$ in terms of the entries of $X_3$ using the relation $X_1=-\frac{1}{2}JX_3^tHX_3$.
Recall that $X_1=\left(\begin{matrix}A&B\\C&D\end{matrix}\right)$; 
we will focus on the case where $x$ is an entry of the block $A=(a_{ij})$.
The cases when $x$ lies in $B, C$, or $D$ follow from analogous computations.
			
From $X_1=-\frac{1}{2}JX_3^tHX_3$ and $F^{\ad}=HF^tH$, we have
\[
A=-\frac{1}{2}F^{\ad}E=\Bigl(-\frac{1}{2}\sum_{s=1}^{n-2\kappa} f_{n-2\kappa+1-s,\kappa+1-i}e_{sj}\Bigr).
\]
Write $X_4:=Y=(y_{ij})$ for now. Then the product $xy$ can be rewritten as
\[
xy=y_{pq}\sum_{s=1}^{n-2\kappa}f_{n-2\kappa+1-s,\kappa+1-i}e_{sj},
\]
for some indices $p, q, i, j$. Recall that we are assuming $x=a_{ij}$.
Applying $\bigwedge^2(X_3,X_4)=0$ and $X_4=X_4^{\ad}$, we have
\begin{align*}
	xy={}&\sum_{s=1}^{n-2\kappa}f_{n-2\kappa+1-s,\kappa+1-i}y_{sq}e_{pj},&(e_{sj}y_{pq}=e_{pj}y_{sq})\\
	={}&e_{pj}\sum_{s=1}^{n-2\kappa}f_{n-2\kappa+1-s,\kappa+1-i}y_{sq},&\\
	={}&e_{pj}\sum_{s=1}^{n-2\kappa}f_{n-2\kappa+1-s,\kappa+1-i}y_{n-2\kappa+1-q,n-2\kappa+1-s},&(X_4=X_4^{\ad})\\
	={}&e_{pj}\sum_{s=1}^{n-2\kappa}f_{n-2\kappa+1-q,\kappa+1-i}y_{n-2\kappa+1-s,n-2\kappa+1-s},&\\
	={}&e_{pj}f_{n-2\kappa+1-q,\kappa+1-i}\sum_{n=1}^{n-2\kappa}y_{n-2\kappa+1-s,n-2\kappa+1-s}.
\end{align*}
For the last equality, we use the identity
\[
	f_{n-2\kappa+1-s,\kappa+1-i}y_{n-2\kappa+1-q,n-2\kappa+1-s}-f_{n-2\kappa+1-q,\kappa+1-i}y_{n-2\kappa+1-s,n-2\kappa+1-s}=0,
\]
which follows from the relation $\tr(X_4)=0$. 
Therefore, we conclude that $xy=0$ follows from $\bigwedge^2(X_3,X_4)=0,X_4=X_4^{\ad}, \tr(X_4)=0$ and $X_1=-\frac{1}{2}JX_3^t HX_3$, i.e., $xy\in\CI'$.

In case (ii), we want to show that the entries of $\bigwedge^2(X_1, X_3)$ lie in the ideal $\CI'$. 
Substitute $X_1=-\frac{1}{2}JX_3^t HX_3$, and notice that
\begin{equation*}
	\left(\begin{matrix} -\frac{1}{2}J&\\&I_{n-2k}\end{matrix}\right)\left(\begin{matrix} X_3^tH X_3&X_3\end{matrix}\right)=\left(\begin{matrix}0&X_3\end{matrix}\right)\left(\begin{matrix} I&\\X_3^tH&I\end{matrix}\right).
\end{equation*}
Therefore, the ideal generated by the entries of $\bigwedge^2(X_1,X_3)=\bigwedge^2\left(\begin{matrix}-\frac{1}{2}JX_3^tHX_3	&	X_3\end{matrix}\right)$ is equal to the ideal generated by the entries of $\bigwedge^2\left(\begin{matrix}0&X_3\end{matrix}\right)$, which is contained in $\CI'$.
\end{proof}

\begin{remark}\label{max_simp:simp-step2-simp-comp}
Using the same computation, we can show that the ideals $\CI_1$ and $\CI_2$ of the ``irreducible components'' defined in Remark \ref{max_simp:simp-step2_comp} can be simplified into the ideals in $k[X_3, X_4]$, generated by the entries of following matrices:
\begin{equation*}
    I_1:\bigwedge^2 X_3, X_4,
    \quad\text{and}\quad
    I_2: \bigwedge^2(X_3,X_4),X_4-X_4^\ad,\tr(X_4),X_3^tHX_3.
\end{equation*}
\end{remark}
		
\subsection{Reducedness of the strengthened spin model for a strongly non-special parahoric subgroup}\label{max_red}
For the simplicity of notation, we substitute 
\begin{equation}\label{max_red:sub}
    \bA=X_4 H,\quad\text{and}\quad \bB=X_3.
\end{equation}
Note that this differs from the $A, B$ we defined in \eqref{equ_chart:reordered-X}. The goal of this and next subsections is to prove the following theorem:
\begin{theorem}\label{max_red:reduced_sym}
Suppose $n\geq 5$ and $I=\{\kappa\}$ is strongly non-special. Consider the affine scheme $\Spec R_s=\Spec \frac{k[\bA,\bB]}{\bigwedge^2(\bA,\bB),\bA-\bA^t,\tr(\bA H)}$. Define the following closed subschemes of $\Spec R_s$:
\begin{equation*}
		\Spec R_{s,1}:=\Spec\frac{k[\bA,\bB]}{\bA,\bigwedge^2 \bB},
		\quad\text{and}\quad
		\Spec R_{s,2}:=\Spec \frac{k[\bA,\bB]}{\bigwedge^2(\bA,\bB),\bA-\bA^t,\tr(\bA H),\bB^tH\bB},
\end{equation*}
and their scheme-theoretic intersection:
\begin{equation*}
		\Spec R_{s,12}:=\Spec \frac{k[\bA,\bB]}{\bA,\bigwedge^2(\bB),\bB^tH\bB}.
\end{equation*}
Define the worst point $\{*\}$ to be the closed point defined by $\bA=0,\bB=0$. Then
\begin{altenumerate}
    \item The affine scheme $\Spec R_s$ is reduced and Cohen-Macaulay of dimension $n-1$, with the singular locus $\Spec R_{s,12}$. It is Gorenstein when $n=4\kappa+2$ (but not clear for the other cases).
    
    \item The affine schemes $\Spec R_{s,1}$ and $\Spec R_{s,2}$ are the irreducible components of $\Spec R_s$. Each is normal and Cohen-Macaulay of dimension $n-1$, with singular locus the worst point.
    \item The affine scheme $\Spec R_{s,12}$ is normal and Cohen-Macaulay of dimension $n-2$, with singular locus the worst point.
\end{altenumerate}

\end{theorem}

In this subsection, we verify Theorem \ref{max_red:reduced_sym}(i) using the theory of symmetric determinant varieties, as developed in a series of papers by A. Conca. 
Below, we summarize the relevant results:
\begin{theorem}\label{max_red:Conca}
Suppose $n\geq 3$ and $n-2\kappa\geq 2$. Let $R_0=\frac{k[\bA,\bB]}{\bigwedge^2(\bA,\bB),\bA-\bA^t}$, where $\bA$ and $\bB$ are matrices of indetermines of size $(n-2\kappa)\times (n-2\kappa)$, resp.,\ $(n-2\kappa)\times 2\kappa$. 
Then the ring $R_0$ is normal and Cohen-Macaulay of dimension $n$. It is Gorenstein if and only if $n=4\kappa+2$.
\end{theorem}
\begin{proof}
Since $n-2\kappa\geq 2$, the dimension of $\Spec R_0$ is given by \cite[Page 52]{Conca1994-2}. Note that there is a typo in Conca's formula, the correct expression for the dimension is
\begin{equation*}
	\dim R_t(Z)=(n+m+1-s-\frac{t}{2})(t-1).
\end{equation*}
Applying $m=s=n-2\kappa$, $n=n$ and $t=2$, we obtain the dimension in our setting.
			
By \cite[Proposition 2.5]{Conca1994-2}, the ring $R_0$ is a Cohen-Macaulay domain.
By \cite[Theorem 2.4]{Conca1994-1}, $R_0$ is normal.
According to the same theorem, it is Gorenstein if and only if $2(n-2\kappa)=n+2$, i.e., when $n=4\kappa+2$.
\end{proof}

\begin{corollary}\label{max_red:cm}
The ring $R_s$ is Cohen-Macaulay. When $n=4\kappa+2$, it is Gorenstein.
\end{corollary}
\begin{proof}
Recall that a local ring $(R,\mathfrak{m})$ is Cohen-Macaulay (resp.\ Gorenstein) if and only if $R/(t)$ is Cohen-Macaulay (resp.\ Gorenstein) for some non-zero-divisor $t\in \mathfrak{m}\subset R$ in the maximal ideal; cf.\ \cite[\href{https://stacks.math.columbia.edu/tag/02JN}{02JN}]{Stacks} (resp.\ \cite[\href{https://stacks.math.columbia.edu/tag/0BJJ}{0BJJ}]{Stacks}).

The ring $R_0=\frac{k[\bA,\bB]}{\bigwedge^2(\bA,\bB),\bA-\bA^t}$ is an integral domain by Theorem \ref{max_red:Conca}, so the element $\tr(\bA H)$ is not a zero-divisor. 
It follows that the scheme $\Spec R_s=\Spec \frac{R_0}{\tr(\bA H)}$ is Cohen-Macaulay. 
Moreover, when $n=4\kappa+2$, Theorem \ref{max_red:Conca} implies that $\Spec R_0$ is Gorenstein, and hence $\Spec R_s=\Spec \frac{R_0}{\tr(\bA H)}$ also Gorenstein in this case.
\end{proof}

\begin{remark}\label{max_red:non-gorenstein}
The cited result \cite[\href{https://stacks.math.columbia.edu/tag/0BJJ}{0BJJ}]{Stacks} is purely local. For example, given a closed subscheme $Y\subset X=\Spec R_0$, it is possible for $Y$ to be Gorenstein even if $X$ is not.
\end{remark}

Denote the ideal
\begin{equation*}
		I=\Bigl(\bigwedge^2 (\bA,\bB), \bA-\bA^t,\tr(\bA H)\Bigr)\subset k[\bA,\bB].
	\end{equation*}
And recall
\begin{equation*}
    I_1=\Bigl(\bA,\bigwedge^2\bB\Bigr),
	\quad
	I_2=\Bigl(\bigwedge^2(\bA,\bB),\bA-\bA^t,\tr(\bA H),\bB^tH\bB\Bigr).
\end{equation*}
Recall from Lemma \ref{max_simp:simp-step2_comp-radical} that we have $\sqrt{I}=\sqrt{I_1}\cap\sqrt{I_2}$.

\begin{proposition}\label{max_red:irr}
\begin{altenumerate}
\item The spectra of $R_{s,1}=R_{s}/I_1$ and $R_{s,2}=R_{s}/I_2$ are irreducible of dimension $n-1$, and are smooth away from the worst point.
\item The spectrum of $R_{s,12}=R_{s}/(I_1+I_2)$ is irreducible of dimension $n-2$, and is smooth away from the worst point.
\end{altenumerate}
\end{proposition}
\begin{proof}
For part (i), the irreducibility of $\Spec R_{s,1}$ follows directly from \cite[(2.11)(1.1)]{BU1998}.

For $\Spec R_{s,2}=\Spec k[\bA,\bB]/I_2=\Spec \frac{k[\bA,\bB]}{\bigwedge^2(\bA,\bB),\bA-\bA^t,\tr(\bA H),\bB^t H\bB}$, note that the ideal $I_2$ is homogeneous. Therefore, to show that $\Spec R_{s,2}$ is irreducible of dimension $n-1$, it suffices to prove that the associated projective scheme
\[
	\text{Proj }\frac{k[\bA,\bB]}{\bigwedge^2(\bA,\bB),\bA-\bA^t,\tr(\bA H),\bB^tH\bB},
\]
is irreducible of dimension $n-2$.

Write the entries of the matrix $\bigl(\begin{matrix}\bA&\bB\end{matrix}\bigr)$ as $x_{ij}$, where $1\leq i\leq n-2\kappa$ and $1\leq j\leq n$.
To show irreducibility, by \cite[\href{https://stacks.math.columbia.edu/tag/01OM}{01OM}]{Stacks}, it suffices to verify that each principal open subset $D_+(x_{ij})$ of this projective variety is irreducible, and that the intersection of any two such principal open subsets is nonempty.

Consider the principal open subset $D_+(x_{pq})$ for some fixed $p$ and $q$.
From the $2$-minor relation $x_{ij}x_{pq}=x_{iq}x_{pj}$, we deduce that
$
\frac{x_{ij}}{x_{pq}}=\frac{x_{iq}}{x_{pq}}\cdot\frac{x_{pj}}{x_{pq}}.
$
This identity holds even when $i=p$ or $j=q$. Write
\begin{equation*}
	T_i:=\frac{x_{iq}}{x_{pq}},\quad S_j:=\frac{x_{pj}}{x_{pq}},
\end{equation*}
then we have $\frac{x_{ij}}{x_{pq}}=T_i S_j$.

Therefore, the principal open subset $D_+(x_{ij})$ is isomorphic to the spectrum of a factor ring of the polynomial ring $k[T_i, S_j]$ for $1\leq i\leq n-2\kappa, 1\leq j\leq n$, where, by construction, we always have $T_p=S_q=1$.
Our goal is to express the ideal $I_2$ in terms of the variables $S_i$ and  $T_j$.
				
The relation $\bA=\bA^t$ translates to the identity $T_iS_j=T_jS_i$ for $1\leq i,j\leq n-2\kappa$. 
In particular, we have
\begin{equation}\label{max_red:AAt}
	S_i=T_i\cdot S_p.
\end{equation}
It is straightforward to verify that these relations exhaust all the consequences of the symmetry $\bA=\bA^t$.
				
Next, we compute the trace:
\begin{equation}\label{max_red:trAH}
	\tr(\bA H)=x_{1,n-2\kappa}+x_{2,n-2\kappa-1}+\cdots+x_{n-2\kappa,1}=T_1S_{n-2\kappa}+T_2S_{n-2\kappa-1}+\cdots+T_{n-2\kappa}S_1.
\end{equation}

Now, write $\bB=(b_{ij})$, so that $\frac{b_{ij}}{x_{pq}}=\frac{x_{i,n-2\kappa+j}}{x_{pq}}=T_iS_{n-2k+j}$. We have
\begin{equation*}
	\bB^tH\bB=\Bigl(\sum_{\alpha=1}^{n-2\kappa}b_{n-2\kappa+1-\alpha,i}b_{\alpha,j}\Bigr)_{2\kappa\times 2\kappa}.
\end{equation*}
Substituting in terms of $T_i$ and $S_j$, we get
\begin{align}
\begin{split}\label{max_red:BtHB}
	\sum_{\alpha=1}^{n-2\kappa}b_{n-2\kappa+1-\alpha,i}b_{\alpha,j}={}&\sum_{\alpha=1}^{n-2\kappa} T_{n-2\kappa+1-\alpha} S_{n-2\kappa+i} T_\alpha S_{n-2\kappa+j},\\
	={}  &(S_{n-2\kappa+i}S_{n-2\kappa+j})\cdot \Bigl(\sum_{\alpha=1}^{n-2\kappa} T_{n-2\kappa+1-\alpha}T_\alpha\Bigr).	
\end{split}
\end{align}
The next simplification depends on whether we are in case (a) $1\leq q \leq n-2\kappa$ or (b) $n-2\kappa+1\leq q\leq n$.
	
\noindent\textbf{Case (a).} When $1\leq q\leq n-2\kappa$. 
Since $S_p$ is invertible by \eqref{max_red:AAt}, the equation \eqref{max_red:trAH} is equivalent to
\begin{equation*}
	T_1T_{n-2\kappa}+T_2T_{n-\kappa}+\cdots+T_{n-2\kappa}T_1=0.
\end{equation*}
Substituting this into \eqref{max_red:BtHB}, we see that all entries of $\bB^tH\bB$ vanish. 
Therefore, the affine coordinate ring of the principal open subset $D_+(x_{pq})$ is isomorphic to the factor ring
\begin{equation*}
	\frac{k[T_i,S_j]}{T_p-1,S_q-1,S_i-T_iS_p,T_1T_{n-2\kappa}+T_2T_{n-2\kappa-1}+\cdots+T_{n-2\kappa}T_1}.
\end{equation*}
Since we assume $\kappa<m-1$ when $n$ is even and $\kappa\leq m-1$ when $n$ is odd, we have $n-2\kappa>2$, and the element $\sum_{i=1}^{n-2\kappa}T_iT_{n-2\kappa+1-i}$ is irreducible. Therefore, any element of the form $T_iS_j$ is not a zero-divisor. This shows that $D_+(x_{pq})\cap D_+(x_{ij})\neq\emptyset$ for any $i,j$.
				
Replacing each $S_i$ by $T_iS_p$, the affine coordinate ring of $D_+(x_{pq})$ becomes isomorphic to
\begin{equation*}
	\frac{k[T_1,\cdots,T_{n-2\kappa},S_{n-2\kappa+1},\cdots,S_n,\frac{1}{T_q}]}{T_p-1,T_1T_{n-2\kappa}+\cdots+T_{n-2\kappa}T_1}.
\end{equation*}	
The further simplification now divides into the following cases:
\begin{altitemize}
	\item When $p\neq q$, and $2p\neq n-2\kappa$, the factor ring is isomorphic to
	\begin{equation*}
		k[T_1,\cdots,\wh{T_p},\cdots,\wh{T_{n-2\kappa+1-p}},\cdots,T_{n-2\kappa},S_{n-2\kappa+1},\cdots,S_n,\frac{1}{T_q}].
	\end{equation*}
Hence, the principal open subset $D_+(x_{pq})$ is an open subset of an affine space of dimension $n-2$, and in particular, it is irreducible.
					
	\item When $p\neq q$ and $2p=n-2\kappa$, the factor ring is isomorphic to
	\begin{equation*}
		\frac{k[T_1,\cdots,\wh{T_p},\cdots,T_{n-2\kappa},S_{n-2\kappa+1},\cdots,S_n,\frac{1}{T_q}]}{2(T_1T_{n-2\kappa}+\cdots+T_{p-1}T_{n-2\kappa+2-p})-1}.
	\end{equation*}
	Hence, the principal open subset $D_+(x_{pq})$ is a smooth irreducible variety of dimension $n-2$.
	\item When $p=q$ and $2p\neq n-2\kappa$, the factor ring is isomorphic to
	\begin{equation*}
		k[T_1,\cdots,\wh{T_p},\cdots,\wh{T_{n-2\kappa+1-p}},\cdots,T_{n-2\kappa},S_{n-2\kappa+1},\cdots,S_n].
	\end{equation*}
	Hence, the principal open subset $D_+(x_{pq})$ is an affine space of dimension $n-2$, and in particular, it is irreducible.
					
	\item When $p=q$, and $2p=n-2\kappa$, the factor ring is isomorphic to
	\begin{equation*}
		\frac{k[T_1,\cdots,\wh{T_p},\cdots,T_{n-2k},S_{n-2\kappa+1},\cdots,S_n]}{2(T_1T_{n-2\kappa}+\cdots+T_{p-1}T_{n-2\kappa+2-p})-1}.
	\end{equation*}
	Hence, the principal open subset $D_+(x_{pq})$ is isomorphic to a smooth irreducible hypersurface of dimension $n-2$ in an affine space.
\end{altitemize}

\noindent\textbf{Case (b).} When $n-2\kappa+1\leq q\leq n$. Since $S_q=S_{(n-2\kappa+(q-(n-2\kappa))}=1$, the $(q+2\kappa-n,q+2\kappa-n)$'s entry of \eqref{max_red:BtHB} yields 
\begin{equation}\label{max_red:TTsum}
	T_1T_{n-2\kappa}+\cdots+T_{n-2\kappa}T_1=0.
\end{equation}
This identity implies both \eqref{max_red:trAH} and the vanishing of all entreies in \eqref{max_red:BtHB}. 
Therefore, the affine coordinate ring of the principal open subset $D_+(x_{pq})$ is isomorphic to the factor ring:
\begin{equation*}
	\frac{k[T_j,S_i]}{T_p-1,S_q-1, S_i-T_iS_p,T_1T_{n-2\kappa}+T_2T_{n-2\kappa-1}+\cdots+T_{n-2\kappa}T_1}.
\end{equation*}
Note that the relation $S_i-T_iS_p$ also holds when $i=p$.
Since $n-2\kappa>2$, any element of the form $T_iS_j$ is not a zero-divisor. This implies that $D_+(x_{pq})\cap D_+(x_{ij})\neq\emptyset$ for any $i,j$.
Replacing each $S_i$ by $T_iS_p$, the affine coordinate ring is isomorphic to
\begin{equation*}
	\frac{k[T_1,\cdots,T_{n-2\kappa},S_{n-2\kappa+1},\cdots,\wh{S_q},\cdots,S_n,S_p]}{T_p-1,T_1T_{n-2\kappa}+T_2T_{n-2\kappa-1}+\cdots+T_{n-2\kappa}T_1}.
\end{equation*}
Similarly, the further simplification now divides into the following cases:
\begin{altitemize}
	\item When $2p\neq n-2\kappa$, the factor ring is isomorphic to
	\begin{equation*}
		k[T_1,\cdots,\wh{T_p},\cdots,\wh{T_{n-2\kappa+1-p}},\cdots,T_{n-2\kappa},S_{n-2\kappa+1},\cdots,\wh{S_q},\cdots,S_n,S_p].
	\end{equation*}
	Hence, the principal open subset $D_+(x_{pq})$ is an irreducible affine space of dimension $n-2$.

	\item When $2p=n-2\kappa$, the factor ring is isomorphic to
	\begin{equation*}
		\frac{k[T_1,\cdots,\wh{T_p},\cdots,T_{n-2\kappa},S_{n-2\kappa+1},,\cdots,\wh{S_q},\cdots,S_n,S_p]}{2(T_1T_{n-2\kappa}+\cdots+T_{p-1}T_{n-2\kappa+2-p})-1}.
	\end{equation*}
	Hence, the principal open subset $D_+(x_{pq})$ is isomorphic to a smooth irreducible hypersurface of dimension $n-2$ in affine space.
\end{altitemize}
In all cases, the principal open subsets $D_+(x_{pq})$ are irreducible of dimension $n-2$. Therefore, $\Spec R_{s,2}$ is irreducible of dimension $n-1$.

Next, consider the open subsets $D(x_{pq})\subset \Spec R_{s,2}$ for $1\leq p\leq n-2\kappa$ and $1\leq q\leq n$. The complement of the union $\bigcup_{p,q}D(x_{pq})$ corresponds to the worst point.
Using the same computation as we did for $D_+(x_{pq})$, we can show that each affine open $D(x_{pq})$ is smooth. 
In fact, the coordinate ring of $D(x_{pq})$ is obtained from that of $D_+(x_{pq})$ by adjoining the variables ``$x_{pq}$'' and ``$\frac{1}{x_{pq}}$''.
It follows that the singular locus of $\Spec R_{s,2}$ is contained in the closed point defined by $\bA=0,\bB=0$, which is singular by the Jacobian criterion. 
This completes the proof of part (i).

For part (ii), the same argument in part (i) also works for $\Spec R_{s,12}$.
\end{proof}

\begin{corollary}\label{max_red:sing-locus}
    The affine scheme $\Spec R_s$ is smooth outside $\Spec R_{s,12}$.
\end{corollary}
\begin{proof}
Using the same argument as Proposition \ref{max_red:irr}, we can show that $\Spec R_s-\{*\}$ is reduced. 
To be more precise, we compute the affine coordinate rings of the subsets $D(x_{pq})$ for all possible $p$ and $q$. Using the substitution in \eqref{max_red:AAt} and \eqref{max_red:trAH}, we find:
\begin{altitemize}
\item When $1\leq q\leq n-2\kappa$, the open subset $D(x_{pq})$ is isomorphic to
\begin{equation*}
    \Spec\frac{k[T_1,\cdots,T_{n-2\kappa},S_p,S_{n-2\kappa+1},\cdots,S_n,x_{pq},\frac{1}{x_{pq}}]}{T_p-1,S_pT_q-1,T_1T_{n-2\kappa}+T_2T_{n-2\kappa-1}+\cdots+T_{n-2\kappa} T_1};
\end{equation*}
\item When $n-2\kappa+1\leq q\leq n$, the open subset $D(x_{pq})$ is isomorphic to
\begin{equation*}
    \Spec\frac{k[T_1,\cdots,T_{n-2\kappa},S_p,S_{n-2\kappa+1},\cdots,S_n,x_{pq},\frac{1}{x_{pq}}]}{T_p-1,S_p(T_1T_{n-2\kappa}+T_2T_{n-2\kappa-1}+\cdots+T_{n-2\kappa} T_1)}.
\end{equation*}
\end{altitemize}
By \eqref{max_red:AAt} and \eqref{max_red:BtHB}, over each $D(x_{pq})$, the relation $\bA=0$ is equivalent to $S_p=0$, and $\bB^tH\bB=0$ is equivalent to $T_1T_{n-2\kappa}+T_2T_{n-2\kappa-1}+\cdots+T_{n-2\kappa} T_1=0$.
From these isomorphisms, we see that a point in $D(x_{pq})$ is singular if and only if $n-2\kappa+1\leq q\leq n$, $S_p=0$, and $T_1T_{n-2\kappa}+T_2T_{n-2\kappa-1}+\cdots+T_{n-2\kappa} T_1$, i.e., the point lies in the locus $D(x_{pq})\cap \Spec R_{s,12}$.
Therefore, $\Spec R_s-\{*\}$ is reduced, and $\Spec R_s-\{*\}$ is smooth.
\end{proof}

\begin{corollary}\label{max_red:lm}
    The affine scheme $\Spec R_s$ is reduced.
\end{corollary}
\begin{proof}
Recall Serre's criterion, which states that a ring $R$ is reduced if and only if it satisfies condition $(R_0)$ and $(S_1)$; cf.\ \cite[\href{https://stacks.math.columbia.edu/tag/031R}{031R}]{Stacks}.
In our case, $R_s$ is Cohen-Macaulay by Corollary \ref{max_red:cm}, and thus satisfies $(S_1)$.
Therefore, it remains to verify that $R_s$ satisfies $(R_0)$, i.e., it is regular in codimension $0$. This is equivalent to showing that the singular locus has codimension $\geq 1$, which follows from Corollary \ref{max_red:sing-locus}.
\end{proof}

\subsection{Geometry of the local model with strongly non-special level}
In this subsection, we will prove the remaining part of Theorem \ref{max_red:reduced_sym}.
\begin{proposition}\label{max_red:int}
The affine scheme $\Spec R_{s,12}$ is integral.
\end{proposition}
\begin{proof}
Since $\Spec R_{s,12}$ is irreducible by Proposition \ref{max_red:irr}(ii), it remains to show that it is reduced.
We will use Gr\"obner basis methods to show this.

To simplify notation, let $s:=n-2k$ and $t:=2k$, so that $\bB$ is an $s\times t$ matrix.
We choose the following lexicographic monomial order on the entries of $\bB$: $b_{ij}<b_{pq}$ if $i<p$ or $i=p$ and $j<q$. This is a diagonal monomial order, meaning that for each $2$-minor $[ij\mid pq]:=b_{ip}b_{jq}-b_{iq}b_{jp}$ ($i<j$, $p<q$), the initial monomial is the diagonal term $b_{ip}b_{jq}$.

By \cite[Theorem 4.1.1]{BCRV2022}, the $2$-minors $[ij\mid pq]$ form a Gr\"obner basis of the ideal $(\bigwedge^2 \bB)\subset k[\bB]$, and the initial ideal $\text{in}(\bigwedge^2\bB)$ is generated by the monomials $b_{ip}b_{jq}, i<j$, $p<q$.

The $(\alpha,\beta)$-entry ($\alpha\leq \beta$) of the matrix $\bB^tH\bB$ is given by $\sum_{\delta=1}^t b_{t+1-\delta,\alpha}b_{\delta,\beta}$. 
After reduction modulo the ideal generated by the $2$-minors, the remainder term $f_{\alpha\beta}$ is:
\begin{altitemize}
\item When $t=2r$ is even,
\begin{equation*}
    f_{\alpha\beta}=2\sum_{\delta=1}^r b_{t+1-\delta,\alpha}b_{\delta,\beta}.
\end{equation*}
\item When $t=2r+1$ is odd, 
\begin{equation*}
    f_{\alpha\beta}=2\left(\sum_{\delta=1}^r b_{t+1-\delta,\alpha}b_{\delta,\beta}\right)+b_{r+1,\alpha}b_{r+1,\beta}.
\end{equation*}
\end{altitemize}

We claim that the set of elements 
$$\{[ij\mid pq]:i<j,p<q\}\cup\{f_{\alpha\beta}:\alpha\leq \beta\}\subset k[\bB]$$
forms a Gr\"obner basis for the ideal $\left(\bigwedge^2\bB,\bB^tH\bB\right)$.
We will verify this in the case $t=2r+1$, the same argument applies when $t=2r$.

Given polynomials $f$ and $g$, define their $S$-polynomial as
\begin{equation*}
    S(f,g):=\frac{\lcm(\ini(f),\ini(g))}{\ini(f)}f-\frac{\lcm(\ini(f),\ini(g))}{\ini(g)}g.
\end{equation*}
Recall the Buchberger Criterion from \cite[Theorem 1.2.13]{BCRV2022}, which states that a set of generators $(f_1,\cdots,f_n)$ of an ideal $I$ forms a Gr\"obner basis if and only if all $S$-polynomials $S(f_i,f_j)$ reduce to $0$ modulo $f_1,\cdots,f_n$.
In our setting, it suffices to consider the following two cases:
\begin{altenumerate2}
\item $f=f_{\alpha\beta}$ and $g=[ij\mid pq]$, where $\ini(f)$ and $\ini(g)$ share a  common factor.
\item $f=f_{ij}$ and $g=f_{pq}$, where $\ini(f)$ and $\ini(g)$ share a common factor.
\end{altenumerate2}
We will compute the $S$-polynomial in each case.

\noindent\textbf{Case (a).} 
When $f=f_{\alpha\beta}$ and $g=[ij\mid pq]$, where $\ini(f)$ and $\ini(g)$ share a  common factor.
Since $\ini(f)=2b_{1\beta}b_{t\alpha}$ and $\ini(g)=b_{ip}b_{jq}$, the condition that they share a common factor implies that either $b_{ip}=b_{1\beta}$ or $b_{jq}=b_{t\alpha}$, i.e. $i=1,p=\beta$ or $j=t,q=\alpha$.

Consider the case $i=1,p=\beta$. Then $\alpha<\beta=p<q$, and we have
\begin{equation*}
    S(f,g)=2b_{t\alpha} b_{j\beta}b_{1q}+b_{jq}\left(2\left(\sum_{\delta=2}^r b_{t+1-\delta,\alpha}b_{\delta,\beta}\right)+b_{r+1,\alpha}b_{r+1,\beta}\right),
    \quad
    \ini(S(f,g))=2b_{t\alpha}b_{1q}b_{j\beta}.
\end{equation*}
Note that
\begin{equation*}
    S(f,g)-b_{j\beta}f_{\alpha q}=
    \left(\sum_{\delta=2}^r 2b_{t+1-\delta,\alpha}(b_{jq}b_{\delta\beta}-b_{j\beta}b_{\delta q})\right)+b_{r+1,\alpha}(b_{jq}b_{r+1,\beta}-b_{j\beta}b_{r+1,q}),
\end{equation*}
where $\ini(b_{j\beta}f_{\alpha q})=\ini(S(f,g))$, and each term $\ini(b_{t+1-\delta,\alpha}(b_{jq}b_{\delta\beta}-b_{j\beta}b_{\delta q}))<\ini(S(f,g))$.
Therefore, the $S$-polynomial $S(f,g)$ reduces to zero modulo the chosen generators.
A Similar computation applies to the case $j=t, q=\alpha$.

\noindent\textbf{Case (b).} When $f=f_{ij}$ and $g=f_{pq}$, where $\ini(f)$ and $\ini(g)$ share a common factor.
Since $\ini(f_{ij})=2b_{ti}b_{1j}$ and $\ini(f_{pq})=2b_{tp}b_{1q}$, this assumption implies that either $i=p$ or $j=q$.
When $i=p$, we have
\begin{equation*}
    S(f_{ij},f_{pq})=2\left(\sum_{\delta=2}^r b_{t+1-\delta,p}(b_{1q}b_{\delta j}-b_{1j}b_{tq})\right)+b_{r+1,p}(b_{1q}b_{r+1,j}-b_{1j}b_{r+1,q}).
\end{equation*}
Each term in the sum lies in the ideal generated by the $2$-minors $[ij \mid pq]$.
Hence, $S(f,g)$ reduces to zero modulo the chosen geneators.
A similar computation works for the case when $j=q$.

Therefore, the set of elements $\{[ij\mid pq]:i<j,p<q\}\cup\{f_{\alpha\beta}:\alpha\leq \beta\}$ forms a Gr\"obner basis of the ideal $\left(\bigwedge^2\bB,\bB^tH\bB\right)$. In particular, the initial ideal $\ini(I)$ is generated by monomials of the form $b_{ij}b_{pq}$ with $(i,j)\neq (p,q)$. 

Now, suppose $0\neq \ov{b}\in \frac{k[\bB]}{\bigwedge^2 \bB,\bB^tH\bB}$ is a nilpotent element.
Then $\ov{b}^N=0$ for some $N$, and hence $\ini(b)^N\in \ini\left(\bigwedge^2\bB,\bB^tH\bB\right)$. This is impossible unless $N=1$ since the latter is generated by square-free quadratic monomials.
\end{proof}

\begin{proposition}\label{max_red:irr-red}
The affine schemes $\Spec R_{s,1}$ and $\Spec R_{s,2}$ are integral.
\end{proposition}
\begin{proof}
The affine scheme $\Spec R_{s,1}$ is integral by the theory of determinant varieties, see \cite[(2.11)]{BU1998}. 

To show that $\Spec R_{s,2}$ is integral, we retain the notation from Lemma \ref{max_simp:simp-step2_comp-radical}. 
By Remark \ref{max_simp:simp-step2-simp-comp}, we have the following isomorphisms:
\begin{equation*}
    R_{s}\simeq \frac{k[X]}{IJ+K},\quad
    R_{s,1}\simeq \frac{k[X]}{I+K},\quad
    R_{s,2}\simeq \frac{k[X]}{J+K},\quad
    R_{s,12}\simeq \frac{k[X]}{I+J+K}.
\end{equation*}
as described in Theorem \ref{sm:simp_step3_ring}. We know that:
$IJ+K$ is radical by Corollary \ref{max_red:lm},
$I+K$ is prime as shown above,
$I+J+K$ is prime by Proposition \ref{max_red:int},
$\sqrt{J+K}$ is prime by Proposition \ref{max_red:irr}(i).
We want to show that $J+K$ itself is a prime ideal. 

Suppose $f^n\in J+K\subset I+J+K$. Then $f\in I+J+K$, so we may write $f=\alpha+\beta$ with $\alpha\in I$ and $\beta\in J+K$. 
Then $f^n=\alpha^n+\beta\gamma$ for some $\gamma\in k[X]$, hence $\alpha^n\in J+K$.
Since $\alpha^{n+1}\in I(J+K)\subset IJ+K$, we have $\alpha\in IJ+K\subset J+K$.
Therefore, $f=\alpha+\beta\in J+K$ and $J+K=\sqrt{J+K}$ is a prime ideal.
\end{proof}

\begin{corollary}\label{max_red:irr-comp}
\begin{altenumerate}
    \item The affine schemes $\Spec R_{s,1}$ and $\Spec R_{s,2}$ are the irreducible components of $\Spec R_{s}$. They are normal and Cohen-Macaulay, with singular locus $\{*\}$.
    \item The affine scheme $\Spec R_{s,12}$ is the intersection of the irreducible components of $\Spec R_{s}$. It is normal and Cohen-Macaulay, with singular locus $\{*\}$.
\end{altenumerate}
\end{corollary}
\begin{proof}
\begin{altenumerate}
\item Recall from Lemma \ref{max_simp:simp-step2_comp} that $(\Spec R_{s})^\red=(\Spec R_{s,1})^\red\cup (\Spec R_{s,2})^\red$. 
By Proposition \ref{max_red:irr-red} and Corollary \ref{max_red:lm}, we may remove the ``red''. Therefore, $\Spec R_{s,1}$ and $\Spec R_{s,2}$ are the irreducible components of $\Spec R_s$. By Theorem \ref{moduli_fc:coh} they are Cohen-Macaulay. They are also normal, since their singular locus, defined by $\bA=0,\bB=0$, has dimension $0$.

\item Recall that for a noetherian scheme $Z$ with two irreducible Cohen-Macaulay components $X$ and $Y$, if the intersection $X\cap Y$ has codimension $1$ in both $X$ and $Y$, then $X\cap Y$ is Cohen-Macaulay if and only if $Z=X\cup Y$ is Cohen-Macaulay; cf.\ \cite[Lemma 4.22]{Gortz2001}.

Since $\Spec R_{s}$ is Cohen-Macaulay by Theorem \ref{max_red:reduced_sym}, and $\Spec R_{s,1}$ and $\Spec R_{s,2}$ are Cohen-Macaulay by (i), it follows from \cite[Lemma 4.22]{Gortz2001} that $\Spec R_{s,12}$ is also Cohen-Macaulay.
\end{altenumerate}
\end{proof}

Now, we have finished the proof of Theorem \ref{max_red:reduced_sym}: 
part (i) follows from Corollary \ref{max_red:sing-locus} and Corollary \ref{max_red:lm};
part (ii)(iii) follow from Corollary \ref{max_red:irr-comp}.

\begin{remark}\label{max_red:intersection}
Since the intersection of the irreducible components $\Spec R_{s,12}$ is singular with the singular point defined by $\bA=0,\bB=0$, in particular, it is not regular.
Therefore, the local model does not admit semi-stable reduction.
\end{remark}

\begin{corollary}\label{max_red:normal-cohen-macaulay}
The local model $M^\loc_{\{\kappa\}}(n-1,1)$ for strongly non-special $\{\kappa\}$ over the ring of integers is normal and Cohen-Macaulay. It is Gorenstein when $n=4\kappa+2$. 
\end{corollary}
\begin{proof}
Recall that a flat scheme of finite type $Y$, over a discrete valuation ring with normal generic fiber and reduced special fiber, is normal, cf.\ \cite[Proposition 9.2]{PZ2013}.
For us, since the generic fiber of the local model is smooth, the normality follows. 
In particular, $\pi\in\CO_F$ is not a zero-divisor. Therefore, by \cite[\href{https://stacks.math.columbia.edu/tag/02JN}{02JN},\href{https://stacks.math.columbia.edu/tag/0BJJ}{0BJJ}]{Stacks}, the local model $M_I^\loc$ is Cohen-Macaulay (resp.\ Gorenstein) if and only if its special fiber $M^\loc_{I,s}$ is Cohen-Macaulay (resp.\ Gorenstein).
\end{proof}

\begin{remarks}\label{max_red:recover}
We can also recover several known results from Theorem \ref{sm:simp_step3_ring}.
\begin{altenumerate}
\item When $\kappa=0$, the matrix $X_3$ has size $0$. Hence the ring becomes $\frac{k[X_4]}{\wedge^2(X_4), X_4-X_4^{\ad},\tr(X_4)}$, this recovers Pappas's result \cite{Pappas2000} (but note that Pappas uses a different choice of hermitian form).
				
	\item When $\kappa=m$ and $n=2m+1$, the matrices $X_3$, resp.,\ $X_4$ are of size $1\times (n-1)$ and $1\times 1$. The equation $\tr(X_4)=0$ implies $X_4=0$, and the ring $R=\frac{k[X_3,X_4]}{\bigwedge^2 (X_3,X_4), X_4-X_4^{\ad},\tr(X_4)}$ is isomorphic to $k[X_3]$, this recovers Smithling's computation \cite[3.10, 3.11]{Smithling2015}.
				
	\item When $\kappa=m-1$ and $n=2m$, the equations $\tr(X_4)=0$ and $X_4=X_4^{\ad}$ are equivalent to $X_4=\left(\begin{matrix}  & * \\ * &\end{matrix}\right)$, hence the ring is isomorphic to the ring
	\begin{equation*}
		\frac{k[X_3,x_1,x_2]}{\bigwedge^2\left(X_3\quad\left(\begin{matrix}  & x_2 \\ x_1 &\end{matrix}\right)\right)}\simeq \frac{k[X_3,x_1,x_2]}{x_1x_2,\bigwedge^2X_3,\left(\begin{matrix}  x_2 & \\ & x_1\end{matrix}\right)X_3}.
	\end{equation*}
	This recovers the case studied by Yu \cite[1.4.22, 1.4.23, 1.4.24]{Yu2019}.

	\item We do not expect to recover the defining equations in the $\pi$-modular case \cite{RSZ2017} ($I=\{m\}$ for $n=2m$) since then both $X_3, X_4$ have size $0$. This is not surprising, as this is the only case where the worst point does not lie in the local model; cf.\ \cite[Remark 7.4]{PR2009}.
\end{altenumerate}
\end{remarks}

\section{Reducedness for the general parahoric subgroup}\label{general}
\subsection{Recollection of the (partial) affine flag varieties}
In this subsection, we quickly review the theory of (partial) affine flag varieties of the ramified unitary group, and its relation to the local model.
		
\subsubsection{Local Dynkin diagram of quasi-split unitary groups}\label{general_dyn}
For the remaining part of the section, we will use the local Dynkin diagram, rather than the convention \eqref{intro_back:local-dynkin} used in \cite[4.a]{PR2008}.
When $n=2m+1$, the local Dynkin diagram takes the form
\begin{equation*}
	\xy 
	(-10,0)*{s}="s1";
	(-5,0)*{\bullet};
	(0,0)*{\bullet};
	(5,0)*{\bullet};
	(10,0)*{\bullet};
	(15,0)*{s}="s2";
	(-7.5,0)*{<};
	(12.5,0)*{<};
	{\ar@2{-} "s1";(-5,0)};
	{\ar@{-} (-5,0);(0,0)};
	{\ar@{.} (0,0);(5,0)};
	{\ar@{-} (5,0);(10,0)};
	{\ar@2{-} (10,0);"s2"};
	\endxy.
\end{equation*}
We label the vertices from left to right as $0,1,\cdots,m$.
		
When $n=2m$, the local Dynkin diagram takes the form 
\begin{equation*}
	\xy 
	(-10,0)*{\bullet}="v1";
	(-5,0)*{\bullet}="v2";
	(0,0)*{\bullet}="v3";
	(5,0)*{\bullet}="v4";
	(10,0)*{\bullet}="v5";
	(15,3)*{s}="s1";
	(15,-3)*{s}="s2";
	(-7.5,0)*{>};
	{\ar@2{-} (-10,0);(-5,0)};
	{\ar@{-} (-5,0);(0,0)};
	{\ar@{.} (0,0);(5,0)};
	{\ar@{-} (5,0);(10,0)};
	{\ar@{-} (10,0);"s1"};
	{\ar@{-} (10,0);"s2"};
	\endxy.
\end{equation*}
A special vertex $m'$ corresponds to the lattice
\begin{equation*}
	\Lambda_{m'}:=\sum_{i=1}^{m-2}\pi^{-1}\CO_F e_i+\pi^{-1}\CO_F e_{m-1}+\pi^{-1}\CO_{F}e_{m+1}+\sum_{j=m+2}^n\CO_F e_{j},
\end{equation*}
using the notation from \S \ref{moduli_setup_alg-group}. The remaining vertices, ordered from left to right, are labeled by ${0,1,\cdots,m-2,m}$.
The parahoric subgroups associated with $m'$ and $m$ are isomorphic under diagram automorphism. 
Moreover, the parahoric subgroup corresponding to the index set $\{m-1,m\}$ in \S \ref{moduli_setup_alg-group} coincides with the one generated by the vertices $\{m',m\}$.
When working with the local Dynkin diagram, the associated lattices do not form a lattice chain as defined in the sense of \cite[Definition 3.1]{RZ}; see also \ \S \ref{application_nonsplit_moduli}. However, this issue is not relevant for the current section.
For the remaining part of the section, the index set $\CI$ will always refer to a non-empty subset of vertices of the local Dynkin diagram. 
We will continue to refer the associated periodic lattice as a ``lattice chain'' indexed by $\CI$. 
Besides, the various moduli functors, $M_\CI^\square$, remain well-defined in this context.

\subsubsection{Unitary affine flag varieties}
Let $F_0=k((t))$ be the field of Laurent power series with coefficients in $k$, and $F=k((u))$ with uniformizer $u$ such that $u^2=t$. We choose the split hermitian form $\phi$ as in \S \ref{moduli_setup_notation} and associate with it the unitary similitude group $G=\GU(F^n,\phi)$.
        
We briefly recall some basic facts about the unitary similitude affine flag variety. The precise details can be found in \cite[Section 3.2]{PR2009}. 
For any non-empty subset $\CI$ of vertices of the local Dynkin diagram, Pappas and Rapoport define a moduli functor $\Fl_\CI$, which is represented by an ind-scheme called \emph{unitary affine flag variety}, which parametrizes the set of periodic self-dual ``lattice chains'' subject to certain conditions.
		
Note that the functor $\Fl_\CI$ does not, in general, describe the same affine flag variety as described in \cite{PR2008}; rather, it corresponds to the ind-scheme $L\CG/L^+\CP_\CI$.

In what follows, we will apply the Frobenius-splitting properties of the Schubert varieties in the $\Fl_\CI$, which descend from those in $L\CG/L^+\CP_\CI^\circ$, cf.\ \cite[\S 3.3]{PR2009}.

Using the moduli description, the special fiber $M_{\CI,s}^\naive$ of the naive local model admits an embedding into the unitary affine flag variety $\Fl_\CI$; cf.\ \cite[\S 3.2]{PR2009}:
\begin{equation*}
	M^\loc_{\CI,s}\subset M_{\CI,s}\subset M^\naive_{\CI,s}\hookrightarrow \Fl_\CI.
\end{equation*}
This embedding is equivariant with respect to the action of the positive loop group $L^+\CP_\CI$. As a result, their reduced loci are unions of affine Schubert varieties.

This embedding arises functorially from the constructions. In other words, for any inclusion $\CI'\subset \CI$, we have a commutative diagram:
\begin{equation*}
\begin{aligned}
\xymatrix{
M^\loc_{\CI,s}\ar@{^(->}[r]\ar[d]	&	M_{\CI,s}\ar@{^(->}[r]\ar[d]	&	\Fl_\CI\ar[d]\\
M^\loc_{\CI',s}\ar@{^(->}[r]	&	M_{\CI',s}\ar@{^(->}[r]	&\Fl_{\CI'}.
}
\end{aligned}
\end{equation*}
		
The coherence conjecture implies that $M^\loc_{\CI,s}$ is in fact the union of Schubert varieties indexed by the admissible set $\Adm^\CI(\mu)$, where $\mu$ is the minuscule cocharacter corresponding to signature $(n-1,1)$; cf.\ \cite[Theorem 11.3]{PR2008} and \cite[Theorem 8.1]{Zhu2014}. In particular, $M_{\CI,s}^\loc$ is reduced; cf.\ Theorem \ref{moduli_fc:coh}.

\subsection{Flatness when $p\geq 3$}\label{general_split-case}	
Assume $p\geq 3$ in this subsection.
Fix $\CI$ a non-empty subset of vertices of the local Dynkin diagram. Consider the natural embedding
\begin{equation}\label{general_pull-back}
	M_{\CI,s}^\loc\subset M_{\CI,s}\hookrightarrow\Fl_\CI.
\end{equation}
For any $\tau\in \CI$, we define $\wt{M}_{\CI,\tau}$ to be the pull-back along the affine flag varieties
\begin{equation*}
\begin{aligned}
\xymatrix{
\wt{M}_{\CI,\tau}\ar@{}[rd]|{\square}\ar@{^(->}[r]\ar[d]_{\pi_{\CI,\{j\}}}&\Fl_\CI\ar[d]\\
M_{\{\tau\},s}\ar@{^(->}[r]&\Fl_{\{\tau\}}.
}
\end{aligned}
\end{equation*}

\begin{proposition}\label{general:intersection-of-tilde-M}
The special fiber $M_{\CI,s}$ of the strengthened spin model, viewed as a closed subscheme of $\Fl_\CI$, is equal to the scheme-theoretic intersection $\bigcap_{i\in \CI}\wt{M}_{\CI,\{i\}}\subset \Fl_\CI$.
\end{proposition}
\begin{proof}
The case where $m'\in \CI$ will be addressed at the end of the proof. For now, assume $m'\not\in \CI$.
In the pull-back diagram \eqref{general_pull-back}, the schemes $M_{\{\tau\},s}$, $\Fl_{\{\tau\}}$, and $\Fl_\CI$ all admit moduli interpretations. Using these descriptions, we identify the pull-back $\wt{M}_{\CI,\tau}$ with the moduli functor that assigns to each $k$-algebra $R$ the set of $R[[u]]$-lattice chains indexed by $\pm \CI+n\BZ$:
\begin{equation}\label{general_lattice-chain}
    \cdots\subset L_{i_1}\subset L_{i_2}\subset\cdots\subset L_{i_N}\subset u^{-1}L_{i_1}=L_{i_{1}+n}\subset\cdots
\end{equation}
in $\Fl_\CI$, such that the subchain $\{L_j\}_{j\in\pm\CI+n\BZ}$ satisfies the following conditions:
\begin{altenumerate}
\item There is a natural inclusion of lattices in $R((u))^n$:
\begin{equation*}
    u\left(\lambda_j\otimes_{k[[t]]}R[[t]]\right)\subset L_j\subset u^{-1}\lambda_j\otimes_{k[[t]]}R[[t]].
\end{equation*}

\item Via the identification 
\begin{equation*}
    u^{-1}\lambda_j\otimes_{k[[t]]}R[[t]]\otimes_{R[[t]]} R\simeq u^{-1}\lambda_j\otimes_{k[[t]]}R\simeq \Lambda_j\otimes_{\CO_{F_0}}R,
\end{equation*}
the lattices $L_j$ correspond to a chain of filtrations $\{\CF_j\subset \Lambda_j\otimes_{\CO_{F_0}}R\}_{j\in\pm\CI+n\BZ}$, satisfying the axioms of the strengthened spin model.
\end{altenumerate}

From this moduli description of $\wt{M}_{\CI,\tau}$, it follows that the scheme-theoretic intersection $\bigcap_{i\in\CI}\wt{M}_{\CI,i}$ parametrizes those lattice chains in \eqref{general_lattice-chain} that satisfy both conditions (i) and (ii). This coincides with the moduli functor of the strengthened spin model $M_{I,s}$.

Finally, consider the case when $m'\in \CI$. If $m\not\in \CI$, then the same argument applies after replacing $m'$ with $m$ via the diagram automorphism. 
When $\{m,m'\}\subset\CI$, the lattice chain takes the form:
\begin{equation*}
\xy
(-17,7)*+{\cdots\subset L_\bullet\subset L_\bullet};
(-4.5,3.5)*+{\rotatebox{-45}{$\,\, \subset\,\,$}};
(-4.5,10.5)*+{\rotatebox{45}{$\,\, \subset\,\,$}};
(0,14)*+{L_m};
(0,0)*+{L_{m'}};
(4,3.5)*+{\rotatebox{45}{$\,\, \subset\,\,$}};
(4,10.5)*+{\rotatebox{-45}{$\,\, \subset\,\,$}};
(19,7)*+{L_\bullet \subset  L_\bullet\subset\cdots.}
\endxy
\end{equation*}
Except that, the same argument follows.

\end{proof}

\begin{theorem}\label{general:p-geq-3}
    When $p\geq 3$, Theorem \ref{intro_back:main} is true.
\end{theorem}
\begin{proof}
Since $\CP_{\{m-1,m\}}^{\circ}=\CP^{\circ}_{\{m,m'\}}$, we have a group-theoretic identification of affine flag varieties $L\CG/L^+\CP_{\{m-1,m\}}^{\circ}\cong L\CG/L^+\CP_{\{m,m'\}}^{\circ}$. By the coherent conjecture, this induces a natural isomorphism between the corresponding local models $M^{\loc}_{\{m-1,m\}}\cong M^{\loc}_{\{m',m\}}$, with both identified as unions of Schubert cells indexed by admissible sets. 
By Remark \ref{max_red:recover}, together with \cite[Proposition 9.12]{RSZ2018}, or alternatively by \cite{Yu2019}, we have $M^{\loc}_{\{m-1,m\}}=M_{\{m-1,m\}}$. 
By Proposition \ref{general:intersection-of-tilde-M}, we also have $M^{\loc}_{\{m,m'\}}=M_{\{m,m'\}}$. Hence, we obtain an abstract isomorphism $M_{\{m-1,m\}}\cong M_{\{m,m'\}}$.
It therefore suffices to prove the theorem for the index set $\CI$ interpreted as a subset of the local Dynkin diagram.

By Theorem \ref{max_red:reduced_sym} in the strongly non-special case, and \S \ref{intro_moduli} for the other cases, the special fiber $M_{\{j\},s}$ of the strengthened spin model is reduced.
Therefore, the inclusion $M_{\{j\},s}^\loc=M_{\{j\},s}\hookrightarrow\Fl_\CI$ is a union of Schubert varieties.
It follows that the pullbacks $\wt{M}_{\CI,\{j\}}\subset \Fl_\CI$ are also unions of Schubert varieties.
Since the intersection 
$M_\CI=\bigcap_{i\in \CI}\wt{M}_{\CI,\{i\}}\subset \Fl_\CI$ is an intersection of unions of Schubert varieties, it is compatiblely-Frobenius split by \cite[Corollary 2.7]{Gortz2001}, and hence reduced.
\end{proof}

\begin{corollary}\label{general:wedge-from-ss}
For the ramified unitary local model of signature $(n-1,1)$, the strengthened spin condition {\fontfamily{cmtt}\selectfont LM8} implies the wedge condition {\fontfamily{cmtt}\selectfont LM6}.
\end{corollary}
\begin{proof}
Since the strengthened spin models are flat, it suffices to verify the claim over their special fiber. By Proposition \ref{general:intersection-of-tilde-M}, this further reduces to the maximal parahoric cases.

Notice that the computations in Proposition \ref{equ_ss-comp:g_S-dual} hold for all indices $0\leq \kappa\leq m$. Therefore, a similar basis as in Corollary \ref{equ_ss-comp:basis} can be constructed for self-dual, $\pi$-modular, and almost $\pi$-modular cases.
When $\kappa$ is self-dual or almost $\pi$-modular, we recover the same basis as in the Corollary \ref{equ_ss-comp:basis}. In the almost $\pi$-modular case, this agrees with \cite[Corollary 4.14]{Smithling2015}.
In particular, all basis elements $e_S$ correspond to $S$ of type $(1,n-1)$ or $(0,n)$. 
Thus, in the affine chart centered at the worst point, all $2\times 2$-minors of $X$ vanish, and the wedge condition is automatically satisfied; cf.\ Proposition \ref{equ_coord_translation:not-show-up}.

In the $\pi$-modular case, only case (iv) of Proposition \ref{equ_ss-comp:g_S-dual} occurs, so the element $e_{\{n+1,\cdots,2n\}}$ does not appear in the basis. As a result, the worst point does not lie in the local model; cf.\ \cite[Remark 7.4]{PR2009}.
Nevertheless, the remaining elements in Corollary \ref{equ_ss-comp:basis} still form a basis.
In this case, one considers the open chart around the ``best point'', cf.\ \cite[\S 5.3]{PR2009}.
Since all basis elements $e_S$ are of type $(1,n-1)$, the wedge condition again follows from the strengthened spin condition.
\end{proof}
		
From Proposition \ref{general:intersection-of-tilde-M} one can deduce
\begin{corollary}\label{general:lm-vertex}
The local model is equal to the following scheme-theoretical intersection
\begin{equation*}
    M^\loc_{\CI,s}=\bigcap_{i\in \CI}\pi^{-1}_{\CI,\{i\}}(M^\loc_{\{i\},s}),
\end{equation*} 
where $\pi_{\CI,\{i\}}:M^\loc_{\CI,s}\rightarrow M^\loc_{\{i\},s}$ is the natural projection.\qed
\end{corollary}
		
Since the special fiber $M^\loc_{\CI,s}$ of the local model is identified with the union of Schubert varieties indexed by $\text{Adm}^\CI(\mu)$, we obtain the following group-theoretic relation:
		
\begin{corollary}\label{general:vertex-wise}
We have the following identification:
\begin{equation*}
	\text{\emph{Adm}}^\CI(\mu)=\bigcap_{i\in \CI}\pi_{\CI,\{i\}}^{-1}(\text{\emph{Adm}}^{\{i\}}(\mu)).\eqno\qed
\end{equation*}
\end{corollary}
This proves the vertex-wise conjecture in \cite[\S 4.2]{PR2009} for the ramified unitary group of signature $(n-1,1)$. Haines and He have proven the conjecture for general groups using purely combinatorial methods, cf.\ \cite{HH2017}.
It is worth noting that one of the original motivations for the vertex-wise conjecture was to prove Corollary \ref{general:lm-vertex}; cf.\ \cite[Proposition 4.5]{PR2009} and \cite[Remark 8.3]{HH2017}.

\subsection{Flatness of the local model when $p=0$}\label{general_char-0}
In this subsection, we further show that the strengthened spin model $M_I$ is flat when residue field $k$ has characteristic zero. We follow the strategy of G\"ortz \cite[Theorem 4.18]{Gortz2001}.
We construct a standard model that universally represents the special fiber of  the strengthened spin model $M_I$.

Let $\BZ':=\BZ[1/2]$. 
For each $0\leq i\leq n$, define $\Lambda_i$ to be the free $\BZ'$-module with basis $e^{[i]}_1,\cdots,e^{[i]}_n,f^{[i]}_1,\cdots,f^{[i]}_n$.
We identify $\Lambda_n$ with $\Lambda_0$ throughout.
For each $i$, we define a nilpotent operator
\begin{equation*}
    \Pi:\Lambda_i\longrightarrow \Lambda_i,
    \quad
    e^{[i]}_j\mapsto f^{[i]}_j, f^{[i]}_j\mapsto 0, \text{ for any } 1\leq j\leq n.
\end{equation*}
For any $i\leq j$, we define transition maps
\begin{equation*}
    \varphi_{ij}:\Lambda_i\longrightarrow \Lambda_j,
    \quad\text{and}\quad
    \varphi_{ji}:\Lambda_j\longrightarrow\Lambda_i,
\end{equation*}
given by the following matrices:
\begin{equation*}
    \varphi_{ij}=\left(\begin{matrix}
    I_i&0&0&0&0&0\\
    0&0&0&0&0&0\\
    0&0&I_{j^\vee}&0&0&0\\
    0&0&0&I_i&0&0\\
    0&I_{j-i}&0&0&0&0\\
    0&0&0&0&0&I_{j^\vee}
    \end{matrix}\right),
    \quad\text{and}\quad
    \varphi_{ji}=\left(\begin{matrix}
    0&0&0&0&0&0\\
    0&I_{j-i}&0&0&0&0\\
    0&0&0&0&0&0\\
    I_i&0&0&0&0&0\\
    0&0&0&0&I_{j-i}&0\\
    0&0&I_{j^\vee}&0&0&0
    \end{matrix}\right).
\end{equation*}
Note that $\varphi_{ij}$ also represents the transition maps $\Lambda_{i,k}\rightarrow \Lambda_{j,k}$ in \S \ref{moduli_setup_notation}. Similarly, $\varphi_{ji}$ corresponds to  the transition maps $\Lambda_{j,k}\rightarrow\Lambda_{i+n,k}$ in \S \ref{moduli_setup_notation}.
We define $\varphi_{0n}$ and $\varphi_{n0}$ to be the identity map.

For $0\leq i\leq n/2$, we define a perfect ``symmetric'' pairing $\sform:\Lambda_{n-i}\times \Lambda_{i}\rightarrow \BZ'$ by setting
$$(e^{[n-i]}_\tau, f^{[i]}_{n+1-\tau} )=1\quad\text{and}\quad( f^{[n-i]}_\tau, e^{[i]}_{n+1-\tau} )=1$$
for all $1\leq \tau\leq n$, and zero on all other pairs.

For any $0\leq i\leq n/2$, let $W(\Lambda_{i}):=\bigwedge^n\Lambda_{i}$ denote the $n$-th wedge product. We define $L_i\subset W(\Lambda_i)$ to be the free submodule generated by the basis given in Corollary \ref{equ_ss-comp:basis}.

For any nonempty subset $I\subset \{0,1,\cdots,m\}$ satisfying \eqref{intro_back:local-dynkin}, we define the \emph{standard model (of the special fiber)} $\BM^s_I$ as a projective scheme over $\Spec \BZ'$. It represents the moduli functor that assigns to each $\BZ'$-algebra $R$ the set of families:
\begin{equation*}
    (\sF_i \subset \Lambda_{i,R})_{i\in I\sqcup (n-I)},
\end{equation*}
satisfying the following conditions:
\begin{altenumerate}		
\item For all $i$, $\sF_i$ is a $\Pi$-stable submodule, and a $R$-direct summand of $\Lambda{i,R}$of rank $n$.
\item For all $i,j$, the natural map $\varphi_{ij,R}:\Lambda_{i,R} \to \Lambda_{j,R}$ carries $\sF_i$ into $\sF_j$.
\item For all $i$, the $R$-bilinear pairing
\begin{equation*}
\Lambda_{i,R} \times \Lambda_{n-i,R}
\xra{\sform \otimes R} R
\end{equation*}
identifies $\sF_i^\perp$ with $\sF_{n-i}$ inside $\Lambda_{n-i,R}$.
\item  For any $i\in I$, the line $\bigwedge^n\CF_i\subset W(\Lambda_i)_R$ lies in the submodule $L_{i,R}$.
\end{altenumerate}

Given any strengthened spin model $M_{I}$ with the special fiber $M_{I,s}$ over $k$,
we define the following morphism: 
\begin{align*}
    f_s:M_{I,s}\longrightarrow {}&\BM_I^s\times_{\Spec\BZ'}\Spec k=:\BM^s_{I,k},\\
    (\CF_i\subset \Lambda_{i,R})_{\pm I+n\BZ}\mapsto{}&(\CF_i\subset \Lambda_{i,R})_{I \sqcup (n-I)}.
\end{align*}
Here, the $\pi\otimes 1$-action on the LHS is replaced by the $\Pi$-action on the RHS. 
Since the notations on both sides are compatible, the map $f_s$ is well-defined.
By construction of the standard model, it is clear that $f_s$ is an isomorphism.

By Theorem \ref{general:p-geq-3}, when $\text{char}(k)=p\geq 3$, the fiber $\BM^s_{I,k}\simeq M_{I,s}$ is reduced.
Applying the same arguments as in \cite[Theorem 4.18]{Gortz2001}, we have:
\begin{proposition}\label{general_char-0_generic}
The generic fiber $\BM^s_{I,\eta}$ of the standard model is reduced.\qed
\end{proposition}
As a consequence, we obtain the following result:
\begin{theorem}
\begin{altenumerate}
\item For $p\neq 2$, the special fiber $M_{I,s}$ of the strengthened spin model is reduced.
\item For $p\neq 2$, the strengthened spin model $M_I$ is flat, and thus represents the local model $M_I^\loc$.
\end{altenumerate}
\end{theorem}
\begin{proof}
\begin{altenumerate}
\item When $p\geq 3$, this follows from Theorem \ref{general:p-geq-3}. 
When $p=0$, the special fiber $M_{I,s}$ embeds into the generic fiber of the standard model. Its reducedness then follows from Proposition \ref{general_char-0_generic}.
\item This follows from the topological flatness established in Theorem \ref{moduli_fc:top-flat}.
\end{altenumerate}
\end{proof}
\begin{remark}
By the same arguments, Corollaries \ref{general:lm-vertex} – \ref{general:vertex-wise} also hold when $p=0$.
\end{remark}
		
\section{Consequences and applications}\label{application}
\subsection{Moduli description of the general unitary local model}\label{application_nonsplit}
In this subsection, we provide a moduli description of the Pappas–Zhu local model for any ramified unitary group with signature $(n-1,1)$.
In particular, this includes cases when the hermitian form is not split, and thus may yield a non-quasi-split unitary group.
Even in the split case, this description is intrinsic — it does not rely on a choice of basis for the vector space or on a choice of uniformizer in the ramified quadratic extension.
As a result, it offers greater flexibility for applications.
While the strategy used here is well known and has appeared in the work of various experts, we include a detailed exposition due to the lack of explicit references.
		
\subsubsection{Hermitian forms over a local field}
Suppose $F/F_0$ is a ramified extension of $p$-adic fields ($p\neq 2$) with uniformizers $\pi\in F$ and $\pi_0\in F_0$, respectively, such that $\pi^2=\pi_0$. 
Let $V$ be a $F/F_0$-hermitian space of dimension $\dim_F V\geq 3$, equipped with a hermitian form $\phi:V\times V\rightarrow F$. We keep this notation throughout.
		
For each dimension $n$, there are two isomorphism classes of hermitian spaces, denoted $V^+$ and $V^-$, distinguished by their Hasse invariant.
The split hermitian space $V^+=F^n$ is equipped with the hermitian form defined by $h_n(x e_i,y e_j)=\bar{x}y\delta_{i,n+1-j}$. The associated unitary similitude group $\GU(V^+,h_n)$ is quasi-split.

The non-split space is $V^{-}=F^{n-2}\oplus F^2$, equipped with the non-split hermitian form $h_{n}^-=h_{n-2}\oplus h_\alpha$, where $h_{n-2}$ is the split hermitian form on $F^{n-2}$, and $h_\alpha((x_1,x_2),(y_1,y_2))=\overline{x_1}y_1- \alpha\overline{x_2}y_2$ for some $\alpha\in F_0^\times$ representing the nontrivial class in $F_0^\times/\text{Norm}(F^\times)\simeq \BZ/2\BZ$.

When $n=2m$ is even, the non-split hermitian space $(V^-,h_n^-)$ gives rise to a non-quasi-split unitary similitude group $\GU(V^-,h_n^-)$. In this case, the hermitian space does not contain any $\pi$-modular vertex lattice. The vertex lattice $\Lambda_{m-1}$ of type $2m-2$ defines a special parahoric subgroup.

When $n$ is odd, the groups $\GU(V^+,h_n)$ and $\GU(V^-,h_n^-)$ are isomorphic and quasi-split.

In both cases, the hermitian form becomes splits over a finite unramified extension.
Therefore, the associated unitary similitude group becomes quasi-split after passing to such an extension.

\subsubsection{Pappas-Zhu local model}
Let $(H,\phi)$ be an arbitrary hermitian space over $F/F_0$, and let $G := \GU(H, \phi)$ be the associated unitary similitude group.

Fix a local model triple $(G,\{\mu\},K)$ (cf.\ \cite[\S 2.1]{HPR}), where $\{\mu\}$ is a geometric conjugacy class of cocharacter of $G$, corresponding to signature $(n-1,1)$ (the reflex field is $F$ since $n\geq 3$); cf.\ \cite[\S 2.4]{PR2009} or \S \ref{intro_moduli}. Here, $K\subset G(F_0)$ is a parahoric subgroup, with associated smooth group scheme $\CK$ defined over $\CO_{F_0}$.
Associated to this local model triple, we obtain the Pappas–Zhu local model $M_K^\loc(G,\{\mu\})$, defined over $\CO_F$ and equipped with an action of $\CK_{\CO_F}$; cf.\ \cite[Theorem 2.5]{HPR}.

Let $L_0/F_0$ be a finite unramified extension such that the base change of the hermitian form $\phi_{L_0}$ becomes splits. 
Set $L:=FL_0$.
By choosing a splitting basis, we obtain an isomorphism
\begin{equation}\label{application_nonsplit:local-model-triple-isom}
    G_{L_0}=\GU(H,\phi)_{L_0}\simeq \GU(V^+,h_n).
\end{equation}
Under this isomorphism, the parahoric subgroup $\CK(\CO_{L_0}) = K_{L_0} \subset G(L_0)$ becomes conjugate to $P_I^\circ$ for some index set $I$ as described in \S\ref{moduli_setup}.
We say that $K$ is \emph{strongly non-special} if the corresponding $P_I^\circ$ is strongly non-special.
Therefore, we obtain an isomorphism of the local model triples:
\begin{equation}\label{application_nonsplit:local-model-isom}
    (G,\{\mu\},K_{L_0})_{L_0}:=(G_{L_0},\{\mu_{L_0}\},K_{L_0})\simeq (\GU(V^+,h_n),\{\mu\},P_I^\circ).
\end{equation}
By \cite[Theorem 2.5, Proposition 2.7]{HPR}, this induces an isomorphism of Pappas-Zhu models
\begin{equation*}
    M_K^\loc(G,\{\mu\})\otimes_{\CO_F}\CO_L\xrightarrow{\sim}M_{K_{L_0}}^\loc(G_{L_0},\{\mu_{L_0}\})\xrightarrow{\sim} M^\loc_{P_{I}^\circ}(\GU(V^+,h_n),\{\mu\}).
\end{equation*}
These isomorphisms are equivariant with respect to the group actions of $\CK_{L_0}$ and $\CP^\circ_{I}$, respectively.

The following proposition is well-known to experts.
\begin{proposition}\label{application_nonsplit:PZ-vs-loc}
The Pappas-Zhu model $M^\loc_{P_{I}^\circ}(\GU(V^+,h_n),\{\mu\})$ associated with the local model triple $(\GU(V^+,h_n),\{\mu\},P_I^\circ)$ admits a natural closed embedding into the naive model $M_I^\naive$. 
Under this closed embedding, it is identified with the local model $M^\loc_I$ defined in \S \ref{moduli_fc}.
\end{proposition}
\begin{proof}
By the discussion in \cite[8.2.4]{PZ2013} (see also \cite[Lemma 4.1]{HPR}), the Pappas-Zhu model admits a closed embedding $M^\loc_{P_{I}^\circ}(\GU(V^+,h_n),\{\mu\})\hookrightarrow M^\naive_I$. 
Since both the Pappas-Zhu model and the local model $M^\loc_I$ are defined as the scheme-theoretic closure over the generic fiber, the identification follows.
\end{proof}
		
\subsubsection{Moduli description}\label{application_nonsplit_moduli}
For a parahoric subgroup $K\subset G(F_0)$, we choose a polarized self-dual lattice chain  $\CL$ in $H$, in the sense of \cite[Definition 3.14]{RZ}, such that the connected component of its stabilizer is $\CK$; see \cite{RZ} or \cite[\S 5]{PZ2013} for an overview.

When $n$ is even, and the corresponding facet of the parahoric subgroup contains two special vertices $m'$ and $m$ (cf.\ \S \ref{general_dyn}), the naive filtration in the Bruhat-Tits building takes the form:
\begin{equation*}
\xy
(-17,7)*+{\cdots\subset\Lambda_\bullet\subset \Lambda_\bullet};
(-4.5,3.5)*+{\rotatebox{-45}{$\,\, \subset\,\,$}};
(-4.5,10.5)*+{\rotatebox{45}{$\,\, \subset\,\,$}};
(0,14)*+{\Lambda_m};
(0,0)*+{\Lambda_{m'}};
(4,3.5)*+{\rotatebox{45}{$\,\, \subset\,\,$}};
(4,10.5)*+{\rotatebox{-45}{$\,\, \subset\,\,$}};
(19,7)*+{\Lambda_\bullet \subset  \Lambda_\bullet\subset\cdots,}.
\endxy
\end{equation*}
which is not a lattice chain in the sense of \cite[Definition 3.14]{RZ}. However, it can be modified into one by replacing $\Lambda_{m'}$ with $\Lambda_{m-1}$; cf.\ \cite[4.b]{PR2008}.

The base change $\CL_{L_0}$ of the lattice chain to $\CO_{L_0}$ in $H_{L_0}$ corresponds to the parahoric subgroup $K_{L_0}\subset G(L_0)$.
When $(H,\phi)=(V^+,h_n)$, the lattice chain is, up to conjugation, the one defined in \S \ref{moduli_setup_notation}.

Among those data, we define a moduli functor $\CM^\naive(G,\CL)$ over $\Spec \CO_F$, in the sense of \cite[Definition 3.27]{RZ}, which we refer to as the naive model.
		
We now formulate the strengthened spin condition. 
Recall from Remark \ref{moduli_ss:non-split-form} that, since the symmetric form on $V=H\otimes_{F_0}F$ is always splits, there is a natural decomposition $W=\bigwedge^n V=W_1\oplus W_{-1}$.
The sign $\pm 1$ here is determined via the isomorphism to the pair $(V^+_{F},h_n)$. As mentioned in \S \ref{moduli_PR_spin}, this decomposition is intrinsic, once a framing basis is fixed.	
Moreover, the decomposition $V=V_{-\pi}\oplus V_\pi$ always exists; we define $W^{r,s}:=\bigwedge^rV_{-\pi}\otimes\bigwedge^sV_{\pi}$.	
For any lattice $\Lambda\subset V$, set 
$$W(\Lambda)_{\pm 1}^{r,s} := W_{\pm 1}^{r,s} \cap W(\Lambda) \subset W\quad\text{where}\quad W(\Lambda):=\bigwedge^n(\Lambda\otimes_{\CO_{F_0}}\CO_F).$$ 
For any $\CO_F$-algebra $R$, define
\begin{equation*}
    L^{r,s}_{(-1)^s}(\Lambda)(R):=\im\bigl[ W(\Lambda_i)_{(-1)^s}^{r,s} \otimes_{\CO_F} R \to W(\Lambda_i) \otimes_{\CO_{F}}R\bigr].
\end{equation*}

We introduce the following moduli functor: 
\begin{definition}\label{application_nonsplit:defn}
Let $G$ and $\CL$ be as above.
We define a moduli functor $\CM_\CL(G,\mu)$ over $\Spec\CO_F$, which assigns to each $\CO_F$-algebra $R$ the following data:
\begin{altitemize}
	\item a functor $\Lambda\mapsto \CF_\Lambda$ from $\CL$ (viewed as a category whose morphisms are lattice inclusions in $V$) to the category of $\CO_F\otimes_{\CO_{F_0}}R$-modules; and
	\item a natural transformation of functors $j_\Lambda:\CF_\Lambda\rightarrow \Lambda\otimes R$,
\end{altitemize}
satisfying the following conditions:
\begin{altenumerate}
	\item[\fontfamily{cmtt}\selectfont LM1\&2.] For all $\Lambda\in\CL$, $\CF_\Lambda$ is a $\CO_F\otimes_{\CO_{F_0}}R$-submodule of $\Lambda\otimes_{\CO_{F_0}}R$, which is an $R$-direct summand. The morphism $j_\Lambda$ is the inclusion $\CF_\Lambda\subset\Lambda\otimes_{\CO_{F_0}}R$.
	\item[\fontfamily{cmtt}\selectfont LM3.] For all $\Lambda\in\CL$, the composition
	\begin{equation*}
		\CF_\Lambda\subset\Lambda\otimes_{\CO_{F_0}}R\xrightarrow{\pi} \pi\Lambda\otimes_{\CO_{F_0}}R 
	\end{equation*}
	identifies $\CF_\Lambda$ with $\CF_{\pi\Lambda}$.
	\item[\fontfamily{cmtt}\selectfont LM4.] For all $\Lambda\in\CL$, under the perfect pairing $(\Lambda\otimes_{\CO_{F_0}}R)\times(\wh{\Lambda}\otimes_{\CO_{F_0}}R)\rightarrow R$ induced by $\langle\,,\,\rangle$, where $\wh{\Lambda}$ denotes the dual lattice of $\Lambda$, the submodules $\CF_{\Lambda}$ and $\CF_{\wh{\Lambda}}$ pair to $0$.
	\item[\fontfamily{cmtt}\selectfont LM8.] For all $\Lambda\in\CL$, the line $\bigwedge^n_R\CF_\Lambda\subset W(\Lambda)\otimes_{\CO_F}R$ is contained in $L^{n-1,1}_{-1}(\Lambda)(R)$.
\end{altenumerate}

The naive model $\CM_\CL^\naive(G,\mu)$ is defined similarly, except that axiom {\fontfamily{cmtt}\selectfont LM8} is replaced by the Kottwitz (determinant) condition.
\end{definition}

\begin{remark}\label{application_moduli:comparison-of-moduli}
By definition, when the hermitian form is split and $\CL$ is the standard lattice chain as in \S \ref{moduli_setup}, the functors $\CM_\CL^\naive(G,\mu)$ and $\CM_\CL(G,\mu)$ can be identified with the naive model and the strengthened spin model defined in \S \ref{moduli_naive} and \S \ref{moduli_ss} respectively.
\end{remark}

\begin{theorem}\label{application_moduli:main}
	The moduli functor $\CM_\CL(G,\mu)$ represents the Pappas-Zhu model $M^\loc_{K}(G,\{\mu\})$.
\end{theorem}
\begin{proof}
By  \cite[Lemma 4.1]{HPR}, the Pappas-Zhu model $M_{K}^\loc(G,\{\mu\})$ admits a natural closed embedding into the naive model $\CM_\CL^\naive(G,\mu)$:
\begin{equation*}
    M_{K}^\loc(G,\{\mu\})\hookrightarrow \CM_\CL^\naive(G,\mu)\hookleftarrow \CM_\CL(G,\mu).
\end{equation*}
As a result, it suffices to identify $\CM_\CL(G,\mu)$ and $M_K^\loc(G,\{\mu\})$ as closed subschemes of $\CM_\CL^\naive(G,\mu)$.

Let $L_0/F_0$ be a finite unramified extension that splits the hermitian form, inducing an isomorphism of local model triples and corresponding Pappas-Zhu models as in \eqref{application_nonsplit:local-model-triple-isom} and \eqref{application_nonsplit:local-model-isom}.
Fixing a splitting basis as in \S \ref{moduli_setup}, we obtain the following commutative diagram:
\begin{equation*}
\begin{aligned}
\xymatrix{
\CM^\loc_K(G,\{\mu\})_{\CO_L}\ar@{^(->}[r]\ar[d]^{\cong} &   \CM_\CL^\naive(G,\mu)_{\CO_L}\ar[d]^{\cong}\ar@{^(->}[r] &   \CM_\CL(G,\mu)_{\CO_L}\ar[d]^{\cong}\\
\CM^\loc_{P_I^\circ}(\GU(V^+,h_n),\{\mu\})_{\CO_L}\ar@{^(->}[r]  &   M^\naive_I\ar@{^(->}[r]  &   M_I.
}
\end{aligned}
\end{equation*}
The right-hand square commutes by the moduli description.
The left-hand square commutes since the Pappas-Zhu model is defined as the scheme-theoretic closure of the generic fiber of the naive model (see \cite[Lemma 4.1]{HPR}).

By Proposition \ref{application_nonsplit:PZ-vs-loc}, the two closed subschemes at the bottom of the diagram coincide.
Thus, the diagram identifies $\CM^\loc_K(G,\{\mu\})_{\CO_L}$ with $\CM_\CL(G,\mu)_{\CO_L}$.
Descending from $\CO_L$ to $\CO_F$, we conclude that $M_K^\loc(G,\{\mu\})$ is identified with $\CM_\CL(G,\mu)$.
\end{proof}

\subsection{Moduli description of the irreducible components over the special fiber}\label{application_moduli-irr}
In this subsection, we provide a moduli description of the irreducible components of the special fiber of the local model, in the case where the parahoric subgroup is strongly non-special.

Consider the moduli functor $\CM:=\CM_\CL(G,\mu)$ defined in \eqref{application_nonsplit:local-model-isom}. We may choose a lattice $\Lambda\in \CL$ such that
\begin{equation}\label{application_moduli-irr:lattice}
    \Lambda^\vee\subset \Lambda\subset\pi^{-1}\Lambda^\vee.
\end{equation}
This lattice $\Lambda$ can be extended to a self-dual lattice chain. When the hermitian space is split, such a lattice corresponds to $\Lambda_\kappa$ in \S \ref{moduli_setup}.

The inclusions of the lattices induce two transition maps in the moduli functor:
\begin{equation}\label{application_moduli-irr:trans}
    \CF_{\Lambda^\vee}\xrightarrow{\lambda^\vee}\CF_\Lambda\xrightarrow{\lambda}\CF_{\pi^{-1}\Lambda^\vee}.
\end{equation}
By periodicity, the moduli functor $\CM$ can be restricted to the subchain \eqref{application_moduli-irr:lattice} without loss of generality. For simplicity, we adopt this convention throughout the subsection.

\begin{definition}\label{application_moduli:definition}
The projective scheme $\CM_{s,1}$ is defined as a closed subscheme of $\CM_s$. It represents the moduli problem that assigns to each $k$-algebra $R$ the set of families
\begin{equation*}
    \CM_{s,1}(R):=\left\{(\CF_{\Lambda^\vee}\xrightarrow{\lambda^\vee}\CF_\Lambda\xrightarrow{\lambda}\CF_{\pi^{-1}\Lambda^\vee})\in\CM(S)\mid \lambda(\CF_\Lambda)\subset (\pi\otimes 1)\Lambda^\vee_S\right\}.
\end{equation*}
Similarly, we define $\CM_{s,2}$ as the projective scheme representing the moduli problem that assigns to each $\CO_F\otimes_{\CO_{F_0}} k$-algebra $R$ to the set
\begin{equation*}
    \CM_{s,2}(R):=\left\{(\CF_{\Lambda^\vee}\xrightarrow{\lambda^\vee}\CF_\Lambda\xrightarrow{\lambda}\CF_{\pi^{-1}\Lambda^\vee})\in\CM(S)\mid \lambda^\vee(\CF_{\Lambda^\vee})\subset (\pi\otimes 1)\Lambda_S\right\}.
\end{equation*}
We define their scheme-theoretic intersection as $\CM_{s,12}:=\CM_{s,1}\times_{\CM_s}\CM_{s,2}$.
\end{definition}
\begin{theorem}\label{application_moduli-irr:main}
The geometric special fiber $\CM_{s,\bar{k}}$ of the local model $\CM=\CM_\CL(G,\mu)$ has two irreducible components, denoted $\CM_{s,1,\bar{k}}$ and $\CM_{s,2,\bar{k}}$, whose intersection is $\CM_{s,12,\bar{k}}$. When the hermitian space splits, all statements valid without base change from $k$ to $\bar{k}$.
\end{theorem}
\begin{proof}
After passing to a suitable unramified extension, we may assume that $\CM=M^\loc_{\{\kappa\}}$ for some chosen $\kappa$ as in Theorem \ref{intro_back:max}. 
Up to conjugation, the lattices chosen in \eqref{application_moduli-irr:lattice} can be identified with those in \eqref{moduli_setup:lattice}:
\begin{equation*}
    \Lambda\simeq \Lambda_\kappa,\ 
    \Lambda^\vee\simeq \Lambda_{-\kappa}
    \text{ and }
    \pi^{-1}\Lambda^\vee\simeq \Lambda_{n-\kappa}.
\end{equation*}
By passing to the open affine chart around the worst point, as described in \S \ref{equ_chart}, and using the reordered basis from \eqref{equ_chart:reordered-basis}, it suffices to identify the two irreducible components $\Spec R_{s,1}$ and $\Spec R_{s,2}$ with $\CM_{s,1}$ and $\CM_{s,2}$, respectively.

Choose a $R$-point 
$\CF_\kappa=\left(\begin{matrix}
	X\\I_n
\end{matrix}\right)\in U_{\{\kappa\},s}^\loc$. By \eqref{equ_nw:trans}, we have
\begin{equation*}
    \lambda(\CF_\kappa)=
    \left(\begin{matrix}
	I_{2\kappa}    &          &          &   \\
	&          &          &   0\\
	&          &   I_{2\kappa}&   \\
	&   I_{n-2\kappa}  &          &   
\end{matrix}\right)
\left(\begin{matrix}
    X_1&X_2\\X_3&X_4\\I_{2\kappa}&\\&I_{n-2\kappa}
\end{matrix}\right)=
\left(\begin{matrix}
    X_1&X_2\\0&0\\I_{2\kappa}&0\\X_3&X_4
\end{matrix}\right).
\end{equation*}
Recall from Theorem \ref{equ_coord:ss-red} that the strengthened spin condition implies $X_2=0$.
Moreover, the worst point $(\pi\otimes 1)\Lambda^\vee_R=(\pi\otimes 1)\Lambda_{n-\kappa,R}$ is represented by the matrix $\left(\begin{matrix}
	0\\I_n
\end{matrix}\right)$.
Therefore, the condition $\lambda(\CF_\kappa)\subset (\pi\otimes 1)\Lambda_{n-\kappa,R}$ is equivalent to $X_1=0$.

Similarly, choose a $R$-point $\CF_{-\kappa}=\left(\begin{matrix}
	Y\\I_n
\end{matrix}\right)$; we have
\begin{equation*}
    \lambda(\CF_{-\kappa})=
    \left(
\begin{matrix}
	&          &   0      &   \\
	&   I_{n-2\kappa}  &          &   \\
	I_{2\kappa}    &          &          &   \\
	&          &          &I_{n-2\kappa}
\end{matrix}\right)
\left(\begin{matrix}
    Y_1&Y_2\\Y_3&Y_4\\I_{2\kappa}&\\&I_{n-2\kappa}
\end{matrix}\right)=
\left(\begin{matrix}
    0&0\\Y_3&Y_4\\Y_1&Y_2\\0&I_{n-2\kappa}
\end{matrix}\right).
\end{equation*}
Recall from \eqref{equ_nw:Y} that we have the identification
\begin{equation*}
    Y=
    \left(\begin{matrix}
    Y_1 &   Y_2 \\
    Y_3 &   Y_4 \\
\end{matrix}\right)=
    \left(\begin{matrix}
    -JX_1^t J &   -JX_3^t H  \\
    H X_2^t J   &   HX_4^t H
    \end{matrix}\right).
\end{equation*}
Therefore, by Theorem \ref{equ_coord:ss-red}, the condition $\lambda^\vee(\CF_{-\kappa})\subset (1\otimes \pi)\Lambda_{\kappa,R}$ is equivalent to $X_4=0$. 
The identification now follows, since over the open affine chart of the worst point, the irreducible components of $\CM_s$ are defined by the equations $X_1=0$ and $X_4=0$ respectively;
cf.\ Remark \ref{max_simp:simp-step2_comp} and Corollary \ref{max_red:irr-comp}.
\end{proof}

\subsection{Application to Shimura varieties}\label{application_Shimura}
We next provide a moduli description for the arithmetic model of unitary similitude Shimura varieties. For simplicity, we focus on the case where the hermitian form is defined over an imaginary quadratic field. Except possibly for the Hasse principle, all statements remain valid in the more general setting of a CM extension $\BF/\BF_0$, though the formulations become considerably more involved. For general treatments, see \cite[\S 11]{LZ2021} and \cite[\S 6]{LMZ25}.
		
Let $\BF\slash\BQ$ be an imaginary quadratic extension. Consider an $\BF$-vector space $\BV$ of dimension $n$, equipped with a skew-hermitian form $\varphi:\BV\times \BV\rightarrow \BQ$. 
We define the unitary similitude group $\BG = \GU(\BV, \varphi)$. Fixing an embedding $\BF \hookrightarrow \BC$, we assume that the base change $\varphi_\BC$ has signature $(1, n-1)$. This determines a conjugacy class of homomorphisms $h : \Res_{\BC/\BR} \BG_m \rightarrow \BG_\BR$, and the pair $(\BG_\BR, h)$ defines a Shimura datum of PEL type.
		
Now fix a prime $p\neq 2$ that is ramified in $\BF/\BQ$ so that the base change to $\BQ_p$ yields a ramified extension $F/\BQ_p$.
Let $G:=\BG_{\BQ_p}$ be the unitary similitude group associated with the induced hermitian form $\phi=\varphi_{\BQ_p}$. Choose a parahoric subgroup $K_p\subset G(F)$, which can always be realized as the connected component of the stabilizer of some self-dual chain of lattices $\CL$ in $V:=\BV_F$.

Up to conjugation, such lattice chains correspond bijectively to non-empty subsets of the local Dynkin diagram of $G_{\breve{F}}$. We say that $\CL$ is \emph{non-connected} if $n = 2m$ and the corresponding subset of the Dynkin diagram contains no special vertices; otherwise, we say $\CL$ is \emph{connected}.

Let $K_p':=\text{Stab}_{G(F)}(\CL)$ and $K_p=\text{Stab}^\circ_{G(F)}(\CL)$ denote the full and connected stabilizers, respectively. Then $K_p=K_p'$ if and only if $\CL$ is connected.

Finally, fix an open compact subgroup $K^p \subset G(\BA^p_f)$ and define $\BK := K_p K^p$ and $\BK’ := K_p’ K^p$ as compact open subgroups of $\BG(\BA_\BQ)$.

We define a Deligne-Mumford stack $\CA_{\BK'}$ over $\Spec\CO_F$ by modifying the construction in \cite{RZ}.
Let $R$ be a $\CO_F$-algebra; we associate to it the groupoid of the data $(A,\iota,\bar{\lambda},\bar{\eta})$ satisfying the following conditions:
\begin{altenumerate}
	\item $A=\{A_\Lambda\}_{\Lambda\in\CL}$ is an $\CL$-indexed family of abelian schemes over $\Spec R$, each equipped with a compatible action $\CO_F$:
 \begin{equation*}
     \iota:\CO_F\rightarrow \End{A_\Lambda}\otimes \BZ_p.
 \end{equation*}
	\item $\ov{\lambda}$ a $\BQ$-homogeneous principal polarization on the $\CL$-family $A$.
	\item $\ov{\eta}$ is a $K^p$-level structure
	\begin{equation*}
		\ov{\eta}: H_1(A,\BA_f^p)\simeq V\otimes \BA_f^p\mod K^p,
	\end{equation*}
	which respects the bilinear forms on both sides up to a constant in $(\BA_f^p)^\times$.
\end{altenumerate}
In addition, for each $\Lambda\in \CL$, consider the Hodge filtration on the de Rham homology:
\begin{equation}\label{application_Shimura:hodge-fil}
	\Fil^1(A_\Lambda)\hookrightarrow \BD(A_\Lambda).
\end{equation}
There exists an  \'etale cover $\Spec R'/\Spec R$ such that the chain of de Rham homologies $\{\BD(A_\Lambda)\}_{\Lambda\in \CL}$ becomes isomorphic to the corresponding lattice chain after base change: $\BD(A_\Lambda)\otimes_{\CO_S}\CO_{S'}\simeq \Lambda_{S'}$; cf.\ \cite[Appendix to Chapter 3]{RZ}.
We may then regard $\Fil^1(A_{\Lambda})_{S'}$ as a subbundle of $\Lambda_{S'}$. 
We impose the strengthened spin condition that
\begin{equation*}
	\bigwedge^n_R(\Fil^1(A_\Lambda)_{R'})\subset L_{-1}^{n-1,1}(\Lambda)(R').
\end{equation*}
The strengthened spin condition we impose here is independent of the choice of \'etale cover $\Spec R'/\Spec R$, thanks to the following lemma.

\begin{lemma}[\protect{\cite[Lemma 7.1]{RSZ2018}}]\label{application_shimura_strengthened-spin}
	For any $\CO_F$-algebra $R$ and lattice $\Lambda\in\CL$, the submodule $L^{r,s}_{\pm 1}(\Lambda)(R)\subset \bigwedge^n\Lambda\otimes_{\CO_F}R$ is stable under the natural action of $\underline{\Aut}(\CL)(R)$ on $\bigwedge^n\Lambda\otimes_{\CO_F}R$.\qed
\end{lemma}
		
The main theorem is the following:
\begin{theorem}
The moduli functor $\CA_{\BK'}$ is representable by a flat and normal Deligne-Mumford stack over $\Spec(\CO_F)$. When the lattice chain $\CL$ is connected, $\CA_{\BK’}$ defines an integral model of the Shimura variety $\Sh_\BK(G,\mu)$. 
In the non-connected case, there exists a \'etale double cover $\CA_{\BK}\rightarrow\CA_{\BK'}$, together with an open immersion $\Sh_{\BK}(\BG,h)\otimes_{\BF}F\hookrightarrow\CA_{\BK}$, such that the following diagram is cartesian:
\begin{equation*}
\begin{aligned}
\xymatrix{
\Sh_{\BK}(\BG,h)\otimes_\BF F\ar@{}[rd]|{\square}\ar@{^(->}[r]\ar[d]&\CA_{\BK}\ar[d]\\
\Sh_{\BK'}(\BG,h)\otimes_{\BF}F\ar@{^(->}[r]&\CA_{\BK'}.
}
\end{aligned}
\end{equation*}
\end{theorem}
		
\begin{proof}
By Lemma \ref{application_shimura_strengthened-spin}, the strengthened spin condition descends through the local model diagram and corresponds to the strengthened spin model.
The moduli functor $\CA_{\BK'}$ is a closed subfunctor of the naive moduli functor $\CA_{\BK'}^\naive$ defined in \cite{RZ}, cut out by the strengthened spin condition, which is a closed condition.
Therefore, $\CA_{\BK'}$ is representable by a Deligne-Mumford stack over $\Spec(\CO_F)$.
By Theorem \ref{application_moduli:main} and local model diagram, the stack $\CA_{\BK'}$ is flat and normal.
The \'etale double cover mentioned in the end of the theorem arises from the difference between $P_I$ and $P_I^\circ$, as explained in \cite[Section 1.2, Proposition 1.1]{PR2009}.
\end{proof}
		
\begin{remark}
As pointed out in \cite[\S 1.3]{PR2009}, the Shimura varieties $\Sh_\BK(G,h)$ and $\Sh_{\BK'}(\BG,h)$ have isomorphic geometric connected components.
\end{remark}

\section{Proofs of some results in Section \ref{equ}}
In this section, we give the proofs left in \S \ref{equ}.
Recall that the lattice $\Lambda_{\kappa}$ is generated by
\begin{equation*}
	\pi^{-1}e_1\otimes 1,\cdots,\pi^{-1}e_{\kappa}\otimes 1,e_{\kappa+1}\otimes 1,\cdots,e_n\otimes 1; e_1\otimes 1,\cdots,e_{\kappa}\otimes 1,\pi e_{\kappa+1}\otimes 1,\cdots,\pi e_n\otimes 1.
\end{equation*}
Throughout the proof, we will boldface the worst terms in each expression. Below are the terms of $g_i$, with their worst terms indicated in bold:
\begin{equation*}
	\mathbf{e_1\otimes 1}-\pi e_1\otimes\pi^{-1},\cdots,\mathbf{e_\kappa\otimes 1}-\pi e_\kappa\otimes\pi^{-1};e_{\kappa+1}\otimes 1-\mathbf{\pi e_{\kappa+1}\otimes\pi^{-1}},
	\cdots,
	e_{n}\otimes 1-\mathbf{\pi e_{n}\otimes\pi^{-1}};	
\end{equation*}
\begin{equation*}
	\mathbf{\frac{1}{2}e_1\otimes 1}+\frac{1}{2}\pi e_1\otimes \pi^{-1},
	\cdots,
	\mathbf{\frac{1}{2}e_\kappa\otimes 1}+\frac{1}{2}\pi e_\kappa\otimes \pi^{-1};	
	\frac{1}{2}e_{\kappa+1}\otimes 1+\mathbf{\frac{1}{2}\pi e_{\kappa+1}\otimes \pi^{-1}},
	\cdots
	\frac{1}{2}e_{n}\otimes 1+\mathbf{\frac{1}{2}\pi e_{n}\otimes \pi^{-1}}.	
\end{equation*}
Recall the notation convention $e_{[i,\wh{n+j}]}$ in \eqref{equ_ss-comp:not}.
\subsection{Proof of Proposition \ref{equ_ss-comp:g_S}}\label{Appendix_first-proof}
In this subsection, we will give the proof of Proposition \ref{equ_ss-comp:g_S}. The following lemma will be useful.
\begin{lemma}[\protect{\cite[equation before Lemma 4.9]{Smithling2015}}]\label{equ_ss-comp:two-wedges}
\[
	g_i\wedge g_{n+i} =  (e_i\otimes 1)\wedge(\pi e_i\otimes \pi^{-1})=  (e_i\otimes 1)\wedge (\pi^{-1}e_i\otimes \pi).	\qed
\]
\end{lemma}

\begin{proof}[Proof of Proposition \ref{equ_ss-comp:g_S}]
\begin{altenumerate}
	\item When $S=\{1,\cdots,\wh{i},\cdots,n,n+i\}$ for some $1\leq i\leq \kappa$, we have
\begin{multline*}
g_S=(\mathbf{e_1\otimes 1}-\pi^{-1}e_1\otimes \pi)\wedge\cdots\wedge \wh{(\mathbf{e_i\otimes 1}-\pi^{-1}e_i\otimes \pi)}\wedge\cdots\wedge(\mathbf{e_\kappa\otimes 1}-\pi^{-1}e_\kappa\otimes \pi)\\
\wedge(e_{\kappa+1}\otimes 1-\mathbf{\pi e_{\kappa+1}\otimes \pi^{-1}})\wedge\cdots\wedge(e_n\otimes 1-\mathbf{\pi e_n\otimes \pi^{-1}})\wedge (\mathbf{\frac{1}{2}e_i\otimes 1}+\frac{1}{2}\pi^{-1}e_i\otimes \pi).
\end{multline*}
 A quick simplification shows that 
	\begin{align*}
		g_S={}   &\frac{1}{2}(-1)^{\kappa+i}\pi^{-(n-\kappa)}e_{\{n+1,\cdots,2n\}}\\
		&\quad+\frac{1}{2}(-1)^{\kappa+i}\pi^{-(n-\kappa-1)}\Bigl[\sum_{\sigma=1}^{i-1}(-1)^{\sigma}e_{[\sigma,\wh{n+\sigma}]}+\sum_{\sigma=i+1}^{\kappa}(-1)^{\sigma}e_{[\sigma,\wh{n+\sigma}]}+\sum_{\sigma=\kappa+1}^{n}(-1)^{\sigma}e_{[\sigma,\wh{n+\sigma}]}\Bigr]\\
		&\hspace{8cm}
		+\frac{1}{2}(-1)^{\kappa+1}\pi^{-(n-\kappa-1)}e_{[i,\wh{n+i}]}+o(\pi^{-(n-k-1}),\\
        ={}&
        \frac{1}{2}(-1)^{\kappa+i}\pi^{-(n-\kappa)}\Bigl[
        (-1)^ie_{\{n+1,\cdots,2n\}}+\pi\Bigl(\sum_{\sigma\neq i}(-1)^\sigma e_{[\sigma,\wh{n+\sigma}]}\Bigr)+(-1)^{i+1}\pi e_{[i,\wh{n+i}]}
        \Bigr]+o(\pi^{-(n-\kappa-1)}),\\
        ={}&\frac{1}{2}(-1)^{\kappa+i}\pi^{-(n-\kappa)}\Bigl[
        e_{\{n+1,\cdots,2n\}}+\pi\Bigl(
        2(-1)^{i+1}e_{[i,\wh{n+i}]}+\sum_{\sigma=1}^n(-1)^\sigma e_{[\sigma,\wh{n+\sigma}]}
        \Bigr)
        \Bigr]+o(\pi^{-(n-\kappa-1)}).
	\end{align*}
				
	\item When $S=\{1,\cdots,\wh{i},\cdots,n,n+i\}$ for some $\kappa+1\leq i \leq n$, the argument proceeds exactly as in the case $1 \leq i \leq \kappa$.
					
	\item When $S=\{1,\cdots,\wh{j},\cdots,n,n+i\}$ for some $i,j\leq \kappa$, $i\neq j$, we have
	\begin{align*}
		g_S={}  &(\mathbf{e_1\otimes 1}-\pi^{-1}e_1\otimes \pi)\wedge\cdots\wedge 
		\wh{(\mathbf{e_j\otimes 1}-\pi^{-1}e_i\otimes \pi)}\wedge\cdots\wedge(\mathbf{e_\kappa\otimes 1}-\pi^{-1}e_\kappa\otimes \pi)\\
		&\wedge(e_{\kappa+1}\otimes 1-\mathbf{\pi e_{\kappa+1}\otimes \pi^{-1}})\wedge\cdots\wedge(e_n\otimes 1-\mathbf{\pi e_n\otimes \pi^{-1}})\wedge (\mathbf{\frac{1}{2}e_i\otimes 1}+\frac{1}{2}\pi^{-1}e_i\otimes \pi),\\
		={}     &(e_1\otimes 1)\wedge\cdots\wedge\wh{(e_j\otimes 1)}\wedge\cdots\wedge (e_\kappa\otimes          1)\wedge(-\pi e_{\kappa+1}\otimes \pi^{-1})\\
		&\hspace{5.5cm}
		\wedge\cdots\wedge(-\pi e_n\otimes\pi^{-1})\wedge(\pi^{-1}e_i\otimes \pi)+o(\pi^{-(n-\kappa-1)}),\\
		={}&(-1)^{n-1}(-1)^{n-\kappa}\cdot\pi^{-(n-\kappa-1)}e_{[i,\wh{n+j}]}+o(\pi^{-(n-\kappa-1)}),\\
		={}&(-1)^{\kappa+1}\pi^{-(n-\kappa-1)}e_{[i,\wh{n+j}]}+o(\pi^{-(n-\kappa-1)}).
	\end{align*}
				
	\item When $S=\{1,\cdots,\wh{j},\cdots,n,n+i\}$ for some $i\leq \kappa <j$, we have
	\begin{align*}
		g_S={}&
		(e_1\otimes 1)\wedge\cdots\wedge\cdots(e_\kappa\otimes 1)\wedge(-\pi e_{\kappa+1}\otimes \pi^{-1})\wedge\cdots\wedge\wh{(-\pi e_j\otimes \pi^{-1})}\\
		&\hspace{5cm}\wedge\cdots\wedge(-\pi e_n\otimes \pi^{-1})\wedge(\pi^{-1} e_i\otimes \pi)+o(\pi^{-(n-\kappa-2)}),\\
		={}&(-1)^{n-1}\cdot(-1)^{n-\kappa-1}\cdot\pi^{(n-\kappa-1)+1}e_{[i,\wh{n+j}]}+o(\pi^{-(n-\kappa-2)}),\\
		={}&(-1)^{\kappa}\pi^{-(n-\kappa-2)}e_{[i,\wh{n+j}]}+o(\pi^{-(n-\kappa-2)}).
	\end{align*}

	\item When $S=\{1,\cdots,\wh{j},\cdots,n,n+i\}$ for some $j\leq \kappa<i$, we have
	\begin{align*}
		g_S={}&
		(e_1\otimes 1)\wedge\cdots\wedge\wh{(e_j\otimes 1)}\wedge\cdots\wedge (e_\kappa\otimes 1)\wedge(-\pi e_{\kappa+1}\otimes \pi^{-1})\\
		&\hspace{5.5cm}
		\wedge\cdots\wedge(-\pi e_n\otimes \pi^{-1})\wedge(e_i\otimes 1)+o(\pi^{-(n-\kappa)}),\\
		={}&(-1)^{n-1}\cdot(-1)^{n-\kappa}\cdot\pi^{-(n-\kappa)}e_{[i,\wh{n+j}]}+o(\pi^{-(n-\kappa)}),\\
		={}&(-1)^{\kappa+1}\pi^{-(n-\kappa)}e_{[i,\wh{n+j}]}+o(\pi^{-(n-\kappa)}).
	\end{align*}

	\item When $S=\{1,\cdots,\wh{j},\cdots,n,n+i\}$ for some $i,j>\kappa$, we have
	\begin{align*}
		g_S={}&
		(e_1\otimes 1)\wedge\cdots\wedge\cdots(e_\kappa\otimes 1)\wedge(-\pi e_{\kappa+1}\otimes \pi^{-1})\wedge\cdots\wedge\wh{(-\pi e_j\otimes \pi^{-1})}\\
		&\hspace{5cm}\wedge\cdots\wedge(-\pi e_n\otimes \pi^{-1})\wedge(e_i\otimes 1)+o(\pi^{-(n-\kappa-1)}),\\
		={}&(-1)^{n-1}\cdot(-1)^{n-\kappa-1}\cdot\pi^{(n-k-1)}e_{[i,\wh{n+j}]}+o(\pi^{-(n-\kappa-1)}),\\
		={}&(-1)^{\kappa}\pi^{-(n-\kappa-1)}e_{[i,\wh{n+j}]}+o(\pi^{-(n-\kappa-1)}).
	\end{align*}

\end{altenumerate}
\end{proof}
		
\subsection{Proof of Proposition \ref{equ_ss-comp:g_S-dual}}\label{Appendix_second-proof}
Recall from Remark \ref{equ_ss-setup:sigma-compute} that if $S=\{1,2,\cdots,\hat{j},\cdots,n,n+i\}$, then
\begin{equation*}
	\sgn(\sigma_S)=(-1)^{1+\cdots+n-j+n+i+\frac{n(n+1)}{2}}=(-1)^{n+i+j}.
\end{equation*}
Moreover, it is straightforward to see that
\begin{lemma}\label{equ_ss-comp:dual}
	$g_S$ and $g_{S^\perp}$ have the same worst term.\qed
\end{lemma}
\begin{proof}[Proof of Proposition \ref{equ_ss-comp:g_S-dual}]\hfill

\bigskip
\noindent\textbf{Case 1.} When $S=S^\perp$ and $i\neq j$, these conditions are equivalent to
$j=i^\vee=n+1-i$ with $i\neq j$. 
We have
\[
	\sgn(\sigma_S)=(-1)^{n+i+j}=(-1)^{2n+1}=-1.
\]
Therefore, we have $\WT(g_S-\sgn(\sigma_S)g_{S^\perp})=2\WT(g_S)$, and (i)(ii)(iii) follow directly from Proposition \ref{equ_ss-comp:g_S}.
			
\bigskip		
\noindent\textbf{Case 2.} When $i=j$,
we have $\sgn(\sigma_S)=(-1)^{n+i+j}=(-1)^n.$

\noindent (iv) When $i=j\leq \kappa$, we have $i^\vee=j^\vee>n-\kappa$. Therefore, we have
\begin{align*}
g_S={}&\frac{1}{2}(-1)^{\kappa+i}\pi^{-(n-\kappa)}\Bigl[e_{\{n+1,\cdots,2n\}}
+\pi\Bigl(
2(-1)^{i+1}e_{[i,\wh{n+i}]}+\sum_{\sigma=1}^n(-1)^\sigma e_{[\sigma,\wh{n+\sigma}]}
\Bigr)\Bigr]+o(\pi^{-(n-\kappa-1)});\\
g_{S^\perp}={}&\frac{1}{2}(-1)^{\kappa+i^\vee+1}\pi^{-(n-\kappa)}\Bigl[e_{\{n+1,\cdots,2n\}}
+\pi\Bigl(
2(-1)^{i^\vee+1} e_{[i^\vee,\wh{n+i^\vee}]}+\sum_{\sigma=1}^n(-1)^\sigma e_{[\sigma,\wh{n+\sigma}]}
\Bigr)\Bigr]+o(\pi^{-(n-\kappa-1)}).\\
\end{align*}
Since $(-1)^n(-1)^{\kappa+i^\vee+1}=(-1)^{\kappa+i}$, the leading terms cancel. 
Therefore, we have
\begin{align*}
	g_S-\sgn(\sigma_S)g_{S^\perp}
	={} &\frac{1}{2}(-1)^{\kappa+i}\pi^{-(n-\kappa-1)}\biggl(
	2(-1)^{i+1}e_{[i,\wh{n+i}]}-2(-1)^{i^\vee+1} e_{[i^\vee,\wh{n+i^\vee}]}\biggr),\\
	={} & (-1)^{\kappa+1}\pi^{-(n-\kappa-1)}\left(e_{[i,\wh{n+i}]}+(-1)^ne_{[i^\vee,\wh{n+i^\vee}]}\right).
\end{align*}	

\noindent (v) When $\kappa+1\leq i=j\leq m$, we have $i^\vee=j^\vee\geq m+1>\kappa$. Therefore, we have 
\begin{align*}
	&g_S-\sgn(\sigma_S)g_{S^\perp}\\
	={} &   \Bigl(\frac{1}{2}(-1)^{\kappa+i+1}-\frac{1}{2}(-1)^n
	(-1)^{\kappa+i^\vee+1}\Bigr)\pi^{-(n-\kappa)}e_{\{n+1,\cdots,2n\}}\\
	&\qquad   +\frac{1}{2}(-1)^{\kappa+i+1}\pi^{-(n-\kappa-1)}
	\Bigl(
	2(-1)^{i+1} e_{[i,\wh{n+i}]}+\sum_{\sigma=1}^n(-1)^\sigma e_{[\sigma,\wh{n+\sigma}]}
	\Bigr)\\
	&\qquad   -\frac{1}{2}(-1)^n(-1)^{\kappa+i^\vee+1}\pi^{-(n-k-1)}\Bigl(
	2(-1)^{i^\vee+1} e_{[i^\vee,\wh{n+i^\vee}]}+\sum_{\sigma=1}^n(-1)^\sigma e_{[\sigma,\wh{n+\sigma}]}\Bigr)+o(\pi^{-(n-\kappa-1)}),\\
	={} &   (-1)^{\kappa+i+1}\pi^{-(n-\kappa)}\\
		&\qquad\Bigl[e_{\{n+1,\cdots,2n\}}+\pi\Bigl(
		(-1)^{i+1} e_{[i,\wh{n+i}]}
		+(-1)^{i^\vee+1} e_{[i^\vee,\wh{n+i^\vee}]}
		+\sum_{\sigma=1}^n(-1)^\sigma  e_{[\sigma,\wh{n+\sigma}]}\Bigr)\Bigr]+o(\pi^{-(n-\kappa-1)}).
\end{align*}

\noindent (vi) When $i = j = m + 1$ (which occurs only when $n$ is odd), the result follows directly from case (v).

\bigskip
\noindent\textbf{Case 3.}	When $S\neq S^\perp,i\neq j$ and $i<j^\vee$, as in Case 1, it suffices to consider only the worst term.

\noindent (vii) When $i<j^\vee\leq \kappa$, we have
$i<k<j,j^\vee\leq \kappa<i^\vee$. In this case, we have
\begin{equation*}
	\WT(g_S)=(-1)^\kappa\pi^{-(n-\kappa-2)}e_{[i,\wh{n+j}]},
	\quad
	\WT(g_{S^\perp})=(-1)^\kappa\pi^{-(n-\kappa-2)}e_{[j^\vee,\wh{n+i^\vee}]}.
\end{equation*}
Hence
\[
    \WT(g_S-\sgn(\sigma_S)g_{s^\perp})=(-1)^\kappa\pi^{-(n-\kappa-2)}
	\left(
	e_{[i,\wh{n+j}]}-(-1)^{n+i+j}e_{[j^\vee,\wh{n+i^\vee}]}
	\right).
\]
(viii) When $i\leq \kappa<j^\vee<\kappa^\vee$, we have
$i\leq \kappa<j,\kappa<i^\vee\neq j^\vee$. In this case, we have
\begin{equation*}
	\WT(g_S)=(-1)^\kappa\pi^{-(n-\kappa-2)}e_{[i,\wh{n+j}]},
	\quad\text{and}\quad
	\WT(g_{S^\perp})=(-1)^\kappa\pi^{-(n-\kappa-1)}e_{[j^\vee,\wh{n+i^\vee}]}.
\end{equation*}
Hence
\[
    \WT(g_S-\sgn(\sigma_S)g_{S^\perp})=(-1)^{n+\kappa+1+i+j}\pi^{-(n-\kappa-1)}e_{[j^\vee,\wh{n+i^\vee}]}.
\]
(xi) When $i\leq \kappa$ and $\kappa^\vee\leq j^\vee$, we have
$i\neq j\leq \kappa,\kappa<i^\vee\neq j^\vee$. In this case, we have
\begin{equation*}
	\WT(g_S)=(-1)^{\kappa+1}\pi^{-(n-\kappa-1)}e_{[i,\wh{n+j}]},
	\quad\text{and}\quad
	\WT(g_{S^\perp})=(-1)^\kappa\pi^{-(n-\kappa-1)}e_{[j^\vee,\wh{n+i^\vee}]}.
\end{equation*}
Hence
\[
	\WT(g_S-\sgn(\sigma_S)g_{S^\perp})=(-1)^{k+1}\pi^{-(n-k-1)}
	\left(
	e_{[i,\wh{n+j}]}+(-1)^{n+i+j}e_{[j^\vee,\wh{n+i^\vee}]}
	\right).
\]
(x) When $\kappa<i<j^\vee<\kappa^\vee$, we have
$k<i\neq j,\kappa<i^\vee\neq j^\vee$. In this case, we have
\begin{equation*}
	\WT(g_S)=(-1)^\kappa\pi^{-(n-\kappa-1)}e_{[i,\wh{n+j}]},
	\quad\text{and}\quad
	\WT(g_{S^\perp})=(-1)^\kappa\pi^{-(n-\kappa-1)}e_{[j^\vee,\wh{n+i^\vee}]}.
\end{equation*}
Hence
\[
    \WT(g_S-\sgn(\sigma_S)g_{S^\perp})=
	(-1)^\kappa\pi^{-(n-\kappa-1)}
    \left( 
	e_{[i,\wh{n+j}]}-(-1)^{n+i+j}e_{[j^\vee,\wh{n+i^\vee}]}
	\right).
\]
(xi) When $k<i<\kappa^\vee\leq j^\vee$, we have
$j\leq \kappa<i,\kappa<i^\vee\neq j^\vee$. In this case, we have
\begin{equation*}
	\WT(g_S)=(-1)^{\kappa+1}\pi^{-(n-\kappa)}e_{[i,\wh{n+j}]},
	\quad\text{and}\quad
	\WT(g_{S^\perp})=(-1)^\kappa\pi^{-(n-\kappa-1)}e_{[j^\vee,\wh{n+i^\vee}]}.
\end{equation*}
Hence
\[
	\WT(g_S-\sgn(\sigma_S)g_{S^\perp})=(-1)^{\kappa+1}\pi^{-(n-\kappa)}e_{[i,\wh{n+j}]}. 
\] 
(xii) When $\kappa^\vee\leq i<j^\vee$, we have
$j<\kappa<i$ and $i^{\vee}\le \kappa < j^{\vee}$. In this case, we have
\begin{equation*}
	\WT(g_S)=(-1)^{\kappa+1}\pi^{-(n-\kappa)}e_{[i,\wh{n+j}]},
	\quad\text{and}\quad
	\WT(g_{S^\perp})=(-1)^{\kappa+1}\pi^{-(n-\kappa)}e_{[j^\vee,\wh{n+i^\vee}]}.
\end{equation*}
Hence
\[
	\WT(g_S-\sgn(\sigma_S)g_{S^\perp})=(-1)^{\kappa+1}\pi^{-(n-\kappa)}\left(
	e_{[i,\wh{n+j}]}-(-1)^{n+i+j}e_{[j^\vee,\wh{n+i^\vee}]}
	\right).
\]
\end{proof}

\section{Integral equations of the local model}\label{intequ}

In this section, we explicitly express the affine ring of $U_{\{\kappa\}}^\loc\subset M_{\{\kappa\}}^\loc$ defined in \S \ref{equ_chart}.

\begin{theorem}\label{intequ:main}
Over the affine chart $U^\loc$ at the worst point of the local model, we write
\[
    X=\left(\begin{array}{c|c}
	X_1&X_2\\
	\hline
	X_3&X_4
	\end{array}\right)=
	\left(\begin{array}{cc|c}
		A&B&L\\
		C&D&M\\
		\hline
		E&F&X_4
	\end{array}\right)
\]
Then the affine chart $U^\loc$ is isomorphic to the factor ring of $\mathscr{O}_F[X]$ modulo the ideal generated by the entries of the following matrices:
\begin{altitemize}
	\item[\fontfamily{cmtt}\selectfont LM1.] $X_1^2+X_2X_3=\pi_0 I,\, X_1X_2+X_2X_4=0,\, X_3X_1+X_4X_3=0,\, X_3X_2+X_4^2=\pi_0I,$
	\item[\fontfamily{cmtt}\selectfont LM2.] $-JX_1+X_3^tHX_3+X_1^tJ=0,\, -JX_2+X_3^tHX_4=0,\, X_2^tJ+X_4^tHX_3=0,\, X_4^tHX_4-\pi_0 H_{n-2k}=0,$
	\item[\fontfamily{cmtt}\selectfont  LM2.] $X_1JX_1^t-\pi_0J=0,\, X_1JX_3^t-X_2H=0,\, X_3JX_1^t+HX_2^t=0,\, X_3JX_3^t-X_4H+HX_4^t=0$,
	\item[\fontfamily{cmtt}\selectfont  LM6.] $\bigwedge^2(X+\pi\id)=0,\, \bigwedge^n(X-\pi\id)=0$,
	\item[\fontfamily{cmtt}\selectfont  LM8.] $B=B^{\ad},\, C=C^{\ad},\, D=-2\pi I-A^{\ad},\, M=\pi E^{\ad},\, L=-\pi F^{\ad},\, X_4=X_4^{\ad},\, \tr(X_4)=-(n-2\kappa-2)\pi$.
\end{altitemize}
Hence, after simplification, the affine chart $U^\loc_{\{\kappa\}}$ is isomorphic to the scheme
\begin{equation*}
	\Spec\frac{\CO_F[X_3,X_4]}{\bigwedge^2(X_3,X_4+\pi\id),X_4-X_4^{\ad},\tr(X_4)+(n-2k-2)\pi}.
\end{equation*}
Let $\bA=H(X_4+\pi\id)$ and $\bB=X_4$, it further simplifies into
\begin{equation*}
	\Spec\frac{\CO_F[\bA,\bB]}{\bigwedge^2(\bA,\bB),\bA-\bA^{t},\tr(\bA H)-2\pi}.
\end{equation*}
\end{theorem}
		
\subsubsection{}
The equations defined above determine a closed subscheme $U'$ of $U^\naive_{\{\kappa\}}\subset M^\naive_{\{\kappa\}}$.
We want to show that $U_{\{\kappa\}}^\loc\subset U'\subset U^\naive_{\{\kappa\}}$.
Once this inclusion holds, since $U^\loc_{\{\kappa\}}\subset U^\naive_{\{\kappa\}}$ is equality over the generic fiber, and $U'$ agrees with $U^\loc_{\{\kappa\}}$ over the special fiber by having the same defining equations, it follows from \cite[Proposition 14.17]{GW2020} that $U^\loc_{\{\kappa\}}=U'$. 
Therefore, it suffices to verify that the relations in {\fontfamily{cmtt}\selectfont  LM8} hold on $U$.
The argument is the same as in the proof of Theorem \ref{equ_coord:ss-red}, except that we must keep track of additional terms in the integral equations.

Note that since flatness has already been established, one may alternatively verify the defining equations over the generic fiber, as in \cite{Zachos}.

Recall that, with respect to the standard basis, we have $\CF_\kappa=\Span\left(\begin{matrix}\CX\\I_n\end{matrix}\right)$, where $\CX$ is defined in \eqref{equ_chart:standard-X}.
We write
\begin{equation*}
	\bigwedge^n\CF_k=\sum_S c_S e_S=\sum_{T \text{ balanced }}a_{T}(g_T-\sgn(\sigma_T)g_{T^\perp}),
\end{equation*}
where $S,T\subset \{1,\cdots,2n\}$ are subsets of size $n$. 
All relevant relations arise from comparing the coefficients $a_T$ and $c_S$ for those $S$ of the type $(n,0)$ and $(n-1,1)$, as in the special fiber analysis of \S \ref{equ_coord}.
Our goal is to express each $g_T-\sgn(\sigma_T)g_{T^\perp}$ as a linear combination of $c_S e_S$ with $S$ of type $(n,0)$ and $(n-1,1)$.
We write ``$O$'' for the terms involving$S$ of other types.

\subsubsection{}
Recall that we define $c_{ij}:=c_{\{i,n+1,\cdots,\wh{n+j},\cdots,2n\}}$.
In particular, we have
\begin{equation*}
    c_{ij}e_{[i,\wh{n+j}]}=e_{n+1}\wedge\cdots\wedge e_{n+j-1}\wedge (x_{ij}e_i)\wedge e_{n+j+1}\wedge\cdots e_{2n}=(-1)^{j-1}x_{ij}e_{[i,\wh{n+j}]}.
\end{equation*}
Hence, $x_{ij}=(-1)^{j-1}c_{ij}$.
			
It is straightforward from Proposition \ref{equ_ss-comp:g_S} or the proof of Proposition \ref{equ_ss-comp:g_S-dual} that we have 
\bigskip	

\noindent\textbf{Case 1.} When $S=S^\perp, i\neq j$, which is equivalent to $i+j=n+1$ with $i\neq j$.

\noindent (i) When $i\leq k$, we have
\[
	g_S-\sgn(\sigma_S)g_{S^\perp}
	=2(-1)^\kappa \pi^{-(n-\kappa-2)}e_{[i,\wh{n+j}]}+O.
\]
\noindent (ii) When $\kappa<i\leq n-\kappa,i\neq j$, we have
\[
    g_S-\sgn(\sigma_S)g_{S^\perp}
	=2(-1)^{\kappa+1} \pi^{-(n-\kappa-1)}e_{[i,\wh{n+j}]}+O.
\]
\noindent (iii) When $i>n-\kappa\Rightarrow j\leq \kappa,i\neq j$, we have 
\[
	g_S-\sgn(\sigma_S)g_{S^\perp}
	=2(-1)^{\kappa+1} \pi^{-(n-\kappa)}e_{[i,\wh{n+j}]}+O.
\]
			
\noindent\textbf{Case 2.} When $i=j$.

\noindent (iv) When $i=j\leq \kappa$, we have
\[
	g_S-\sgn(\sigma_S)g_{S^\perp}
	=(-1)^{\kappa+1}\pi^{-(n-\kappa-1)}
	\left(e_{[i,\wh{n+i}]}
	+(-1)^ne_{[i^\vee,\wh{n+i^\vee}]}\right)+O.
\]
\noindent (v) When $\kappa<i=j\leq m$, we have
\begin{align*}
	&g_{S_i}-\sgn(\sigma_{S_i})g_{S_i^\perp}\\
 ={}&(-1)^{\kappa+i+1}\pi^{-(n-\kappa)}\Bigl[e_{\{n+1,\cdots,2n\}}+\pi\Bigl(
	(-1)^{i+1}e_{[i,\wh{n+i}]}
	+(-1)^{i^\vee+1}e_{[i^\vee,\wh{n+i^\vee}]}
	+\sum_{\sigma=1}^n(-1)^\sigma e_{[\sigma,\wh{n+\sigma}]}\Bigr)\Bigr]+O
\end{align*}
(vi) When $i=j=m+1$ (which occurs only when $n$ is odd), we have
\begin{equation*}
    g_S-\sgn(\sigma_S)g_{s^\perp}=(-1)^{\kappa+m}\pi^{-(n-\kappa)}\Bigl[e_{\{n+1,\cdots,2n\}}
    +\pi\Bigl(2(-1)^{m}e_{[m+1,\wh{n+m+1}]}+
	\sum_{\sigma=1}^n (-1)^\sigma e_{[\sigma,\wh{n+\sigma}]}\Bigr)\Bigr]+O.
\end{equation*}
			
\noindent\textbf{Case 3.} When $S$ balanced, with $S \neq S^\perp$ and $i\neq j.$

\noindent (vii) When $i<j^\vee\leq \kappa$, we have
\[
	g_S-\sgn(\sigma_S)g_{s^\perp}=(-1)^\kappa\pi^{-(n-\kappa-2)}
	\left(
	e_{[i,\wh{n+j}]}-(-1)^{n+i+j}e_{[j^\vee,\wh{n+i^\vee}]}
	\right)+O.
\]
\noindent (viii) When $i\leq \kappa<j^\vee<n-\kappa+1$, we have
\[
    g_S-\sgn(\sigma_S)g_{S^\perp}=
	(-1)^\kappa\pi^{-(n-\kappa-2)}e_{[i,\wh{n+j}]}+
	(-1)^{n+\kappa+1+i+j}\pi^{-(n-\kappa-1)}e_{[j^\vee,\wh{n+i^\vee}]}+O.
\]
\noindent (ix) When $i\leq \kappa,j^\vee\geq n-\kappa+1$, we have
\[
	g_S-\sgn(\sigma_S)g_{S^\perp}=(-1)^{\kappa+1}\pi^{-(n-\kappa-1)}
	\left(
	e_{[i,\wh{n+j}]}+(-1)^{n+i+j}e_{[j^\vee,\wh{n+i^\vee}]}
	\right)+O.
\]
(x) When $\kappa<i<j^\vee<n-\kappa+1$, we have
\[
	g_S-\sgn(\sigma_S)g_{S^\perp}=
	(-1)^\kappa\pi^{-(n-\kappa-1)}
	\left( 
	e_{[i,\wh{n+j}]}-(-1)^{n+i+j}e_{[j^\vee,\wh{n+i^\vee}]}
	\right)+O.
\]
\noindent (xi) When $\kappa<i<n-\kappa+1\leq j^\vee$, we have
\[
	g_S-\sgn(\sigma_S)g_{S^\perp}
	=(-1)^{\kappa+1}\pi^{-(n-k)}e_{[i,\wh{n+j}]}
	+(-1)^{n+\kappa+1+i+j}\pi^{-(n-\kappa-1)}e_{[j^\vee,\wh{n+i^\vee}]}.
	+O.
\]
\noindent (xii) When $n-\kappa+1\leq i<j^\vee$, we have
\[
	g_S-\sgn(\sigma_S)g_{S^\perp}=(-1)^{\kappa+1}\pi^{-(n-\kappa)}\left(
	e_{[i,\wh{n+j}]}-(-1)^{n+i+j}e_{[j^\vee,\wh{n+i^\vee}]}
	\right)+O.
\]
Note that only case $(viii)$ and $(xi)$ are different from Proposition \ref{equ_ss-comp:g_S-dual}.

\subsubsection{}
Next, we compare all the coefficients appearing in the above expression.
			
\begin{altitemize}
\item 
For $1\leq i\leq \kappa$ and $n-\kappa+1\leq j\leq n$, we have
\begin{equation*}
c_{ij}=(-1)^\kappa\pi^{-(n-\kappa-2)}a_{ij},
\quad\text{and}\quad
c_{j^\vee i^\vee}=(-1)^{\kappa+n+1+j^\vee+i^\vee}\pi^{-(n-\kappa-2)}a_{j^\vee i^\vee}.
\end{equation*}
Hence
\begin{equation*}
x_{ij}=(-1)^{j-1}(-1)^\kappa\pi^{-(n-\kappa-2)}a_{ij},
\quad\text{and}\quad
x_{j^\vee i^\vee}=(-1)^{i^\vee-1}(-1)^{\kappa+n+1+i^\vee+j^\vee}\pi^{-(n-\kappa-2)}a_{ij}.
\end{equation*}
Therefore, we obtain $x_{ij}=x_{j^\vee i^\vee}$, which proves that $B=B^{\ad}$.
Similarly we have $C=C^{\ad}$.
				
\item 
For $1\leq i\leq \kappa$ and $\kappa+1\leq j\leq n-\kappa$, we have
\begin{equation*}
c_{ij}=(-1)^\kappa\pi^{-(n-\kappa-2)}a_{ij},
\quad\text{and}\quad
c_{j^\vee i^\vee}=(-1)^{n+\kappa+1+i+j}\pi^{-(n-\kappa-1)}a_{ij}.
\end{equation*}
Hence
\begin{equation*}
x_{ij}=(-1)^{j-1}(-1)^\kappa\pi^{-(n-\kappa-2)}a_{ij},
\quad\text{and}\quad
x_{j^\vee i^\vee}=(-1)^{i^\vee-1}(-1)^{n+\kappa+1+i+j}\pi^{-(n-\kappa-1)}a_{ij}.
\end{equation*}
Therefore, we obtain $x_{ij}=\pi x_{j^\vee i^\vee}$, which proves that $F=\pi L^{\ad}$.
				
\item 
For $\kappa+1\leq i\leq n-\kappa$ and  $1\leq j\leq \kappa$, we have
\begin{equation*}
c_{ij}=(-1)^{\kappa+1}\pi^{-(n-\kappa)}a_{ij},
\quad\text{and}\quad
c_{j^\vee i^\vee}=(-1)^{i^\vee -1}(-1)^{n+\kappa+1+i+j}\pi^{-(n-\kappa-1)}.
\end{equation*}
Hence
\begin{equation*}
x_{ij}=(-1)^{j-1}(-1)^{\kappa+1}\pi^{-(n-\kappa)}a_{ij},
\quad\text{and}\quad
x_{j^\vee i^\vee}=(-1)^{i^\vee-1}(-1)^{n+\kappa+1+i+j}\pi^{-(n-\kappa-1)}.
\end{equation*}
Therefore, we obtain $\pi x_{ij}=-x_{j^\vee i^\vee}$, which proves that $\pi M=-E^{\ad}$.
				
\item 
For $\kappa+1\leq i,j\leq n-\kappa$.
\begin{altenumerate}
\item When $i\neq j$ and $i+j<n+1$, we have
\begin{equation*}
c_{ij}=(-1)^\kappa\pi^{-(n-\kappa-1)}a_{ij},
\quad\text{and}\quad
c_{j^\vee i^\vee}=(-1)^{n+1+\kappa+i+j}\pi^{-(n-\kappa-1)}a_{j^\vee i^\vee}
\end{equation*}
Hence
\begin{equation*}
x_{ij}=(-1)^{j-1}(-1)^\kappa\pi^{-(n-k-1)}a_{ij},
\quad
x_{j^\vee i^\vee}=(-1)^{i^\vee-1}(-1)^{n+1+\kappa+i+j}\pi^{-(n-\kappa-1)}a_{ij}.
\end{equation*}
Therefore $x_{ij}=x_{j^\vee i^\vee}$.
\item When $i=j$.
For $i\neq m+1$, we have
\begin{equation*}
c_{ii}=(-1)^i\pi^{-(n-\kappa-1)}\sum_{\substack{\kappa\leq i\leq n-\kappa\\\sigma\neq i }} (-1)^{\kappa+\sigma+1}a_{\sigma\sigma},
\quad
c_{i^\vee i^\vee}=(-1)^{i^\vee}\pi^{-(n-\kappa-1)}\sum_{\substack{\kappa\leq i\leq n-\kappa\\\sigma\neq i }} (-1)^{\kappa+\sigma+1}a_{\sigma\sigma}.
\end{equation*}
					
For $i\neq m+1$, we have
\begin{equation*}
c_{ii}=(-1)^i\pi^{-(n-\kappa-1)}\sum_{\substack{\kappa\leq i\leq n-\kappa\\\sigma\neq i }} (-1)^{\kappa+\sigma+1}a_{\sigma\sigma},
\quad
c_{i^\vee i^\vee}=(-1)^{i^\vee}\pi^{-(n-\kappa-1)}\sum_{\substack{\kappa\leq i\leq n-\kappa\\\sigma\neq i }} (-1)^{\kappa+\sigma+1}a_{\sigma\sigma}.
\end{equation*}
					
For $i=m+1$, we have
\begin{equation*}
c_{m+1,m+1}=(-1)^\kappa\pi^{-(n-\kappa-1)}a_{m+1,m+1}+(-1)^{m+1}\pi^{-(n-\kappa-1)}\sum_{\sigma=\kappa+1}^M (-1)^{\kappa+\sigma+1}a_{\sigma\sigma}.
\end{equation*}
Therefore, when $i\neq m+1$, we have
\begin{equation*}
x_{ii}=(-1)^{\kappa}\pi^{-(n-\kappa-1)}\sum_{\sigma\neq i}(-1)^\sigma a_{\sigma\sigma},
\quad
x_{i^\vee i^\vee}=(-1)^{\kappa}\pi^{-(n-\kappa-1)}\sum_{\sigma\neq i}(-1)^\sigma a_{\sigma\sigma}.
\end{equation*}
\begin{equation*}
x_{m+1,m+1}=(-1)^{\kappa}\pi^{-(n-\kappa-1)}\sum_{\sigma\neq m+1}(-1)^\sigma a_{\sigma\sigma}+(-1)^{m+\kappa}\pi^{-(n-\kappa-1)}a_{m+1,m+1}.
\end{equation*}
\end{altenumerate}
Hence, we deduce that $X_4=X_4^{\ad}$.
				
\item
Next since $c_{\{n+1,\cdots,2n\}}=1$, we get
\begin{equation*}
\sum_{\sigma=\kappa+1}^M a_\sigma (-1)^{\kappa+\sigma+1} \pi^{-(n-\kappa)}=1.
\end{equation*}
Hence, we have
\begin{equation*}
\sum_{\sigma=\kappa+1}^M (-1)^\sigma a_\sigma=(-1)^{\kappa+1}\pi^{n-\kappa}.
\end{equation*}

\begin{altenumerate}
\item When $n=2m$,
\begin{align*}
\tr(X_4)={}	&	2(-1)^{\kappa}\pi^{-(n-\kappa-1)}\sum_{i=\kappa+1}^m \Bigl(\sum_{\sigma\neq i}(-1)^\sigma a_\sigma\Bigr),\\
={}&(-1)^\kappa\pi^{-(n-\kappa-1)}2\Bigl((m-\kappa-1)\sum_{\sigma=\kappa+1}^M (-1)^\sigma a_\sigma\Bigr),\\
={}	&	-(n-2\kappa-2)\pi.
\end{align*}
\item When $n=2m+1$,
\begin{align*}
\tr(X_4)={}	&	x_{m+1,m+1}+2\sum_{i=\kappa+1}^m x_{ii},\\
={}	&	(-1)^{m+\kappa}\pi^{-(n-\kappa-1)}a_{m+1}+(-1)^{\kappa}\pi^{-(n-\kappa-1)}\Bigl(\sum_{\sigma\neq m+1}(-1)^\sigma a_\sigma\Bigr)\\
&\hspace{6cm}+2(-1)^\kappa\pi^{-(n-\kappa-1)}\Bigl(\sum_{i=\kappa+1}^m\sum_{\sigma\neq i}(-1)^\sigma a_\sigma\Bigr),\\
={}	&	(-1)^\kappa\pi^{-(n-\kappa-1)}(2m-1-2\kappa)(-1)^{\kappa+1}\pi^{n-\kappa}\\
={}	&	-(n-2\kappa-2)\pi.
\end{align*}
\end{altenumerate}
This proves that $\tr(X_4)=-(n-2\kappa-2)\pi.$
				
\item Finally, when $1\leq i,j\leq \kappa$.
\begin{altenumerate}
\item When $i\neq j$, we have 
\begin{equation*}
c_{ij}=(-1)^{\kappa+1}\pi^{-(n-\kappa-1)}a_{ij},
\quad\text{and}\quad
c_{j^\vee i^\vee}=(-1)^{n+1+\kappa+i^\vee+j^\vee}\pi^{-(n-\kappa-1)}a_{ij}
\end{equation*}
Hence
\begin{equation*}
x_{ij}=(-1)^{j-1}(-1)^{\kappa+1}\pi^{-(n-\kappa-1)}a_{ij},
\quad\text{and}\quad
x_{j^\vee i^\vee}=(-1)^{i^\vee-1}(-1)^{n+1+\kappa+i^\vee+j^\vee}\pi^{-(n-\kappa-1)}a_{ij}.
\end{equation*}
Therefore, we obtain $x_{ij}=-x_{j^\vee i^\vee}$.
\item When $i=j$, we have
\begin{align*}
c_{ii}={}&(-1)^{\kappa+1}\pi^{-(n-\kappa-1)}a_{ii}+(-1)^i\sum_{\sigma=\kappa+1}^M(-1)^{\kappa+\sigma+1}\pi^{-(n-\kappa-1)},\\
c_{i^\vee i^\vee}={}&(-1)^{n+1+\kappa}\pi^{-(n-\kappa-1)}a_{i^\vee i^\vee}+(-1)^{i^\vee}\sum_{\sigma=\kappa+1}^M(-1)^{\kappa+\sigma+1}\pi^{-(n-\kappa-1)}.
\end{align*}
Hence
\begin{align*}
x_{ii}={}&(-1)^{\kappa+i}\pi^{-(n-\kappa-1)}a_{ii}-\sum_{\sigma=\kappa+1}^M(-1)^{\kappa+\sigma+1}\pi^{-(n-\kappa-1)}a_{\sigma,\sigma},\\
x_{i^\vee i^\vee}={}&(-1)^{n+\kappa+i^\vee}\pi^{-(n-\kappa-1)}a_{ii}-\sum_{\sigma=\kappa+1}^M(-1)^{\kappa+\sigma+1}\pi^{-(n-\kappa-1)}a_{\sigma,\sigma}.
\end{align*}

Recall that $\sum_{\sigma=\kappa+1}^M (-1)^\sigma a_\sigma=(-1)^{\kappa+1}\pi^{n-\kappa}$. Hence,
$x_{ii}+x_{i^\vee i^\vee}=-2\pi$, which implies that $D+A^{\ad}=-2\pi I$.
\end{altenumerate}
\end{altitemize}



\begin{thebibliography}{AB}
\bibitem[Ar09]{Arzdorf2009}
K. Arzdorf - On Local Models with Special Parahoric Level Structure,
\emph{Michigan Math. J.} \textbf{58} (2009) p. 683-710
			
\bibitem[BV98]{BU1998}
W. Bruns \& U. Vetter - \emph{Determinantal Rings},
Lecture notes in mathematics Subseries: Instituto de Matemática Pura e Aplicada, Rio de Janeiro; Springer: Berlin Heidelberg, 1988.

\bibitem[BCRV22]{BCRV2022}
W. Bruns \&  A. Conca \& C. Raicu \& Matteo Varbaro - \emph{Determinants, Gr\"obner Bases and Cohomology}, Springer Monographs in Mathematics. Cham: Springer International Publishing, 2022. 
			
\bibitem[Co94-1]{Conca1994-1}
A. Conca - Divisor Class Group and Canonical Class of Determinantal Rings Defined by Ideals of Minors of a Symmetric Matrix,
\emph{Arch. Math 1994}, (3)\emph{63} (1994) p. 216–224. 
			
\bibitem[Co94-2]{Conca1994-2}
A. Conca - Symmetric Ladders,
\emph{Nagoya Mathematical Journal} \emph{136} (1994) p. 35–56.
			
\bibitem[G\"o01]{Gortz2001}
U. G\"ortz - On the Flatness of Models of Certain Shimura Varieties of PEL-Type, \emph{Math Ann} \emph{321} (2001) p. 689–727. 
			
\bibitem[GW20]{GW2020}
U. G\"ortz \& T. Wedhorn -\emph{Algebraic Geometry I: Schemes: With Examples and Exercises},
Springer Studium Mathematik - Master; Springer Fachmedien Wiesbaden: Wiesbaden, 2020. 
			
\bibitem[HH17]{HH2017}
T. J. Haines \& X. He - Vertexwise Criteria for Admissibility of Alcoves,
\emph{American Journal of Mathematics} \emph{139} (2017) p. 769–784.

\bibitem[HR22]{HR22}
T. Haines \& T. Richarz, - Normality and Cohen–Macaulayness of Parahoric Local Models,
\emph{J. Eur. Math. Soc.} (2) \textbf{25} (2022) p. 703–729.

\bibitem[HLS24]{HLS-regular}
Q. He \& Y. Luo \& Y Shi - Regular models of ramified unitary Shimura varieties at maximal parahoric level,
\emph{arXiv preprint} \url{https://arxiv.org/abs/2410.04500} (2024)

\bibitem[HLS25]{HLS-basic}
Q. He \& Y. Luo \& Y. Shi - The basic locus of ramified unitary Rapoport-Zink space at maximal vertex level,
\emph{arXiv preprint} \url{https://arxiv.org/abs/2502.06218} (2025)

   
\bibitem[HPR20]{HPR}
X. He \& G. Pappas \& M. Rapoport - Good and Semi-Stable Reductions of Shimura Varieties, 
\emph{Journal de l’École polytechnique - Mathématiques} \textbf{7} (2020) p. 497–571.



\bibitem[Ki18]{Kirch2018}
D. Kirch  - Construction of a {R}apoport-{Z}ink space for {${\rm GU}(1,1)$} in the ramified 2-adic case,
\emph{Pacific J. Math.} (2)\textbf{293} (2018) p. 341-389.


\bibitem[Kr03]{Kramer}
N. Krämer - Local Models for Ramified Unitary Groups, 
\emph{Abh. Math. Semin. Univ. Hambg.} \textbf{73} (2003) p. 67–80. 

\bibitem[KRZ24]{KRZ2024}
S. Kudla \& M. Rapoport \&  T. Zink - On the {$p$}-adic uniformization of unitary {S}himura curves,
\emph{M\'em. Soc. Math. Fr. (N.S.)} \textbf{183} (2024).

\bibitem[Le16]{Levin16}
B. Levin - Local models for {W}eil-restricted groups,
\emph{Compos. Math.} (12)\textbf{152} (2016) p. 2563-2601. 

\bibitem[LZ21]{LZ2021}
C. Li \& W. Zhang - Kudla–Rapoport Cycles and Derivatives of Local Densities, 
\emph{J. Amer. Math. Soc.} (3) \textbf{35} (2021) p. 705–797. 	

\bibitem[LMZ25]{LMZ25}
Y. Luo \& A. Mihatsch \& Z. Zhang - On unitary {S}himura varieties at ramified primes,
\emph{arXiv preprint} \url{https://arxiv.org/abs/2504.17484} (2025)

\bibitem[LRZ25]{LRZ25}
Y. Luo \& M. Rapoport \& W. Zhang - More regular formal moduli spaces and arithmetic transfer conjectures: the ramified quadratic case,
\emph{In prepration}


	

\bibitem[Mi22]{Mi2022}
A. Mihatsch - Relative Unitary RZ-Spaces and the Arithmetic Fundamental Lemma. \emph{Journal of the Institute of Mathematics of Jussieu} (1) \textbf{21} (2022) p. 241–301. 

   
\bibitem[Pa00]{Pappas2000}
G. Pappas - On the Arithmetic Moduli Schemes of PEL Shimura Varieties,
\emph{J. Algebraic Geom.} (3) \textbf{9} (2000) p. 577–605.
			
\bibitem[PR03]{PR2003}
G. Pappas \& M. Rapoport - Local Models in the Ramified Case. I: The EL-Case,
\emph{J. Algebraic Geom.} (1) \textbf{12} (2003) p. 107–145.

\bibitem[PR05]{PR2005}
G. Pappas \& M. Rapoport - Local models in the ramified case. {II}. {S}plitting models,
\emph{Duke Math. J.} (2)\textbf{127} (2005) p. 193-250.

\bibitem[PR08]{PR2008}
G. Pappas \& M. Rapoport - Twisted Loop Groups and Their Affine Flag Varieties, \emph{Advances in Mathematics} (1) \textbf{219} (2008) p. 118–198. 
			
\bibitem[PR09]{PR2009}
G. Pappas \& M. Rapoport - Local Models in the Ramified Case. III Unitary Groups, \emph{Journal of the Institute of Mathematics of Jussieu} (3) \textbf{8} (2009) p. 507–564. 

\bibitem[PZa22]{PZa2022}
G. Pappas \& I. Zachos - Regular integral models for Shimura varieties of orthogonal type. \emph{Compositio Math} \textbf{158} (2022) p. 831–867.

   
\bibitem[PZ13]{PZ2013}
G. Pappas \& X. Zhu - Local Models of Shimura Varieties and a Conjecture of Kottwitz, 
\emph{Invent. Math.} (1) \textbf{194} (2013) p. 147–254. 

\bibitem[RSZ17]{RSZ2017}
M. Rapoport \& B. Smithling \& W. Zhang - On the Arithmetic Transfer Conjecture for Exotic Smooth Formal Moduli Spaces,
\emph{Duke Math. J.} (12)\textbf{166} (2017) p. 2183-2336. 
			
\bibitem[RSZ18]{RSZ2018}
M. Rapoport \& B. Smithling \& W. Zhang - Regular Formal Moduli Spaces and Arithmetic Transfer Conjectures. \emph{Math. Ann.} (3–4)\textbf{370} (2018) p. 1079–1175.

\bibitem[RSZ20]{RSZ2020}
M. Rapoport \& B. Smithling \& W. Zhang - Arithmetic diagonal cycles on unitary {S}himura varieties,
\emph{Compos. Math.} (9)\textbf{156} (2020) p. 1745-1824.

\bibitem[RSZ21]{RSZ-2021}
M. Rapoport \& B. Smithling \& W. Zhang - On {S}himura varieties for unitary groups, 
\emph{Pure Appl. Math. Q.} (2) \textbf{17} (2021) p. 773-837. 



\bibitem[RZ96]{RZ}
M. Rapoport \& T. Zink - \emph{Period Spaces for $p$-divisible Groups},
Annals of Mathematics Studies, vol. 141, 
Princeton University Press, Princeton, NJ, 1996. 

\bibitem[RZ17]{RZ17}
M. Rapoport \& T. Zink - On the {D}rinfeld moduli problem of {$p$}-divisible groups,
\emph{Camb. J. Math.} (2)\textbf{5} (2017) p. 229-279.

\bibitem[Ri13]{Richarz13}
T. Richarz - Schubert varieties in twisted affine flag varieties and local models, 
\emph{J. Algebra} \textbf{375} (2013) p. 121-147. 
			
\bibitem[Sm11]{Smithling2011}
B. Smithling - Topological Flatness of Local Models for Ramified Unitary Groups. I. The Odd Dimensional Case. 
\emph{Advances in Mathematics} (4) \textbf{226} (2011) p. 3160–3190. 			
			
\bibitem[Sm14]{Smithling2014}
B. Smithling - Topological Flatness of Local Models for Ramified Unitary Groups. II. The Even Dimensional Case. 
\emph{Journal of the Institute of Mathematics of Jussieu} (2) \textbf{13} (2014) p. 303–393. 
			
\bibitem[Sm15]{Smithling2015}
B. Smithling - On the Moduli Description of Local Models for Ramified Unitary Groups, 
\emph{International Mathematics Research Notices} (24) \textbf{2015} (2015) p. 13493–13532. 
			
\bibitem[Stacks]{Stacks}
The {Stacks Project Authors}. “\textit{Stacks Project}.” \url{https://stacks.math.columbia.edu}.

\bibitem[Ya24]{JYang-2024}
J. Yang - Wildly ramified unitary local models for special parahorics. The odd dimensional case,
\emph{arXiv preprint} \url{https://arxiv.org/abs/2406.16774} (2024)

\bibitem[Yu19]{Yu2019}
S. Yu - On Moduli Description of Local Models for Ramified Unitary Groups and Resolution of Singularity, 
Johns Hopkins University, Baltimore, Maryland (2019).
			
\bibitem[Za23]{Zachos}
I. Zachos - On Orthogonal Local Models of Hodge Type. 
\emph{International Mathematics Research Notices} (13) \textbf{2023} (2023) p. 10799–10836. 
			
\bibitem[Zhu14]{Zhu2014}
X. Zhu - On the Coherence Conjecture of Pappas and Rapoport.
\emph{Ann. Math.} (1)\textbf{180} (2014) p. 1–85. 

\bibitem[ZZ24]{Zachos-Zhao-2024}
I. Zachos \& Z. Zhao - Semi-stable and splitting models for unitary Shimura varieties over ramified places. II,
\emph{arXiv preprint} \url{https://arxiv.org/abs/2405.06163} (2024)



























\end{thebibliography}

\end{document}